%% file: main.tex
\documentclass[a4paper,11pt,oneside]{article}
\usepackage[a4paper]{geometry}
\usepackage[T1]{fontenc}
\usepackage[utf8]{inputenc}
\usepackage[english]{babel}
\usepackage{microtype}
\usepackage{quoting}
\usepackage{amsthm, amsmath, amssymb, amscd, mathrsfs, mathtools} %mathematical packages
\usepackage{dsfont}
\usepackage{booktabs, caption, graphicx} %picture and tables
\usepackage{enumitem}
\usepackage{esint} %spaziatura tra due o più simboli di integrale
\usepackage{framed}
\usepackage{braket}
\usepackage{tikz-cd}
\usepackage{centernot}
\usepackage{fancyhdr}
\usepackage{emptypage}
\usepackage{csquotes}
\usepackage{tensor}
\usepackage{setspace}
\usepackage{tocbasic}
\usepackage[backend=biber, sorting=nyt, giveninits=true, style=alphabetic]{biblatex}
\usepackage{csquotes}
\usepackage{settings} %general settings of the document

\definecolor{mygreen}{RGB}{13, 110, 53}
\definecolor{Green}{RGB}{0, 128, 0}
\definecolor{navyblue}{RGB}{0, 0, 128}
\definecolor{MyMulberry}{RGB}{197, 75, 140}

\usepackage[colorlinks,linkcolor=blue, urlcolor = navyblue, citecolor = Green]{hyperref}

\author{Mattia Freguglia, Andrea Malchiodi and Francesco Malizia \thanks{Scuola Normale Superiore, Piazza dei Cavalieri 7, 56126 Pisa. e-mails: mattia.freguglia@sns.it, andrea.malchiodi@sns.it, francesco.malizia@sns.it}}

\title{Min-max theory and Yamabe metrics on conical four-manifolds}
\date{}

\begin{document}

\maketitle

{\footnotesize
\begin{abstract}

\noindent We prove existence of Yamabe metrics on 
 four-manifolds possessing finitely-many conical points 
with $\Z_2$-group, using for the first time a min-max scheme 
in the singular setting. In our variational argument we need to deform continuously regular bubbles into singular ones, while 
keeping the Yamabe energy sufficiently low.
% \colorbox{yellow}{(2)} This requires us to employ recent positive mass theorems in the conical setting, as well as to study the divergence rate of the mass, associated to the conformal blow-up of our singular manifold, when the blow-up point approaches the singular set.
% \colorbox{yellow}{(3)} 
For doing this, we exploit recent positive mass theorems in the conical setting and study how the mass of the conformal blow-up diverges as the blow-up point approaches the singular set.

% performing conformal blow-ups, the mass 

% diverges at an inverse-quadratically from the distance to the singular set 

% when the 

% . Along the deformation into conical bubbles 
% we 
% need to understand the behavior of the Green's function of the conformal Laplacian 
% when the pole approaches a conical point, providing 
% an asymptotics for the mass depending on the distance from it. 

\vspace{3ex}

\noindent{\it Key Words:} Singular Yamabe problem, Conical Metrics, 
Conformal Geometry, Min-max methods. 

\vspace{2ex}

\noindent{\bf MSC 2020:} 53C18, 53C21, 58J60, 35J20.

\end{abstract}}

%\begin{spacing}{0.75}
%    {\tableofcontents}
%\end{spacing}
\tableofcontents

\input{introduction}

\input{preliminaries}

\input{greenfunction}

\input{variational_argument}

\input{doublebubble}

\input{bubblegreen}

\input{Interpolation}

\printbibliography[heading=bibintoc]

\end{document}

%% file: introduction.tex
\section{Introduction}

In this paper we consider the Yamabe problem, posed in  \cite{Yamabe-60} and consisting in conformally deforming  the metric of a  manifold $(M^n,g)$, $n \geq 3$,  to obtain constant scalar curvature. 
If $\tilde{g} = w^{\frac{4}{n-2}} g$ is a conformal metric, then the scalar curvature $R_{\tilde{g}}$ of $\tilde{g}$ 
is determined by the formula 
\begin{equation}\label{eq:transf}
     L_g w = R_{\tilde{g}} w^{\frac{n+2}{n-2}},  
\end{equation}
where $L_g$ denotes the {\em conformal Laplacian} 
\[
  L_g w = - a \Delta_g w + R_g w, \qquad a = \frac{4(n-1)}{n-2}. 
\]
By equation \eqref{eq:transf}, the Yamabe problem 
is equivalent to solving 
\begin{equation}\tag{$\mathcal{Y}$} \label{eq:Y}
	L_g u = \bar{R} \, u^{\frac{n+2}{n-2}}, \qquad \bar{R} \in \R.
\end{equation}
On closed manifolds the equation has a  variational formulation: working in the  functional space  
$H^1(M) := \big\{ u \colon M \to \R \, | \, u, \nabla u \in L^2 \big\}$, 
solutions are extremals of the {\em Yamabe energy} 
	\begin{equation*}
		Q_g (u) := \frac{\int_M \big( a \abs{\nabla_g u}^2 + R_g u^2 \big) \, d\mu_g }{\big( \int_M \abs{u}^{\frac{2n}{n-2}} \, d\mu_g \big)^{\frac{n-2}{n}}}, 
		\qquad \quad u \in H^1(M). 
	\end{equation*}
We call {\em Yamabe metrics} conformal metrics associated to any critical 
point of $Q_g$. A first natural attempt to find such critical points is 
to try minimizing this quantity, considering the 
{\em Yamabe constant}
\[
   \mathcal{Y}(M, g) := \inf_{\substack{{u \in C^{\infty}(M)} \\ u \neq 0}} Q_g(u). 
\]
Since the conformal Laplacian satisfies 
\begin{equation*}
    L_{\tilde{g}} \phi  = w^{- \frac{n+2}{n-2}} L_g (w \, \phi), \qquad \phi \in C^\infty(M), 
\end{equation*} 
the Yamabe constant only depends on the conformal class $[g]$ of $g$, and 
will be denoted by $\mathcal{Y}(M,[g])$. 

However, due to the noncompactness of the embedding $H^1(M) \hookrightarrow L^{\frac{2n}{n-2}}(M)$, minimizing sequences might develop a bubbling
behavior. Despite this difficulty, the Yamabe problem was completely solved:
first, in \cite{Trudinger-68} it was shown the existence of $\varepsilon_n > 0$ such that
$\mathcal{Y}(M,[g])$ is attained whenever $\mathcal{Y}(M,[g]) < \eps_n$. This
was sharpened in \cite{Aubin-76-2}, proving that $\mathcal{Y}(M,[g]) $ is achieved provided that  
\begin{equation}\label{eq:Y<}
  \mathcal{Y}(M,[g]) < \mathcal{Y}(\Sp^n,[g_{\Sp^n}]), 	
\end{equation}
with $g_{\Sp^n}$ the round metric on the sphere; it was shown that \eqref{eq:Y<} holds when $n \geq 6$ and $g$ is not locally conformally flat, expanding $Q_g$ on functions of the type (in normal 
coordinates at a point where the Weyl tensor does not vanish) 
${U}_\eps(x) 
\cong \eps^{-\frac{n-2}{2}} U(x/\eps)$, for $\eps$ small, 
where $U$ (see \eqref{eq:std-norm-bub})  
is the extremal of the Sobolev inequality in $\R^n$, see also \cite{Tal}. 
When $n \leq 5$ or when $g$ is 
locally conformally flat, and $M$ is not the sphere, strict inequality was proved in~\cite{Schoen-84} 
using the  {Positive Mass Theorem} from~\cite{Schoen-Yau-79},~\cite{Schoen-Yau-81} and~\cite{Schoen-Yau-88}.

\

In this paper we study the Yamabe problem on singular manifolds: in particular on four-manifolds with 
finitely-many $\Z_2$-conical points. 
Singular structures arise naturally when considering Gromov-Hausdorff limits of 
smooth manifolds, such as non-collapsing Einstein or 
\emph{critical} metrics (\cite{And}, \cite{bando-kasue-nakajima-1989-Inventiones}, \cite{Ti-Vi-1}, \cite{Ti-Vi-2}), 
in which orbifold points can form. Isolated singularities might also appear in the extremization of the Yamabe energy with respect to the conformal class $[{g}]$ 
(see \cite{Ak-94} and \cite{Ak-96}), while other types of stratified 
singularities can be introduced to analyze  K\"ahler-Einstein metrics, see 
e.g.~\cite{JMR} and references therein.

Before stating our main result, we first recall some facts concerning the
Yamabe problem in the singular setting, which presents new issues 
compared to the classical case. For example, it was proved in \cite{Ak-Mon} that on $\Sp^n$ with an equatorial conical wedge of codimension $2$ and of angle $\alpha \geq 4 \pi$, 
the Yamabe constant is not attained. In 
\cite{Vi} there are even examples of four-manifolds with orbifold points (indeed,  
conformal compactifications of hyperk\"ahler ALE manifolds) for which the Yamabe equation is not  solvable.

All the above examples fall into the category of \emph{stratified spaces} considered in
\cite[Section~2.1]{ACM}, and for such spaces  it is possible to give a criterion, that we now describe, for the attainment of the Yamabe constant. As for the regular case,  one defines 
\begin{equation*}
    \mathcal{Y}(M,[g])
    :=
    \inf_{\substack{{u \in W^{1,2}(M)} \\ u \neq 0}} Q_{g}(u)
    =
    \inf_{\substack{{u \in W^{1,2}(M)} \\ u \neq 0}} \frac{\int_{\Omega} \big( a \abs{\nabla_{g} u}^2 + R_{g} u^2 \big) \, d\mu_{g} }{\big( \int_{\Omega} \abs{u}^{\frac{2n}{n-2}} \, d\mu_{g} \big)^{\frac{n-2}{n}}}, 
\end{equation*}
where $\Omega$ is the \emph{regular part} of $M$ 
and $W^{1,2}(M)$ denotes the space obtained as the closure of Lipschitz functions with respect to the $W^{1,2}$-norm. For $P \in M$, we recall the notion of {\em local Yamabe constant} \begin{equation*}
  \mathcal{Y}_P := \lim_{r \to 0^+} \inf \big\{ Q_{g}(u) : u \in W^{1,2}_0(B_r(P)) \big\}, 
\end{equation*}
which is related to the minimal blow-up Yamabe energy at the point $P$. 
For example,  if $P$ is a conical point with link $(Y, h_0)$, see~\eqref{eq:diffeo} below, then the local Yamabe constant $\mathcal{Y}_P$ coincides with that of the scaling-invariant  metric
$ds^2 + s^2 h_0$ on $(0,+\infty) \times Y$,  
% is upper bounded by $\mathcal{Y}(\Sp^n,[g_{\Sp^n}])$
which is sometimes explicitly known. 
When $Y = \Sp^{n-1} / \Gamma$, then $\mathcal{Y}_P = k^{-\frac{2}{n}} \mathcal{Y}(\Sp^n,[g_{\Sp^n}])$, with $k$  the cardinality of the  isometry group
$\Gamma$, see \cite{Ak-orb}, and if $(Y,h_0)$ is an Einstein manifold with 
$\mathrm{Ric}(h_0) = (n-2) h_0$, one has that $\mathcal{Y}_P = ( \mathrm{Vol}_{h_0} (Y) / \mathrm{Vol}_{g_{\Sp^{n-1}}}(\Sp^{n-1}))^{\frac{2}{n}} \mathcal{Y}(\Sp^n,[g_{\Sp^n}])$, see 
Corollary~1.3 in~\cite{Petean}. The local Yamabe constant is also known for 
wedge-type singularities, see \cite{Mond-17}.

In analogy 
with the results in \cite{Aubin-76-2}, in \cite{Ak-Bo} and \cite{ACM} it was shown that  the  Yamabe equation is solvable under the condition 
\begin{equation} \label{eq:sing-quot}
  \mathcal{Y}(M,[g]) < \Y_S := \min_{Q \in M} \mathcal{Y}_Q. 
\end{equation}
This inequality was verified in \cite{Vi} in some special cases
for connected sums of $\mathbb{CP}^2$'s, and in \cite{freguglia-malchiodi-2024-pre} for
conical manifolds with strictly stable Einstein links in
dimension $n \geq 4$ under the conditions~\eqref{eq:diffeo} and~\eqref{s:xi-conf} described below. 

\bigskip

The purpose of this paper is to produce, for the first time 
to our knowledge, a min-max scheme for the Yamabe problem in the 
singular setting, and to apply it to find new 
existence results on singular four-manifolds possessing 
finitely-many conical points with $\Z_2$-group. As it will 
be clarified below, such conditions imply that the limiting Yamabe energy of 
two singular bubbles  coincides with that of one regular bubble, which will 
be useful to derive compactness estimates. 

\smallskip

Let us first precisely describe the objects we deal with: we consider compact metric spaces $(M,d)$ such that there exist a finite number of points $\{ P_1, \dots , P_l \}$ and a Riemannian metric $g$ on $M \setminus \{ P_1, \dots , P_l \}$ that induces the same distance $d$.

In addition, we ask that for every point $P_i$ there exists an open neighborhood $\mathcal{U}_i$ of $P_i$ and a diffeomorphism $\sigma_i$ such that
\begin{equation}  \label{eq:diffeo} \tag{$H_{P}$}
    \sigma_i \colon \mathcal{U}_i \setminus \{ P_i \} \to (0,1) \times Y_i, \quad 
    \hbox{and} \quad (\sigma_i)_{*} g = ds^2 + s^2 h_i(s),
\end{equation}
where $h_i(s)$ is a family of metrics on the smooth closed manifold $Y_i$, regular \underline{up to $s=0$}. Moreover, $(Y_i, h_{i}(0))$ is called the \emph{link} over the \emph{conical} point $P_i$, and we will use the notation $h_{i,0} = h_i(0)$.

\smallskip

Therefore, the metric ball of radius $s$ around 
a conical point $P$ is of the type  
\[
    B_s(P) = \Big( [0,s) \times Y \Big) \Big|_\sim,
\]
namely a topological cylinder collapsed on one side. 
Our main result is as follows: 
 
\begin{theorem}\label{t:main}
 	 Let $(M, g)$ be a closed four-manifold  
 	   with finitely-many conical points $\{P_1, \dots, P_l\}$ 
 	  such that \eqref{eq:diffeo} holds with $Y_i = \Sp^3/\Z_2$ for all $i$.  
 	 Then, if $l \geq 2$, $(M,g)$ admits a Yamabe metric. 
 \end{theorem}

Before discussing the proof, some comments are in order.

\begin{remark} (\textcolor{blue}{a}\label{rmk:a1}) 
In \cite{LeB-88} some examples of ALE manifolds with {\em negative mass} 
were given, which  become conical after conformal compactification. 
This may lead to the failure of \eqref{eq:sing-quot}, 
as it happens for  the non-existence example in \cite{Vi}. Therefore, our  
condition $l \geq 2$ is necessary for solving \eqref{eq:Y}. 

(\textcolor{blue}{b}) In \cite{freguglia-malchiodi-2024-pre} 
it was required that 
\begin{equation}\label{s:xi-conf} \tag{$\xi_P$}
	h'(0) \neq \nabla_{h_0}^2 f + f h_0, \quad \text{for all smooth $f \colon Y \to \R$,} 
\end{equation}
  to distinguish at first order 
 the metric $g$ from conformally deformed purely conical ones, see Proposition~\ref{l:Melrose}. 
 Here we make no such assumptions, which in particular allows us to deal  
 with the case $h'(0) = 0$, applying e.g. to {\em orbifold metrics}, namely those  that 
 smoothly locally lift to double covers.

 (\textcolor{blue}{c})  In the case when \eqref{eq:sing-quot} is not verified, 
 our result gives  an answer to Problem 5.6 in \cite{akutagawa-2021-sugaku-survey}. 
Indeed, by the fact that our links are of type $\Sp^3/\Z_2$, our (variational) solutions have globally Lipschitz gradient by the regularity results in~\cite{Bo-Pre} and~\cite{ACM}. 
However, when the local lift of $g$ admits a smooth extension, by the same 
argument of \cite[Theorem 3.1]{Ak-orb}, one can show that Yamabe  
metrics are also (smooth) orbifold metrics. 
 \end{remark}

We next describe the strategy and the main ingredients of our proof. 
As we remarked before, under the assumptions of  Theorem \ref{t:main} 
we have that the local Yamabe constant $\mathcal{Y}_P$ 
coincides with $\Y_4 := \mathcal{Y}(\Sp^4,[g_{\Sp^4}])$  if $P$ is 
regular, and with $\frac{\sqrt{2}}{2} \Y_4$ if 
$P$ is a conical point. If \eqref{eq:sing-quot} holds, 
then we have a minimizing Yamabe metric by the result in \cite{ACM}. 
We can therefore assume from now on that 
\begin{equation} \label{eq:no-att}
    \mathcal{Y}(M,[g]) = \frac{\sqrt{2}}{2} \Y_4 \ \text{and  it is not attained}.
\end{equation}
Consider next two singular (conical) points $P_1, P_2$, and 
two functions of the type (localized via cut-offs)
\[
U_{\eps, 1} = \eps^{-1} U(\eps^{-1} s_1); \qquad U_{\eps, 2} = \eps^{-1} U(\eps^{-1} s_2), 
\]
where $U$ is the extremal of the Sobolev inequality in 
$\R^n$, see \eqref{eq:std-norm-bub}, and where $s_i$ stands for the geodesic distance from $P_i$, $i = 1, 2$. By the results in 
\cite{Petean}, such functions are also extremals 
for the Sobolev inequality in the Ricci-flat dilation-invariant 
cone with link $\Sp^3/\Z_2$. Since we are assuming that $\mathcal{Y}(M,[g])$ is not attained, it must then be 
\[
 Q_{g}(U_{\eps, i}) \searrow \frac{\sqrt{2}}{2} \Y_4 \qquad \hbox{ as } \eps \to 0, \quad i = 1, 2. 
\]
For $\bar{\eps} > 0$ sufficiently small, we can then consider the class of {\em admissible maps} 
\begin{equation*}%\label{eq:path-Gamma}
    \Pi = \Pi_{\bar{\eps}} := \big\{ \gamma \in C([0,1]; W^{1,2}(M)) \; | \; 
    \gamma(0) = U_{\bar{\eps}, 1}, \gamma(1) = U_{\bar{\eps}, 2} \big\}, 
\end{equation*}
and the min-max value 
\begin{equation*}%\label{eq:c-minmax}
 c:= \inf_{\gamma \in \Pi} \max_{t \in [0,1]} Q_{g} 
 (\gamma(t)). 
\end{equation*}
It is possible to show via concentration-compactness arguments that, 
for $\bar{\eps} > 0$ sufficiently small, one has the strict inequality 
$c > \max \{ Q_{g}(U_{\bar{\eps}, 1}), Q_{g}(U_{\bar{\eps}, 2}) \}$, see Lemma \ref{lemma:concentration-compactness} for a precise statement. 

In view of this, standard variational tools imply the existence of a 
Palais-Smale sequence for the Yamabe energy at level $c > \frac{\sqrt{2}}{2} \Y_4$. It turns out that we also 
have the following inequality 
\begin{equation}\label{eq:c<Y4}
c < \Y_4,
\end{equation}
see Proposition \ref{prop:gammabar-yam}. By the fact that the  energy of two 
singular bubbles coincides with that of a regular one, see \eqref{eq:j1j2}, 
all possible blow-up scenarios would be ruled out. This implies the existence of a critical point of the Yamabe 
energy, and therefore a solution of the Yamabe problem. Variants of the  above variational scheme were used in e.g. 
\cite{Coron}, \cite{Cao}, \cite{Bianchi-96} and \cite{Car-Ma}, in the context of 
critical equations in bounded domains, nonlinear field 
equations in $\R^n$, Kazdan-Warner's problem and  singular Liouville equations 
on compact surfaces. 

\bigskip

However, in the present situation it is particularly delicate to prove the upper bound in \eqref{eq:c<Y4}, and we will describe next how we proceed.  The idea is to consider a curve $\hat{\gamma} \colon [0,1] \to M$ joining the points $P_1$ and $P_2$ and 
not passing through any other singular point. We associate to it an admissible curve $\bar{\gamma} \in \Pi$ such that the $L^{\frac{2n}{n-2}}$-norm 
of $\bar{\gamma}(t)$ is concentrated near $\hat{\gamma}(t)$ and such that 
the Yamabe energy of $\bar{\gamma}(t)$ is always below $\Y_4$ for all $t$. 

If $\hat{\gamma}(t)$ is outside a fixed neighborhood of $\{P_1, P_2\}$, 
we consider a test function as in \cite{Schoen-84}, with the profile of a
regular bubble $U_\eps$ centered at $\hat{\gamma}(t)$ and glued to a suitable multiple of the Green's function of $L_g$ with pole at $\hat{\gamma}(t)$. 
Thanks to a recent positive mass theorem for (conformal blow-ups at regular points of) manifolds with conical singularities from \cite{Dai-Sun-Wang-24-preprint}, we can guarantee the desired upper bound. 
The upper bound on $Q_g$  becomes though particularly delicate when $\hat{\gamma}(t)$  approaches one of the singular points, say $P_1$, since we need to 
deform a regular bubble into a singular one. 

First, we prove such a property in the flat cone obtained 
as a quotient of $\R^4$ via the antipodal action. Lifting to 
$\R^4$, we consider a symmetric sum of regular bubbles $U_{\eps,t} + U_{\eps,-t}$, 
where 
\[
U_{\eps,t}(y) = \frac{\epsilon^{-1}}{\Uetden}; 
\qquad U_{\eps,-t}(y) = \frac{\epsilon^{-1}}{\Uetmden},
\]
with ${\bf e}_1$ the first coordinate vector. 
When $t$ runs from zero to infinity, 
quotienting by the antipodal action, we obtain the desidered deformation from a 
conical bubble into a regular one, with Yamabe energy varying from $\frac{\sqrt{2}}{2} \Y_4$ and $\Y_4$,  staying always strictly between these two values. This estimate could be viewed as a {\em non-perturbative} 
version of an asymptotic expansion from \cite{ES-86}, where the
authors exploit the interaction (decreasing the Yamabe energy) of bubbles  highly concentrated at different points  to tackle variationally the Kadzan-Warner problem, see also \cite{Ba-Co}.

Adapting this construction to singular manifolds, we need to smoothly interpolate 
with the family of regular bubbles described above, glued to the Green's 
functions $G_p$ for the conformal Laplacian. In this crucial step, one needs then to understand with sufficient precision the behaviour of 
$G_p$ when $p$ approaches the conical point $P_1$. This is done in Section \ref{sec:greenfunct}, where we show that the mass grows proportionally 
to the squared inverse  distance from $P_1$. Since the mass can be identified 
as the constant term in the expansion of $G_p$ in \textit{conformal 
normal coordinates} at $p$, see \cite{Lee-Parker-87}, in this step 
we also need to analyze the dependence of such coordinates when $p$ 
appraches $P_1$.

\medspace

\begin{remark} In general dimension, or in the presence of conical 
points with links of different type, a min-max scheme as above might 
produce  bubbling at multiple  points, singular or regular. It might still be possible to rule out some non-compactness scenarios, as with the 
blow-up analysis in \cite{Ju-Vi} (showing isolated-simple bubbling behavior), 
but other tools and ideas would be needed.  
It would also be interesting to develop variational arguments for other types of singularities, as for example conical wedges of codimension two. 
\end{remark}

\bigskip

The plan of the paper is the following. In Section \ref{sec:prel} we collect some useful preliminaries on the Euclidean Sobolev quotient and on conical metrics. In Section \ref{sec:greenfunct} we discuss the existence of the Green's function on conical manifolds, and derive via parametrix the asymptotic behavior of the mass when the 
pole approaches a conical point, showing that it diverges proportionally to the inverse of the squared distance from the singularity. In Section \ref{sec:variational-argument} we introduce our min-max scheme, show that it is variationally admissible and construct a min-max path. 
Near regular points, we can prove upper bounds on the Yamabe energy using the positive mass theorem from \cite{Dai-Sun-Wang-24-preprint}. 
Finally, Section \ref{section:expansions} is devoted to proving upper bounds for the Yamabe energy along the min-max path, in the delicate regime when the center of a regular bubble approaches a conical point, and continuously deforms into a singular bubble, exploiting 
the results in Sections \ref{sec:prel}-\ref{sec:greenfunct}. 

\begin{center}
	{\bf Acknowledgments} 
\end{center}

\noindent 
The authors are supported by the PRIN Project 2022AKNSE4 {\em Variational and Analytical aspects of Geometric PDEs}, and 
A.M. by the project {\em Geometric problems with loss of compactness} from Scuola Normale Superiore. The authors are members of GNAMPA, as part of INdAM.

%% file: preliminaries.tex
\section{Preliminary facts}
\label{sec:prel}

In this section we introduce some useful preliminary facts. We first investigate some properties of the Sobolev quotient in $\R^4$, in particular on suitable 
linear combinations of two bubble functions. We then 
consider regularity and lifting properties of conical metrics as described in the introduction.

\subsection{On the Euclidean Sobolev quotient}
\label{sec:lemmas}

Let $\mathcal{S}_n > 0$ be the Sobolev constant in dimension $n$, that is, the largest constant such that
        \begin{equation*} %\label{eq:Sob-con}
            \mathcal{S}_n \norm{u}_{L^{\frac{2n}{n-2}}(\R^n)}^2 \le \norm{\nabla u}_{L^2(\R^n)}^2, \qquad \forall u \in C_c^{\infty}(\R^n).
        \end{equation*}
We define the space $\mathscr{D}^{1,2}(\R^n):=\overline{C_c^{\infty}(\R^n)}^{\norm{\nabla \cdot}_{L^2(\R^n)}}$ and the functional
        \begin{equation}
            \label{eq:flat-quotient}
            Q_{g_{\R^n}}(u):=\frac{\int_{\R^n} \frac{4(n-1)}{n-2} \abs{\nabla u}^2 \, dx }{\big(\int_{\R^n} u^{\frac{2n}{n-2}} \, dx \big)^{\frac{n-2}{n}}}, \qquad \forall u \in \mathscr{D}^{1,2}(\R^n) \ \text{and} \ u \ge 0.
        \end{equation}
Let $U \colon \R^n \to (0,+\infty)$ be the function defined as
        \begin{equation} \label{eq:std-norm-bub}
            U(x)=U(\abs{x}):=c_n  \bigg( \frac{1}{1+\abs{x}^2} \bigg)^{\frac{n-2}{2}}, \qquad c_n:=\bigg( \frac{n(n-2)}{\mathcal{S}_n} \bigg)^{\frac{n-2}{4}}. 
        \end{equation}
    We call it the \emph{normalized bubble}, the reason being that $\norm{U}_{L^{\frac{2n}{n-2}}(\R^n)} = 1$. Moreover, $U$ is a solution of the following equation:
        \begin{equation} \label{eq:Y-eq-Rn}
            - \Delta U = \mathcal{S}_n U^{\frac{n+2}{n-2}} \qquad \hbox{ in } \R^n. 
        \end{equation}
    Starting from $U$, it is possible to construct an $(n+1)$-dimensional family of solutions, that is, for every $\eps > 0$ and every $x_0 \in \R^n$, we define the function
        \begin{equation} \label{eq:bub-family}
            U_{\eps, x_0}(x) := \eps^{-\frac{n-2}{2}} U\bigg(\frac{x-x_0}{\eps}\bigg).
        \end{equation}
    We remark that each element of this family is an absolute minimizer of the functional defined by~\eqref{eq:flat-quotient}.

We define next a family of \emph{double-bubbles} as follows
        \begin{equation}
            \label{eq:double-bubble}
                \widehat{U}_{\eps,x_0}(x) := U_{\eps, x_0}(x) + U_{\eps, x_0}(-x) = U_{\eps, x_0}(x) + U_{\eps, -x_0}(x).
        \end{equation}
By taking quotient via the antipodal map, such functions represent, on a conical manifold, a deformation from a 
regular bubble into a singular one, as $|x_0|$ decreases from $+\infty$ to zero. Our next goal is to estimate from above and below the Euclidean Sobolev quotient on such functions.

\begin{lemma}
    Let $\eps, t > 0$ and let $\nu \in \Sp^3$. Then,
        \begin{equation}
            \label{eq:almost-mon}
                6 \mathcal{S}_4 < Q_{g_{\R^4}}\big(\widehat{U}_{\eps, t \nu} \big) < 6 \sqrt{2} \mathcal{S}_4, \qquad \forall t \in (0,+\infty),
        \end{equation}
        where $\widehat{U}_{\eps, t \nu}$ is the double-bubble defined in \eqref{eq:double-bubble}.
\end{lemma}

\begin{proof}
    Without loss of generality, assume $\nu=\mathbf{e}_1$ and let us denote $\Uethat=\widehat{U}_{\epsilon,t\mathbf{e}_1}$.
    First, we observe that $Q_{g_{\R^4}} (\Uethat) \ge 6 \mathcal{S}_4$ for every $t \ge 0$, with equality if and only if $t=0$. Indeed, the minimum of the quotient is exactly $6 \mathcal{S}_4$ and the minimizers are classified:  they are indeed positive multiples of the functions defined by~\eqref{eq:bub-family}. In particular, if $\Uethat  = k U_{\eta, x_0}$ for a positive constant $k > 0$, then $t=0$ (and $x_0=0$, $\eta=\eps$, $k=2$).

     In ~\cite{Bahri-89}, estimate F3 and equation~(1.6) (or \eqref{eq:gradprod-integral} below), it is proved that
        \begin{equation}
            \label{eq:F3-estimate}
                \int_{\R^4} \nabla \Uet \cdot \nabla \Uetm  \, dx = B \frac{\epsilon^2}{t^2} + o\Big(\frac{\epsilon^2}{t^2}\Big), \qquad \text{as $\frac{\epsilon}{t} \to 0^+$},
        \end{equation}
    where $B > 0$ is some given positive constant. Now, $\Uet$ is a solution of~\eqref{eq:Y-eq-Rn}, hence, combining~\eqref{eq:F3-estimate} with an integration by parts we have that 
        \begin{equation}
            \label{eq:F3-estimate-2}
                \int _{\R^4} \Uet^3 \Uetm  \, dx = \int _{\R^4} \Uet \Uetm^3  \, dx = B \mathcal{S}_4^{-1} \frac{\epsilon^2}{t^2} + o\Big(\frac{\epsilon^2}{t^2}\Big), \qquad \text{as $\frac{\epsilon}{t}\to 0^+$}.
        \end{equation}
    In addition, we also have the following asymptotic result,
        \begin{equation}
            \label{eq:F3-estimate-3}
                \int _{\R^4} \Uet^2 \Uetm^2  \, dx = O\Big(\frac{\epsilon^4}{t^4}\log\frac{\epsilon}{t}\Big), \qquad \text{as $\frac{\epsilon}{t} \to 0^+$},
        \end{equation}
    see e.g. equation (2.10) in \cite{Bianchi-96}, or \eqref{eq:buest3} below.
    From~\eqref{eq:F3-estimate},~\eqref{eq:F3-estimate-2},~\eqref{eq:F3-estimate-3} and the elementary formula $(1+h)^{-1/2}=1-h/2+o(h)$, as $h \to 0^+$, we deduce that, for $\epsilon>0$ fixed, one has
        \begin{equation}
            \label{eq:near-infty}
                Q_{g_{\R^4}} (\Uethat ) = 6 \sqrt{2} \mathcal{S}_4 - 6 \sqrt{2} B \eps^2 t^{-2} + o( t^{-2}), \qquad \text{as $t \to +\infty$}.
        \end{equation}
    
    \smallskip

    For convenience, we introduce the following notation:
        \begin{equation*}
            %\label{eq:f-a-b}
                f(t):= Q_{g_{\R^4}} (\Uethat ), \quad a(t):= 6 \int_{\R^4} \abs{\nabla \Uethat}^2  \, dx, \quad b(t):= \bigg( \int_{\R^4} \Uethat^4 \, dx \bigg)^{\frac{1}{2}}.
        \end{equation*}
    Clearly $f(t)=a(t)/b(t)$. We claim that the following relation holds:
        \begin{gather}
            \label{eq:der-b-1}
                b'(t) = \frac{2}{b(t)} \bigg( \frac{a'(t)}{6 \mathcal{S}_4} + \frac{3}{2} \frac{d}{d t} \int_{\R^4} \Uet^2 \Uetm^2 \, dx \bigg).
        \end{gather}
    In order to keep the formulas short and highlight the key steps, we also introduce the shortcuts $U_{\pm}$ and $\dot{U}_{\pm}$ for $U_{\eps, \pm t }$ and for $\partial (U_{\eps, \pm t } )/ \partial t$, respectively. Then, from the definition of $b(t)$ we obtain
        \begin{align}
            \label{eq:der-b-2}
            b'(t) & = \frac{2}{b(t)} \int_{\R^4} (U_{+} + U_{-})^3 (\dot{U}_{+} + \dot{U}_{-}) \, dx.
        \end{align}
    Expanding the product within the integral and rearranging terms, we have 
        \begin{align}
            \notag
            (U_{+} + U_{-})^3 (\dot{U}_{+} + \dot{U}_{-}) = \ & \underbrace{U_{+}^3 \dot{U}_{-} + U_{-}^3 \dot{U}_{+}}_{=T_1} 
            + \underbrace{U_{+}^3 \dot{U}_{+} + U_{-}^3 \dot{U}_{-}}_{=T_2} \\[1ex]
            \label{eq:grouping}
            & + \underbrace{3 U_{+}^2 U_{-} \dot{U}_{+} + 3 U_{-}^2 U_{+} \dot{U}_{-}}_{=T_3} 
            + \underbrace{3 U_{+}^2 U_{-} \dot{U}_{-} + 3 U_{-}^2 U_{+} \dot{U}_{+}}_{=T_4}.
        \end{align}
    Now, the integral of $T_2$ coincides with the $t$-derivative of $(\norm{U_{+}}_{L^4(\R^4)}^4 + \norm{U_{-}}_{L^4(\R^4)}^4)/4$, which is a constant quantity in $t$, and therefore is zero. We also observe that $2 T_4 = 3 \partial ( U_{+}^2 U_{-}^2 ) / \partial t$. Moreover, 
        \begin{equation}
            \label{eq:bridge-A-C}
                \int_{\R^4} T_1 \, dx
                =
                \mathcal{S}_4^{-1} \underbrace{\int_{\R^4} \nabla U_{+} \cdot \nabla \dot{U}_{-} + \nabla U_{-} \cdot \nabla \dot{U}_{+} \, dx}_{=a'(t)/12}
                =
                \int_{\R^4} T_3 \, dx,
        \end{equation}
    where the first identity follows from~\eqref{eq:Y-eq-Rn} and an integration by parts, while the second one is a consequence of an integration by parts and the fact that $\dot{U}_{\pm}$ are solutions of the linearized equations
        \[
            - \Delta \dot{U}_{\pm} = 3 \mathcal{S}_4 U_{\pm}^2 \dot{U}_{\pm}.
        \]
    Finally, from~\eqref{eq:der-b-2},~\eqref{eq:grouping} and~\eqref{eq:bridge-A-C} we deduce~\eqref{eq:der-b-1}.
    
    We can now conclude the proof of~\eqref{eq:almost-mon}. Let us suppose by contradiction that $\sup f \ge 6\sqrt{2} \mathcal{S}_4$, then~\eqref{eq:near-infty} implies that $f$ attains its supremum, in particular there exists $t_0 > 0$ such that $f(t_0) \ge 6\sqrt{2} \mathcal{S}_4$ and $f'(t_0) = 0$. Therefore,
        \[
            a'(t_0) b(t_0) = b'(t_0) a(t_0) = f(t_0) \bigg( \frac{a'(t_0)}{3 \mathcal{S}_4} + 3 c'(t_0)  \bigg), \qquad c(t) := \int_{\R^4} \Uet^2 \Uetm^2 \, dx,
        \]
    where the second identity follows from~\eqref{eq:der-b-1}. The previous equation is equivalent to
        \begin{equation}
            \label{eq:aa-final-1}
            a'(t_0) \bigg( b(t_0) - \frac{f(t_0)}{3 \mathcal{S}_4} \bigg) = 3 f(t_0) c'(t_0).
        \end{equation}

    At this point we claim that $a(t)$ and $c(t)$ are monotone decreasing functions, with negative first derivative. Given the claim, from~\eqref{eq:aa-final-1} we deduce that $b(t_0) > f(t_0) / 3 \mathcal{S}_4$, and consequently $a(t_0) = f(t_0) b(t_0) > f(t_0)^2 / 3 \mathcal{S}_4$. We know that $f(t_0) \ge 6 \sqrt{2} \mathcal{S}_4$, hence $a(t_0) > 24 \mathcal{S}_4 = a(0)$. This is a contradiction, since $a(t)$ is monotone decreasing.

    It remains to show that $a(t)$ and $c(t)$ are monotone decreasing functions with negative first derivative for every $t > 0$. We start with $a(t)$. Without loss of generality, we can assume that $\eps=1$. We already know from~\eqref{eq:bridge-A-C} that
        \begin{equation}
            \label{eq:a-der1}
            a'(t)=12 \int_{\R^4} \nabla U_{+} \cdot \nabla \dot{U}_{-} + \nabla U_{-} \cdot \nabla \dot{U}_{+} \, dx = 24 \int_{\R^4} \nabla U_{-} \cdot \nabla \dot{U}_{+} \, dx = 24 \mathcal{S}_4 \int_{\R^4} \dot{U}_{+} U_{-}^3 \, dx,
        \end{equation}
    where the second identity is a consequence of $\nabla U_{+}(-x)=-\nabla U_{-}(x)$ and $\nabla \dot{U}_{+}(-x)=-\nabla \dot{U}_{-}(x)$, while the third identity follows from an integration by parts and ~\eqref{eq:Y-eq-Rn}. Using~\eqref{eq:std-norm-bub},~\eqref{eq:bub-family}, and~\eqref{eq:a-der1} we deduce that
        \begin{align}
            \notag
            a'(t) & = 24 \mathcal{S}_4 \int_{\R^4} \frac{U'(\abs{x-t\mathbf{e}_1})}{\abs{x - t\mathbf{e}_1}} (t-x \cdot \mathbf{e}_1) U^3(\abs{x+t\mathbf{e}_1}) \, dx \\[1ex] \notag
            & = 24 \mathcal{S}_4 \int_{\R^3} dz \int_{\R} \frac{U' \Big(\sqrt{\abs{\zeta - t}^2 + \abs{z}^2}\Big)}{\sqrt{\abs{\zeta - t}^2 + \abs{z}^2}} (t-\zeta) U^3\Big(\sqrt{\abs{\zeta + t}^2 + \abs{z}^2}\Big) \, d\zeta \\[1ex] \notag
            & = - 24 \mathcal{S}_4 \int_{\R^3} dz \int_{\R} \frac{U' \Big(\sqrt{\abs{\zeta}^2 + \abs{z}^2}\Big)}{\sqrt{\abs{\zeta}^2 + \abs{z}^2}} \zeta U^3\Big(\sqrt{\abs{\zeta + 2t}^2 + \abs{z}^2}\Big) \, d\zeta \\[1ex] \label{eq:a-der2}
            & = 24 \mathcal{S}_4 \int_{\R^3} dz \int_{0}^{+\infty} \frac{U' \Big(\sqrt{\abs{\zeta}^2 + \abs{z}^2}\Big)}{\sqrt{\abs{\zeta}^2 + \abs{z}^2}} \zeta \Big[ U^3\Big(\sqrt{\abs{\zeta - 2t}^2 + \abs{z}^2}\Big) - U^3\Big(\sqrt{\abs{\zeta + 2t}^2 + \abs{z}^2}\Big) \Big]  \, d\zeta.
        \end{align}
    Combining the monotonicity of $U$ with the observation that $\abs{\zeta + 2t}^2 > \abs{\zeta - 2t}^2$ for every $t, \zeta > 0$, the conclusion follows from~\eqref{eq:a-der2}.
    Just like for the function $a(t)$, we have that
        \begin{equation}
            \label{eq:c-der1}
            c'(t) = 2 \int_{\R^4} \big(U_{+}^2 U_{-} \dot{U}_{-} + U_{-}^2 U_{+} \dot{U}_{+}\big) \, dx = 4 \int_{\R^4} U_{-}^2 U_{+} \dot{U}_{+} \, dx,
        \end{equation}
    where the second identity is a consequence of $U_{+}(-x)=U_{-}(x)$ and $\dot{U}_{+}(-x)=\dot{U}_{-}(x)$. Using again~\eqref{eq:std-norm-bub},~\eqref{eq:bub-family},  from~\eqref{eq:c-der1} we deduce that
        \begin{align*}
            \notag
            c'(t) & = 4 \int_{\R^4} \frac{U'(\abs{x-t\mathbf{e}_1})}{\abs{x - t\mathbf{e}_1}} (t-x \cdot \mathbf{e}_1) U(\abs{x-t\mathbf{e}_1}) U^2(\abs{x+t\mathbf{e}_1}) \, dx \\[1ex] %\label{eq:c-der2}
            & = 4 \int_{\R^3} dz \int_{0}^{+\infty} \frac{U' (\abs{y})}{\abs{y}} \zeta U (\abs{y}) \Big[ U^2\Big(\sqrt{\abs{\zeta - 2t}^2 + \abs{z}^2}\Big) - U^2\Big(\sqrt{\abs{\zeta + 2t}^2 + \abs{z}^2}\Big) \Big]  \, d\zeta,
        \end{align*}
    where in the second identity we set $\abs{y}^2 = \abs{\zeta}^2 + \abs{z}^2$. As in the case of the function $a(t)$, the conclusion follows directly from the radial monotonicity and the positivity of the function $U$.
\end{proof}

\subsection{Properties of conical metrics}

We list and prove here some properties useful 
to understand metrics near conical points. 
We begin with the following general fact.

\begin{lemma}
    \label{lemma:Analisi-2}
    Let $k \ge 1$ be an integer number and let $a(t,y) \in C^{\infty}(\R \times \R^n)$. Then, the function $b(x):=\abs{x}^k a (\abs{x}, x/ \abs{x})$, defined to be zero at $x=0$, is of class $C^{k-1,1}_{loc}(\R^n)$. In particular, if  $a(t,y)$ does not depend on the $t$-variable, then $b \in C^{k-1,1}(\R^n)$.
\end{lemma}

\begin{proof}
    We prove the claim by induction. If $k=1$, then $b$ is continuous on the whole $\R^n$. Moreover, for every $i \in \{ 1, \dots, n \}$, we have
        \begin{equation*}
            %\label{eq:ind-Lip-1}
                \frac{\partial b}{\partial x_i}(x) = \frac{x_i}{\abs{x}} a \bigg(\abs{x},  \frac{x}{\abs{x}} \bigg) + x_i \frac{\partial a}{\partial t} \bigg(\abs{x}, \frac{x}{\abs{x}} \bigg) + \sum_{j=1}^n \frac{\partial a}{\partial y_j} \bigg(\abs{x}, \frac{x}{\abs{x}} \bigg) \bigg( \delta_{ij}  - \frac{x_i x_j}{\abs{x}^2} \bigg).
        \end{equation*}
    This implies that $\nabla b \in L^{\infty}_{loc}(\R^n)$, therefore $b$ is a locally Lipschitz function on $\R^n$. It is clear that if the function $a$ does not depend on the first variable, then $\nabla b \in L^{\infty}(\R^n)$, and consequently, $b$ is globally Lipschitz.

    We now consider $b(x):=\abs{x}^{k+1} a (\abs{x}, x/ \abs{x})$. As before, for every $i \in \{ 1, \dots, n \}$, we have
        \begin{align*}
                \frac{\partial b}{\partial x_i}(x) & = (k+1) \abs{x}^{k} \frac{x_i}{\abs{x}} a \bigg(\abs{x}, \frac{x}{\abs{x}} \bigg) +
                \abs{x}^{k} x_i \frac{\partial a}{\partial t} \bigg(\abs{x}, \frac{x}{\abs{x}} \bigg) + \abs{x}^k \sum_{j=1}^n \frac{\partial a}{\partial y_j} \bigg(\abs{x},  \frac{x}{\abs{x}} \bigg) \bigg( \delta_{ij}  - \frac{x_i x_j}{\abs{x}^2} \bigg) \\[0.5ex]
                & =
                \abs{x}^k \bar{a}_i \bigg(\abs{x}, \frac{x}{\abs{x}} \bigg), \qquad
            \bar{a}_i(t,y):= (k+1) y_i a(t, y) + t y_i \frac{\partial a}{\partial t} (t, y) + \sum_{j=1}^n \frac{\partial a}{\partial y_j} (t, y) \big( \delta_{ij} -  y_i y_j \big).
        \end{align*}
    By the inductive hypothesis, we know that $\nabla b \in C^{k-1,1}_{loc}(\R^n)$. This implies that $b \in C^{k,1}_{loc}(\R^n)$. A similar conclusion holds when the function $a$ does not depend on the first variable.
\end{proof}

We introduce the map
\begin{equation}\label{eq:Phimap-def}
            \Phi \colon \R^n \setminus \{ 0 \}  \to (0,+\infty) \times \Sp^{n-1}; \quad \qquad 
            x  \mapsto (\abs{x}, x/\abs{x}) 
\end{equation}   
In the sequel, we will use the same symbol to denote the restriction of $\Phi$ to the set $B_R(0)$.

\begin{corollary}\label{cor:metric-smooth-extension}
    Let $\tilde{g}$ be a smooth metric on $(0,\delta] \times \Sp^{n-1}$. Let us assume that there exists an integer number $k \ge 1 $ such that
        \begin{equation} \label{eq:reg-cond}
            \tilde{g} - g_0 = r^{k+2} \nu(r),
        \end{equation}
    for some smooth function $\nu \colon [0,\delta] \to \Gamma(\mathcal{S}^2(\Sp^{n-1}))$, where $\Gamma(\mathcal{S}^2(\Sp^{n-1}))$ denotes the space of smooth symmetric two-tensor fields on $\Sp^{n-1}$, and $g_0:=dr^2 +r^2 h_0$ is the purely conical metric on $(0,\delta] \times \Sp^{n-1}$.

    Let $g := \Phi^* \tilde{g}$ be the metric on $\overline{B}_{\delta}(0) \setminus \{ 0 \}$ defined as the pull-back of the metric $\tilde{g}$ through the map $\Phi$ introduced in~\eqref{eq:Phimap-def}. Then, $g$ extends to a $C^{k-1,1}$-metric on $\overline{B}_{\delta}(0)$. 
\end{corollary}

\begin{proof}
    We have $d \Phi_x \colon \R^n \to \R \times T_x \Sp^{n-1}$. A standard computation shows that
        \[
            d \Phi_x(v) = \bigg( \frac{x \cdot v}{\abs{x}}, \frac{v}{\abs{x}} - \frac{(x\cdot v) x}{\abs{x}^3} \bigg).
        \]
    From the previous formula, we deduce that
        \begin{align*} 
            \notag
            g(x)(v,w) := (\Phi^* \tilde{g})(x)(v,w) & = (\Phi^* g_0)(x)(v,w) + (\Phi^* (r^{k+2} \nu(r)))(x)(v,w) \\[1ex]
            & = v \cdot w + \abs{x}^{k} \nu(\abs{x})\bigg( v - \frac{(x\cdot v) x}{\abs{x}^2}, w - \frac{(x \cdot w) x}{\abs{x}^2} \bigg) \\[1ex] \label{eq:pull-back1}
            & = \sum_{i,j = 1}^n v^i w^j \bigg( \delta_{ij} + \abs{x}^k \tilde{a}_{ij}\bigg(\abs{x}, \frac{x}{\abs{x}}\bigg) \bigg),
        \end{align*}
    where
        \begin{equation*}
            \label{eq:pull-back2}
                \tilde{a}_{ij}(t,y):= \nu(t)(y) \Big( e_i - (y \cdot e_i) y, e_j - (y \cdot e_j) y \Big).
        \end{equation*}
    By assumption $\tilde{a}_{ij} \in C^{\infty}([0,\delta] \times \Sp^{n-1})$, so the conclusion follows by applying Lemma~\ref{lemma:Analisi-2} to a smooth extension $a_{ij} \in C^{\infty}(\R \times \R^n)$ of $\tilde{a}_{ij}$, for every $i,j \in \{ 1, \dots, n \}$.
\end{proof}

We will introduce next some useful notation. 
Denote by $B^g_r(x)$ the geodesic ball of radius $r$ and center $x$ in the metric $g$; we will omit the superscript when dealing with the Euclidean metric or when it is clear from the context, while we will omit the center when it coincides with the origin.

Let $\pi \colon \Sp^3\to\R\projective^3$ be the antipodal projection ($\pi(x)=\pi(-x)$). For any tensor field $T$ of rank $(0,q)$ we can define its equivariant lift $\widetilde{T}$ as the pullback $\widetilde{T}:=\pi^*T$. 
Given a conical point $P\in M$ and $\delta>0$ small, we can consider the projection ${\sigma}_P$ given by
    \begin{align}
    \notag
        {\sigma}_P \colon (0,2\delta)\times\Sp^3 & \to B_{2\delta}^g(P) \setminus \{P\}\cong\big( (0,2\delta)\times \R\projective^3, ds^2+s^2h(s)\big) \\
        \label{eq:prj-def}
        (s,y)&\mapsto(s,\pi(y)).
    \end{align}
    Using ${\sigma}_P$, we can define on $(0,2\delta)\times \Sp^3$ the (equivariant) pullback $\tilde{g}:={\sigma}_P^*g$ of $g$.
    Moreover, by virtue of Corollary \ref{cor:metric-smooth-extension} and employing a slight abuse of notation, we know that $\tilde{g}:=\Phi^*\tilde{g}$ extends at the origin with $C^{0,1}$ regularity, that is, we can regard $\tilde{g}$ as a $C^{0,1}$ metric over ${B}_{2\delta}\subset \R^4$ which is smooth outside $0$. Hence, we can also regard ${\sigma}_P$ as a map ${\sigma}_P:B_{2\delta}^{\tilde{g}}\to B_{2\delta}^g(P)$.

By the aforementioned \cite[Theorem 1.1]{freguglia-malchiodi-2024-pre}, we can obtain existence of Yamabe (minimizing) metrics as soon as \eqref{s:xi-conf} is verified. We are therefore interested in  situations in which \eqref{s:xi-conf} is \emph{not} verified. In this last case, we can further reduce to assume $h'(0)=0$  as a corollary of the following argument. Since the argument works in a general setting, we state the result for a generic link $(Y,h_0)$.

\begin{proposition}\label{l:Melrose}
    Let $(Y,h_0)$ be a smooth closed Riemannian $(n-1)$-manifold. Let $f \colon Y \to \R$ be a smooth function, and let $\delta > 0$ be such that $\delta \norm{f}_{\infty} < 1/2$. Then, the map $\alpha \colon [0,\delta] \times Y \to [0,5\delta/4] \times Y$ defined as
        \begin{equation}
            \label{eq:def-alpha-map}
                \alpha(s,z) := (\alpha^1(s,z), \alpha^2(s,z)) = \bigg(s-\frac{s^2}{2} f(z), \psi_s(z) \bigg)
        \end{equation}
    is an embedding, where $\psi_t \colon Y \to Y$ denotes the one-parameter family of diffeomorphisms generated by the vector field $\nabla f/2$, where the gradient is taken with respect to the metric $h_0$.
    
    Moreover, let us consider the metric $\bar{g}$ on $(0,\delta] \times Y$ defined as
        \begin{equation}
            \label{eq:def-metric-bar}
                \bar{g}:=\big(1+2sf(z)\big) \alpha^*\big(dr^2 + r^2 h(r)\big),
        \end{equation}
    where $h \colon [0,5\delta/4] \to \Gamma(\mathcal{S}^2(Y))$ is a smooth function such that $h(0)=h_0$, and $\Gamma(\mathcal{S}^2(Y))$ denotes the space of smooth symmetric two-tensor fields on $Y$. 

    Then, there exist a neighborhood $U$ of $\{ 0 \} \times Y$ in $[0,\delta] \times Y$, a positive number $\eps > 0$, and a diffeomorphism $\Upsilon \colon [0,\eps] \times Y \to U$ such that $\Upsilon(0,p)=(0,p)$ for every $p \in Y$, and
        \begin{equation} \label{eq:reg-cond-3}
            \Upsilon^*\bar{g} = dx^2+x^2 \tilde{h}(x),
        \end{equation}
    for some smooth function $\tilde{h} \colon [0,\eps] \to \Gamma(\mathcal{S}^2(Y))$. In addition, it holds that
        \begin{equation} \label{eq:reg-cond-4}
            \tilde{h}(0) = h_0, \quad \text{and} \quad \tilde{h}'(0) = h'(0) + \nabla^2_{h_0} f + f h_0.
        \end{equation}
\end{proposition}

\begin{proof}
    We start by proving that the map $\alpha$ is injective. Suppose by contradiction that there exist $(s_1,z_1) \neq (s_2, z_2)$ such that $\alpha(s_1,z_1) = \alpha(s_2, z_2)$. In particular, we have $z_1 = \psi_{s_2-s_1}(z_2)$ and we can assume that $s_2 > s_1$. Since $\psi_t$ is the flow generated by the vector field $\nabla f / 2$, it follows that $f(z_1) \ge f(z_2)$. Therefore,
        \[
            \alpha(s_1,z_1) = s_1 - \frac{s_1^2}{2} f(z_1) \le s_1 - \frac{s_1^2}{2} f(z_2) < s_2 - \frac{s_2^2}{2} f(z_2) = \alpha(s_2,z_2),
        \]
    which contradicts $\alpha(s_1,z_1) = \alpha(s_2, z_2)$. Here, we used the fact that, for every $z \in Y$, and every $\delta > 0$ such that $ \delta \norm{f}_{\infty} < 1/2$, the map $\alpha^1(\cdot,z)$ is strictly increasing on $[0,\delta]$.

    We now prove that $\alpha$ is an immersion. To this end, we compute:
        \begin{equation}
            \label{eq:alpha-diff}
                d \alpha_{(s,z)}(t,v) = \big(d \alpha^1_{(s,z)}(t,v), d \alpha^2_{(s,z)}(t,v)\big) = \bigg( \big(1 - s f(z) \big) t - \frac{s^2}{2} df_z(v), \frac{t}{2} \nabla f(\psi_s(z)) + d (\psi_s)_z(v) \bigg).
        \end{equation}
    Fixing $(\ell,w) \in \R \times T_{\psi_s(z)} Y $, we have to solve $d \alpha_{(s,z)}(t,v)=(\ell,w)$. We consider $v_0, v_1 \in T_z Y$ such that $d (\psi_s)_z(v_0) = w$ and $d (\psi_s)_z(v_1) = -\nabla f(\psi_s(z))/2$. In particular, $v_1$ does not depend on $s$, see for example~\cite[Equation~(9.17)]{Lee-Smooth13}, and thus $v_1=- \nabla f(z)/2$. We define $t_* \in \R$ by setting
        \begin{align*}
            t_* = \bigg( 1 - s f(z) - \frac{s^2}{2} df_z(v_1) \bigg)^{-1} \Big( \ell + \frac{s^2}{2} df_z(v_0) \Big) = \bigg( 1 - s f(z) + \frac{s^2}{4} \abs{\nabla f(z)}^2 \bigg)^{-1} \Big( \ell + \frac{s^2}{2} df_z(v_0) \Big).
        \end{align*}
    We observe that the quantity appearing inside the first bracket in the definition of $t_*$ is strictly positive on $[0,\delta] \times Y$, and therefore $t_*$ is well-defined. At this point, one verifies that $d \alpha_{(s,z)}(t_*,v_0+t_*v_1)=(\ell,w)$, thereby concluding the proof that $\alpha$ is an embedding.

    We proceed to show that equation~\eqref{eq:reg-cond-3} holds. First, we prove that there exists a smooth symmetric two-tensor field $H \in \Gamma(\mathcal{S}^2([0,\delta] \times Y))$ such that
        \begin{equation} \label{eq:reg-cond-5}
            \bar{g} = ds^2 + s^2 H(s,z), 
        \end{equation}
    and such that $\iota_0^* H$ defines a Riemannian metric on $Y$, where $\iota_c(z)=(c,z)$ for all $c \ge 0$.
        
    Once this is done, the metric in~\eqref{eq:reg-cond-5} satisfies the assumptions of~\cite[Theorem~1.2]{MelWun04}, which directly implies the existence of $\Upsilon$ such that~\eqref{eq:reg-cond-3} holds.  

    In order to prove~\eqref{eq:reg-cond-5}, we note that
        \begin{equation}
            \label{eq:split1}
                \alpha^*\big(dr^2 + r^2 h(r)\big) = \alpha^* \big( dr^2 \big) + (\alpha^1)^2 \alpha^*\big( h(r) \big).
        \end{equation}
    From~\eqref{eq:alpha-diff}, we obtain
        \begin{equation}
            \label{eq:split2}
            \alpha^* \big( dr^2 \big) = \big( 1 - s f(z) \big)^2 ds^2 - \big(  1 - s f(z)  \big) \frac{s^2}{2} ds \odot df_z + \frac{s^4}{4} df_z \otimes df_z,
        \end{equation}
    where $ds \odot df_z = ds \otimes df_z + df_z \otimes ds$.
    Therefore, we have
        \begin{align}
            \notag
            \big(1+ 2 s f(z)\big) \alpha^* \big( dr^2 \big) = & ds^2 - s^2 f(z)^2 \big(3 - 2 s f(z)\big) ds^2 \\[1ex] \label{eq:split3}
            & - \big(1+2sf(z)\big) \big( 1 - s f(z)  \big) \frac{s^2}{2} ds \odot df_z + \big(1+ 2sf(z)\big) \frac{s^4}{4} df_z \otimes df_z.
        \end{align}
    By combining~\eqref{eq:split1},~\eqref{eq:split3} and the identity $\alpha^1(s,z)=s(1-sf(z)/2)$, we derive~\eqref{eq:reg-cond-5}.
    
    For later use, we write the following expressions in a more explicit form
        \begin{equation*}
            \label{eq:extra-1}
                \alpha^*\big(dr^2 + r^2 h(r)\big)(\partial_s, \partial_{z_j}), \quad \text{and} \quad \alpha^*\big(dr^2 + r^2 h(r)\big)(\partial_{z_i}, \partial_{z_j}).
        \end{equation*}
    From~\eqref{eq:split2}, we deduce that
        \begin{gather}
            \label{eq:extra-2}
            \alpha^* \big( dr^2 \big)(s,z)(\partial_s, \partial_{z_j}) = - \big(  1 - s f(z)  \big) \frac{s^2}{2} \frac{\partial f}{\partial z_j}(z), \\[1ex] \label{eq:extra-2B}
            \alpha^* \big( dr^2 \big)(s,z)(\partial_{z_i}, \partial_{z_j}) = \frac{s^4}{4} \frac{\partial f}{\partial z_i}(z) \frac{\partial f}{\partial z_j}(z).
        \end{gather}
    Similarly from~\eqref{eq:alpha-diff},
        \begin{gather}
            \label{eq:extra-3}
            \alpha^* \big( r^2 h(r) \big)(s,z)(\partial_s, \partial_{z_j}) = s^2 \big( 1 - \frac{s}{2} f(z) \big)^2 h(\alpha^1(s,z))(\alpha^2(s,z)) \bigg( \frac{1}{2}\nabla f(\psi_s(z)),d(\psi_s)_z(\partial_{z_j})\bigg), \\[1ex] \label{eq:extra-3B}
            \alpha^* \big( r^2 h(r) \big)(s,z)(\partial_{z_i}, \partial_{z_j}) = s^2 \big( 1 - \frac{s}{2} f(z) \big)^2 h(\alpha^1(s,z)) (\alpha^2(s,z)) \big(d(\psi_s)_z(\partial_{z_i}),d(\psi_s)_z(\partial_{z_j})\big).
        \end{gather}
    In particular, from~\eqref{eq:extra-2},~\eqref{eq:extra-2B},~\eqref{eq:extra-3}, and~\eqref{eq:extra-3B}, we obtain
        \begin{gather}
            \label{eq:extra-4}
            H(0,z)(\partial_s, \partial_{z_j}) = \lim_{s \to 0^+} s^{-2} \bar{g}(s,z)(\partial_s, \partial_{z_j}) = - \frac{1}{2} \frac{\partial f}{\partial z_j}(z) + \frac{1}{2} h_0(z)\big(\nabla f(z), \partial_{z_j} \big) = 0, \\[1ex] \label{eq:extra-4B}
            H(0,z)(\partial_{z_i}, \partial_{z_j}) = \lim_{s \to 0^+} s^{-2} \bar{g}(s,z)(\partial_{z_i}, \partial_{z_j}) = h_0(z)(\partial_{y_i}, \partial_{y_j}).
        \end{gather}
    Since~\eqref{eq:reg-cond-5} implies that $\iota_s^*H = s^{-2} \iota_s^* \bar{g}$. Then, from~\eqref{eq:extra-2B} and~\eqref{eq:extra-3B}, we conclude that
         \begin{equation}
            \label{eq:H-explicit}
                (\iota_s^* H)(z) =\big( 1 + 2 s f(z) \big) \bigg( \frac{s^2}{4} df_z \otimes df_z + \big( 1 - \frac{s}{2} f(z) \big)^2 \psi_s^* \big( h(\alpha^1(s,z)) \big)(z) \bigg).
        \end{equation}

    It remains to prove~\eqref{eq:reg-cond-4}. We claim that
        \begin{equation} \label{eq:diff-ID}
            d \Upsilon_{(0,p)} = \mathrm{id}, \qquad \forall p \in Y.
        \end{equation}
    Since the restriction of $\Upsilon$ to $\{ 0 \} \times Y$ is the identity map, it follows that
        \begin{equation*}
            %\label{eq:U-yy-EXTRA}
                d \Upsilon_{(0,p)}(0,v) = (0,v), \qquad \forall p \in Y, \forall v \in T_y Y,
        \end{equation*}
    which, when expressed in local coordinates, becomes
        \begin{equation}
            \label{eq:U-yy}
                \frac{\partial \Upsilon^i}{\partial p_j}(0,p) = \delta_{ij}, \qquad \forall p \in Y, \forall j \in \{ 1, \dots, n-1 \},
        \end{equation}
    where $\delta_{ij}$ denotes the Kronecker delta. Using~\eqref{eq:reg-cond-3} and~\eqref{eq:reg-cond-5}, we can write
        \begin{equation*}
            %\label{eq:H-explicit-pre}
            dx^2 + x^2 \tilde{h}(x) = \Upsilon^* \big( ds^2 + s^2 H(s,z) \big).
        \end{equation*}
   As a consequence, we obtain the following identities:
        \begin{gather}
            \label{eq:1-x-1}
            1 = \bigg( \frac{\partial \Upsilon^s}{\partial x}\bigg)^2+ (\Upsilon^s)^2 H(\Upsilon) \bigg(\frac{\partial \Upsilon}{\partial x}, \frac{\partial \Upsilon}{\partial x} \bigg), \\[1ex]
            \label{eq:1-y-1}
            0 = \frac{\partial \Upsilon^s}{\partial x} \frac{\partial \Upsilon^s}{\partial p_j} + (\Upsilon^s)^2 H(\Upsilon) \bigg(\frac{\partial \Upsilon}{\partial x}, \frac{\partial \Upsilon}{\partial p_j} \bigg), \\[1ex]
            \label{eq:1-yy-1}
            x^2 \tilde{h}(x)(\partial_{p_i}, \partial_{p_j}) = \frac{\partial \Upsilon^s}{\partial p_i} \frac{\partial \Upsilon^s}{\partial p_j} + (\Upsilon^s)^2 \big( (\Upsilon \circ \iota_x)^*H \big) (p) (\partial_{p_i}, \partial_{p_j}),
        \end{gather}
    where $\Upsilon$ and its derivatives are implicitly evaluated at $(x,p)$, and $\iota_x(p)=(x,p)$.
    
    Taking the limit as $x$ goes to zero in~\eqref{eq:1-x-1}, and using the fact that~$\Upsilon^s(0,p)=0$, we deduce that
        \begin{equation}
            \label{eq:1-x-2}
                \frac{\partial \Upsilon^s}{\partial x}(0,p) = 1, \qquad \forall p \in Y.
        \end{equation}

    Taking the limit as $x$ goes to zero in~\eqref{eq:1-y-1}, and using~\eqref{eq:1-x-2}, we deduce that
        \begin{equation}
            \label{eq:1-y-2}
                \frac{\partial \Upsilon^s}{\partial p_j}(0,p) = 0, \qquad \forall p \in Y, \forall j \in \{ 1, \dots, n-1 \}.
        \end{equation}

    Moreover, differentiating both sides of~\eqref{eq:1-x-1} with respect to $x$, taking the limit as $x$ goes to zero, and using~\eqref{eq:1-x-2}, we deduce that
        \begin{equation}
            \label{eq:2-x-1}
                \frac{\partial^2 \Upsilon^s}{\partial x^2}(0,p) = 0, \qquad \forall p \in Y.
        \end{equation}
        
    Differentiating both sides of~\eqref{eq:1-x-2} and~\eqref{eq:2-x-1} with respect to $p_j$, we deduce that
        \begin{gather}
            \label{eq:2-xy-1}
                \frac{\partial^2 \Upsilon^s}{\partial x \partial p_j}(0,p) = 0, \qquad \forall p \in Y, \forall j \in \{ 1, \dots, n-1 \}, \\[1ex]  \label{eq:3-xxy-1}
                \frac{\partial^3 \Upsilon^s}{\partial x^2 \partial p_j}(0,p) = 0, \qquad \forall p \in Y, \forall j \in \{ 1, \dots, n-1 \}.
        \end{gather}

    Given that $\Upsilon$ is smooth, a Taylor expansion in the $x$-variable, combined with~\eqref{eq:1-y-2},~\eqref{eq:2-xy-1}, and~\eqref{eq:3-xxy-1}, implies that $(\partial \Upsilon^s / \partial p_j)(x,p) = O(x^3)$ as $x$ goes to zero. Therefore, dividing both sides of~\eqref{eq:1-y-1} by $x^2$, taking the limit as $x$ goes to zero, and using~\eqref{eq:U-yy} and~\eqref{eq:1-y-2}, we obtain
        \begin{equation}
            \label{eq:pull-H-1}
                0 = H(0,p)\bigg( \frac{\partial \Upsilon}{\partial x}(0,p), \partial_{z_j} \bigg)
                = H(0,p)\big(\partial_s, \partial_{z_j} \big) + H(0,p)\big(w(p), \partial_{z_j} \big), 
        \end{equation}
    where $w(p) \in T_pY$ is such that $(\partial \Upsilon/\partial x)(0,p)= \partial_s + w(p)$.

    To conclude the proof of~\eqref{eq:diff-ID}, it remains to show that $w(p)=0$. Combining~\eqref{eq:extra-4} and~\eqref{eq:pull-H-1}, we deduce that
        \begin{equation*}
            %\label{eq:Lie1}
                H(0,p)\big(w(p), \partial_{z_j} \big) = 0, \qquad \forall p \in Y, \forall j \in \{ 1, \dots, n-1 \},
        \end{equation*}
    and since $\iota_0^* H$ is a metric on $Y$, this implies that $w(p) = 0$.

    At this point, we are ready to prove~\eqref{eq:reg-cond-4}. We begin by verifying that the first of the two identities holds. We already know that $(\partial \Upsilon^s / \partial p_j)(x,p) = O(x^3)$ and $\Upsilon^s(x,p)=x+O(x^3)$ as $x$ goes to zero. Therefore, dividing both sides of~\eqref{eq:1-yy-1} by $x^2$, taking the limit as $x$ goes to zero, we obtain
        \begin{equation*}
                \tilde{h}(0)(p)(\partial_{p_i}, \partial_{p_j}) = \iota_0^* H(p) (\partial_{p_i}, \partial_{p_j}) = h_0(p)(\partial_{y_i}, \partial_{y_j}),
        \end{equation*}
    where the last identity follows from~\eqref{eq:extra-4B}.

    We now turn to the proof of the second identity in~\eqref{eq:reg-cond-4}. Dividing both sides of~\eqref{eq:1-yy-1} by $x^2$, differentiating with respect to $x$, and recalling that $(\partial \Upsilon^s / \partial p_j)(x,p) = O(x^3)$ and $\Upsilon^s(x,p)=x+O(x^3)$ as $x$ goes to zero, we deduce that
        \begin{equation}
            \label{eq:H-DER}
                \tilde{h}'(0) = \frac{d}{dx} \big( (\Upsilon \circ \iota_x)^*H \big) \Big|_{x=0} = \frac{d}{ds} \iota_s^* H \Big|_{s=0},
        \end{equation}
    where the last identity follows from~\eqref{eq:diff-ID}, and in particular from the fact that $(\partial \Upsilon/\partial x)(0,p) = \partial_s$.
    
    To complete the proof, we compute the derivative of $\iota_s^* H$, whose explicit expression is given in~\eqref{eq:H-explicit}. In particular, we find that
        \begin{equation}
            \label{eq:H-DER-2}
                \frac{d}{ds} \iota_s^* H \Big|_{s=0} = 2 f h_0 - fh_0 + \frac{d}{ds} \psi_s^* \big( h(\alpha^1(s,z)) \big) \Big|_{s=0}.
        \end{equation}
    Moreover it is known, see for example~\cite[Proposition~1.2.1]{Topp06}, that
        \begin{equation}
            \label{eq:H-DER-3}
            \frac{d}{ds} \psi_s^* \big(h(\alpha^1(s,z)) \big) \Big|_{s=0} = \mathcal{L}_{\frac{\nabla f}{2}} h_0 + h'(0) \frac{\partial \alpha^1}{\partial s}(0,\cdot) = \nabla^2_{h_0} f + h'(0),
        \end{equation}
    where $\mathcal{L}_{\nabla f/2} h_0$ denotes the Lie derivative of $h_0$ with respect to the vector field $\nabla f/2$, and we used the identity $\nabla^2_{h_0} f =\mathcal{L}_{\nabla f/2} h_0$.

    Finally, from~\eqref{eq:H-DER},~\eqref{eq:H-DER-2}, and~\eqref{eq:H-DER-3}, we obtain
        \[
            \tilde{h}'(0) = h'(0) + \nabla^2_{h_0} f + f h_0.
        \]
    This completes the proof.
\end{proof}

We will systematically make use of the following result in the sequel:

\begin{corollary}\label{cor:c1extension}
    Let $(M^n,g)$ be a closed manifold with finitely-many conical points $\{ P_1, \dots, P_l \}$. Assume that, for each $i$, condition~\eqref{eq:diffeo} is satisfied with link $Y_i = \Sp^{n-1} / \Gamma_i$, where $\Gamma_i < O(n)$ is a finite subgroup acting freely on $\Sp^{n-1}$. Suppose that, for each $i$, condition~\eqref{s:xi-conf} does not hold.
    
    Then, there exists a positive function $\mathfrak{u} \in C(M) \cap W^{1,2}(M)$, smooth away from the conical points, such that the following properties hold:
    \begin{itemize}
        \item[1)] For every $i$, there exists a diffeomorphism $\tilde{\sigma}_i$ for which the singular manifold $(M^n, \mathfrak{u} g)$ satisfies~\eqref{eq:diffeo} with $\tilde{\sigma}_i$.
        \item[2)] If $\pi_i \colon \Sp^{n-1} \to \Sp^{n-1} / \Gamma_i$ denotes the quotient map, then the pullback metric $(\tilde{\sigma}_i^{-1} \circ (\mathrm{id \times \pi_i}) \circ \Phi)^* ( \mathfrak{u} g)$ extends to a $C^{1,1}$-metric in a neighborhood of $0 \in \R^n$, where $\Phi$ is the map defined in~\eqref{eq:Phimap-def}.
    \end{itemize}
\end{corollary}

\begin{proof}
    We recall that by definition of conical point, for every $i \in \{ 1, \dots, l \}$, there exist an open neighborhood $\mathcal{U}_i$ of $P_i$ and a diffeomorphism $\sigma_i$ such that
        \[
            \sigma_i \colon \mathcal{U}_i \setminus \{ P_i \} \to (0,1) \times \Sp^{n-1} / \Gamma_i, \quad \text{and} \quad (\sigma_i)_{*} g = ds^2 + s^2 h_i(s).
        \]
    Moreover, since we assume that condition~\eqref{s:xi-conf} does not hold, for every $i \in \{1, \dots, l\}$ there exists a smooth function $f_i \colon \Sp^{n-1} / \Gamma_i \to \R$ such that
        \begin{equation} \label{local-structure}
            h_i'(0) = -\nabla^2_{h_{\Gamma_i}} f_i - f_i h_{\Gamma_i},
        \end{equation}
    where $h_i(0)=h_{\Gamma_i}$ and $h_{\Gamma_i}$ denotes the metric of constant sectional curvature one on $\Sp^{n-1} / \Gamma_i$.

    Using cut-off functions, one can define a positive  $\mathfrak{u} \in C^{\infty}(M \setminus \{ P_1, \dots, P_l  \})$ that, locally around each conical point, satisfies
        \[
            \mathfrak{u} \circ \sigma_i^{-1} = F_i \circ \alpha_i^{-1}, \qquad F_i(s,z):=1 + 2sf_i(z),
        \]
    where $\alpha_i$ is the embedding defined in~\eqref{eq:def-alpha-map}. Since $\alpha_i^1(0,y)=0$, we deduce that $\mathfrak{u}$ extends to a continuous function on the whole $M$. In particular, it holds that
        \[
            \abs{\nabla_{g_i} F_i}^2 = 4 \big( f_i^2 + \abs{\nabla_{h_{\Gamma_i}} f_i}^2 \big), \qquad g_i := ds^2 + s^2 h_{\Gamma_i}.
        \]
    This observation, combined with the identity $d (\alpha_i)_{(0,z)}(t,v)=(t, t \nabla f_i(z)/2 + v)$, see~\eqref{eq:alpha-diff}, implies that the function $\mathfrak{u} \in W^{1,2}(M)$. At this point, we notice that by construction
        \[
           \bar{g} = \alpha_i^* (\sigma_i)_* (\mathfrak{u} g),
        \]
    where $\bar{g}$ is the metric defined in~\eqref{eq:def-metric-bar}. Therefore, the conclusion of the first part of the corollary follows from Proposition~\ref{l:Melrose} taking $\tilde{\sigma}_i = \Upsilon_i^{-1} \circ \alpha_i^{-1} \circ \sigma_i$. Moreover, from~\eqref{local-structure} and the identity~\eqref{eq:reg-cond-4} proved in Proposition~\ref{l:Melrose}, we deduce that
        \[
            (\tilde{\sigma}_i)_*(\mathfrak{u} g) = dx^2 + x^2 \tilde{h}_i(x), \quad \text{with} \quad \tilde{h}_i(x) = h_{\Gamma_i} + O(x^2). 
        \]
    Since $\tilde{h}_i(x)$ is smooth, we can write $\tilde{h}_i(x) = h_{\Gamma_i} + x^2 \nu_i(x)$, for a smooth function $\nu_i$ with values in the space of smooth symmetric two-tensor fields on $\Sp^{n-1} / \Gamma_i$. This implies that the metric $(\mathrm{id} \times \pi_i)^* (\tilde{\sigma}_i)_*(\mathfrak{u} g)$ satisfies condition~\eqref{eq:reg-cond} with $k=2$. In particular, the conclusion of the second part of the corollary is a direct consequence of Corollary~\ref{cor:metric-smooth-extension}.
\end{proof}

As a consequence of the previous corollary, we might assume that $\tilde{g}(=\Phi^*(\tilde{g}))$ extends $C^{1,1}$ at the origin.
It follows that the Christoffel symbols of $\tilde{g}$ are $C^{0,1}$, which, in turn, implies the existence and uniqueness of solutions for the geodesic equation. It is therefore possible to define geodesic normal coordinates.
 In such coordinates around $0$ (and in the above notation), we have the following expansion for $\tilde{g}$ and for the volume element $d\mu_{\tilde{g}}$:
    \begin{equation}\label{eq:lift-metr-exp}
        \tilde{g}_{ij}(x)=\delta_{ij}+O^{''}(\abs{x}^2), \qquad d\mu_{\tilde{g}}(x)=\big(1+O^{''}(\abs{x}^2)\big)\,dx,
    \end{equation}
    where $x\in B_{2\delta}(0)$ and $O^{''}(\abs{x}^2)$ denotes an error term $\Xi\in C^{1,1}(\overline{B}_{2\delta}(0))$ such that $\Xi$ is smooth outside $0$ and $\abs{\nabla^s\Xi(x)}\leq C\abs{x}^{2-s}$ for $s=0,1,2$ and for a constant $C>0$.

%% file: greenfunction.tex
\section{\texorpdfstring{The Green's function of $L_g$ on conical manifolds}{The Green's function of Lg on conical manifolds}}\label{sec:greenfunct}

We begin this section by stating the existence of the Green's function for the conformal Laplacian $L_g$ on manifolds with conical points. We then perform a parametrix in order to obtain an expansion  near a conical point $P$ with link $\R\projective^3$; such an expansion will be fundamental in estimating the Yamabe quotient, which will be done in the next sections.

\smallskip Throughout this section  we will assume that \underline{the metric $g$ admits 
a local $\Z_2$-lift of class $C^{1,1}$}  near all conical points, which 
holds in particular when, in the notation of \eqref{eq:diffeo}, one has $h'(0)=0$, see Corollary 
\ref{cor:metric-smooth-extension}. This condition will not be necessary though for Proposition \ref{prop:green-existence}.

\subsection{Existence of Green's function}

As remarked in \cite{freguglia-malchiodi-2024-pre}, in our setting 
the scalar curvature $R_{{g}}$ is bounded by the inverse of the distance 
from the singular set, and indeed, when \eqref{s:xi-conf} is not verified, even uniformly bounded after a proper conformal change of metric, see Corollary \ref{cor:c1extension}. By this reason, $R_{{g}} \in L^q(M)$ for 
some $q > n/2$, and we fit into the framework of \cite{ACM}. Moreover, as in 
\cite{KW-JDG-75}, it is possible to prove that the sign of $\Y(M,[g])$ coincides with that of the first eigenvalue $\lambda_1(L_g)$, 
defined via Rayleigh's quotient, and since we are under assumption \eqref{eq:no-att}, we have that $\lambda_1(L_g) > 0$. The next result 
can be deduced from  \cite{Maz-edge}, but we provide a short proof for the reader's convenience and for later purposes.

\begin{proposition}\label{prop:green-existence}
    Let $(M,g)$ be a smooth stratified space with an iterated cone-edge metric, as defined in~\cite[Section~2.1]{ACM}. Assume $n \in \{3,4,5\}$, $R_g \in L^{q}(M)$ for some $q > n/2$ and  that $\lambda_1(L_g) > 0$. Then, there exists $G \colon M \times M \to (0,+\infty)$ such that
        \begin{equation*}
            %\label{eq:ident-Green}
                \int_{M} G(x,y) \phi(x) \, d\mu_g(x) = (L_g^{-1} \phi)(y), \qquad \forall \phi \in C(M),
        \end{equation*}
    where $L_g^{-1}$ denotes the inverse of the conformal Laplacian associated to the metric $g$. Moreover,
        \begin{equation*}
            %\label{eq:ident-Green2}
                \sup_{y \in M} \norm{G(\cdot,y)}_{L^r(M)} < +\infty, \qquad \forall r\in \bigg[1,\frac{n}{n-2}\bigg).
        \end{equation*}
\end{proposition}

\begin{proof}
     We first claim that $T:=L_g^{-1}$ is a continuous and linear operator from $L^p(M)$ to $L^{\infty}(M)$, for every $p > n/2$. Given the claim, the conclusion follows directly from Gel'fand's theorem as stated in~\cite[Section~3,~p.~120]{Kovalenko-Semenov78}. Here, we notice that if $p = n/2$, then the conjugate exponent satisfies $p'=n/(n-2)$. Now, we observe that if $n \in \{ 3,4,5 \}$, then $n/2 < 2^*:=2n/(n-2)$, therefore it is enough to prove the claim for every $p \in (n/2, 2^*)$. We begin by recalling that for every $f \in L^{\frac{2n}{n+2}}(M)$, there exists a unique $u \in W^{1,2}(M)$ such that $L_g u = f$. Indeed, as shown in~\cite[Propositions~1.6 and~2.2]{ACM}, for every $\nu \in [1,2^*)$, the inclusion $W^{1,2}(M) \subset L^\nu(M)$ is compact, while for $\nu=2^*$ the inclusion is continuous but not compact. By combining the above results with the assumption $\lambda_1(L_g) > 0$, the function $u$ can be obtained as the minimizer of the following well-defined energy:
        \[
           W^{1,2}(M) \ni u \longrightarrow \frac{1}{2}\int_M a \abs{\nabla u}_g^2 + R_g u^2 \, d\mu_g - \int_M fu \, d\mu_g.
        \]
     Fix now $p \in (n/2,2^*)$, and consider $f \in L^p(M) \subset L^{\frac{2n}{n+2}}(M)$. Let $u \in W^{1,2}(M)$ be the unique solution to $L_gu=f$ previously obtained. It holds that $u \in L^{2^*}(M) \subset L^p(M)$, in particular $u \in H^{2,p}(M)$, where $H^{2,p}(M) := \{ u \in L^p(M) : L_g u \in L^p(M) \}$. As for the smooth case, see~\cite[Equation~(2.5)]{CaLyVe23},  for every $p > n/2$, the space $H^{2,p}(M)$ embeds continuously into $L^{\infty}(M)$. The proof is complete.
\end{proof}

\begin{remark}
    The dimension assumption in the previous proposition is not essential and can be removed, provided one knows that $u \in W^{1,2}(M)$ and $L_gu \in L^p(M)$ imply $u \in L^p(M)$. For $n \in \{3,4,5\}$, this follows directly from  Sobolev's embedding. The implication still holds for larger values of $n$. However, since this article is mainly concerned with the case $n=4$, we  decided to keep the proof short by restricting to a low-dimensional setting.
\end{remark}

\subsection{The Green's function near a conical point}

So far, we proved that there exists a Green's function $G$ for $L_g$ over $M$. Let $P\in M$ be a conical point, and let $\tilde{g}:={\sigma}_P^*g$ be the equivariant lift of $g$ defined in a small neighborhood of $P$ as in \eqref{eq:prj-def}. We recall that  $\tilde{g}$ extends to a $C^{1,1}$-metric over $B_{2\delta}(0)\subset \R^4$ which is smooth outside the origin.

Let $q\in B_\delta^g(P)$ and define $G_q:=G(q,\cdot)$. We denote by $h_q$ the equivariant lift of $G_q$ over $\partial B_\delta(0)$:
\begin{equation*}
    h_q(y):= G_q({\sigma}_P(y)), \quad y\in \partial B_\delta(0).
\end{equation*}
We notice that the lift metric $\tilde{g}$ has, by construction, 
uniformly bounded second derivatives, and it is therefore of class $W^{2,p}$ for all $p > 1$. 
Using for example the arguments in \cite[Section 4]{ACR-25}, given $x\in B_\delta(0)$, we deduce the existence (and uniqueness) of the Green's function $\tilde{G}_x$ of $L_{\tilde{g}}$ on $B_\delta(0)$ with Dirichlet boundary datum:
\begin{equation}\label{eq:dir-green-lift}
    \begin{cases*}
    L_{\tilde{g}}\tilde{G}_x=4a\pi^2\delta_x & \text{in $B_\delta(0)$} \\
    \tilde{G}_x=0 & \text{in $\partial B_\delta(0)$.}
    \end{cases*}
\end{equation}
We know that $\tilde{G}_x\in W^{2,p}_{{loc}}({B}_\delta\backslash\{x\})\cap C^{\infty}_{{loc}}(\overline{B}_\delta\backslash \{0,x\})$ for any $p\in [1,+\infty)$.
Let $H_x$ be the solution to
\begin{equation*}
    \begin{cases*}
        L_{\tilde{g}}H_x=0 & \text{in $B_\delta(0)$} \\
        H_x=h_{{\sigma}_P(x)} & \text{in $\partial B_\delta(0)$;}
    \end{cases*}
\end{equation*}
as above, we have that $H_x\in W^{2,p}({B}_\delta)\cap C^{\infty}_{{loc}}(\overline{B}_\delta\backslash \{0\})$ for any $p\in[1,+\infty)$ due to the regularity of $\tilde{g}$. We also notice that $H_x(y)=H_x(-y)$ $\forall y\in B_\delta$. Indeed, $h_{{\sigma}_P(x)}$ is antipodally symmetric by construction, so we easily see that $\tilde{f}(y):=H_x(-y)$ is another solution for the same problem; however, the solution is unique.

At this point, we claim that 
\begin{equation}\label{eq:green-relations}
    G_{{\sigma}_P(x)}({\sigma}_P(y))= \tilde{G}_x(y)+\tilde{G}_{-x}(y)+ H_x(y), \quad \forall x,y\in B_\delta(0).
\end{equation}
Indeed, the above right-hand side is equivariant, so its projection through $\sigma_P$ is well-defined. The claim then follows from the uniqueness of $G_q$. 
By virtue of \eqref{eq:green-relations}, if we are able to get an expansion for $\tilde{G}_x$, $x\in B_\delta(0)\backslash \{0\}$, then we also obtain an expansion for $G_{{\sigma}_P(x)}$; this will be our next objective.

\subsection{Parametrix of the Green's function for the lifted metric}\label{subseq-parametrix}

In view of \eqref{eq:green-relations}, in order to obtain an expansion for the Green's function on $M$, it is enough to compute a local expansion of the Green's function $\tilde{G}_x$ for the lifted metric $\tilde{g}$.
Let $x\in B_{\delta/4}(0)\setminus \{0\}$ and consider geodesic normal coordinates $\{z^i\}$ centered at $x$ (notice that $\tilde{g}$ is smooth near $x$). Given $r:=\abs{z}$, we start by computing $L_{\tilde{g}}(r^{-2})$, which, by formula (2.18) in \cite{gursky-malchiodi-2015-JEMS}, is equal to
\begin{align}
\label{eq:par-remainder}
    L_{\tilde{g}}(r^{-2})=\frac{2a}{r^3}\p_r\log(\sqrt{\abs{\tilde{g}}})+\frac{R_{\tilde{g}}}{r^2};  
    \qquad 0<r <d_{\tg}(x,0)/2, 
\end{align}
(here $\abs{\tilde{g}}$ denotes the volume element of $\tilde{g}$ in normal coordinates).
 At this point, we would like to bound the above remainder in $L^{p}$ for some $p>4$: this would allow us to employ elliptic estimates in order to get a $C^1$-bound. However, this is not possible in general. 
 
\medskip 

In order to derive a better expansion of the volume element and the scalar curvature at a point $x\in B_{\delta/4}(0)\backslash \{0\}$, we need to perform a conformal change of the metric $\tilde{g}$ \acc localized near  $x$''. The basic idea is to employ \emph{conformal normal coordinates} as in \cite{Lee-Parker-87}, but in our case we further need to keep under control the global behavior of the conformal factor in order to derive a precise asymptotic expansion of the Green's function when the basepoint $x$ is approaching the origin. To this purpose, for each point $x\not=0$, 
we define  an \emph{explicit}  conformal factor $f_x$ which is given by a polynomial. 

\medskip

Let $t:=d_{\tilde{g}}(x,0)>0$ and define $\varphi_t$ to be a smooth, radial and monotone decreasing cutoff such that $\varphi_t(s)=1$ for $s\leq t/4$, $\varphi_t(s)=0$ for $s\geq t/2$ and $\abs{\varphi_{t}^{(k)}(s)}\leq C/{t^k}$ $\forall s$, for $k=1,\dots,4$ and for a suitable $C>0$.
 Following \cite[p. 158]{aubinbook}, given $\{z^i\}$ normal coordinates for $\tilde{g}$ centered at $x\not=0$, we can define a polynomial $\bar{f}=\bar{f}_x$ with the property that
 $\bar{f}(z)$ only contains terms which are quadratic and cubic in $z$. In particular, the coefficients of all quadratic terms are linear expressions of $R_{\tg}(x)$ and $R^{\tg}_{ij}(x)$ (the Ricci tensor), while the coefficients of the cubic terms are linear expressions of $\p_k R_{\tg}(x), \partial_k R^{\tg}_{ij}(x)$. The explicit formula for $\bar{f}_x$ is the following:
 \begin{align}
 \notag
     \bar{f}_x(z):=&\frac{1}{4}\sum_{i}\Big[2R^{\tilde{g}}_{ii}(x)-\frac{1}{3}R_{\tilde{g}}(x) \Big](z^i)^2 +\sum_{i<j} R^{\tilde{g}}_{ij}(x)z^iz^j+\frac{1}{6}\sum_i\Big[\p_i R^{\tg}_{ii}(x)-\frac{1}{6}\p_i R_{\tg}(x)\Big](z^i)^3 \\
     \notag
     &+\frac{1}{6}\sum_{i\not= k}\Big[\p_k R^{\tg}_{ii}(x)+2\p_i R^{\tg}_{ik}(x)-\frac{1}{6}\p_kR_{\tg}(x)\Big](z^i)^2 z^k \\
     \label{eq:conformal-factor-f^q}
     &+\frac{1}{3}\sum_{i<j<k}\big(\p_k R^{\tg}_{ij}(x)+\p_i R^{\tg}_{kj}(x)+\p_j R^{\tg}_{ik}(x)\big) z^i z^j z^k.
 \end{align}
 We can now define
\begin{equation}\label{eq:f_x-definition}
f(z)=f_x(z):=\varphi_t(\abs{z})\bar{f}_x(z),
\end{equation}
 and let 
 \begin{equation}\label{eq:gbar-definition}
 \bar{g}=\bar{g}_x:=e^{f_x}\tilde{g}.
 \end{equation}

\begin{remark}\label{rem:equiv-metric}
    Since we aim  to get an expansion for the Green's function on $M$, we would then like to project $\bar{g}$ through $\pi$ in order to obtain a conformal metric on $M$. This requires $\bar{g}$ to be antipodally symmetric, so that we actually need to symmetrize $f_x$ accordingly. However, this does not affect the next results in any way, and we can freely assume from now on $f_x$ to be antipodally symmetric.
\end{remark}
 
 By the conformal change of scalar and Ricci curvature (see \cite[146]{aubinbook}), we see that
 \begin{equation*}
     R_{\bar{g}}(x)=0, \qquad R^{\bar{g}}_{ij}(x)=0.
 \end{equation*}
 Moreover, we also have (cf. \cite[158]{aubinbook})
 \begin{equation*}
     \p_k R_{\bar{g}}(x)=0, \qquad \p_k R^{\bg}_{ij}+\p_i R^{\bg}_{jk}+\p_j R^{\bg}_{ik}=0, \qquad \forall i,j,k.
 \end{equation*}

\begin{remark}
    By looking e.g. at \cite[Proposition 2.5]{freguglia-malchiodi-2024-pre} and recalling the definition of $\tilde{g}$, we know that, for any $k\in \N$, there exists a positive constant $C(k)>0$ such that
    \begin{equation*}
        \abs{\nabla^{k}_{\tilde{g}}R_{\tilde{g}}(x)}\leq\frac{C(k)}{d_{\tilde{g}}(x,0)^k}, \qquad \abs{\nabla^{k}_{\tilde{g}}R^{\tilde{g}}_{ij}(x)}\leq\frac{C(k)}{d_{\tilde{g}}(x,0)^k}, \qquad \text{when $d_{\tilde{g}}(x,0)<\delta/4$}.
    \end{equation*}
\end{remark}

Using the above remark, the definition of $\varphi_t$ and the formulae for conformal change of scalar and Ricci curvature (cf. \cite[p.146]{aubinbook}), we see that, up to $k=4$, similar estimates also  hold for the scalar curvature and Ricci tensor of $\bar{g}=\bar{g}_x$. Moreover, we have that $d_{\tg}(x,0)$ and $d_{\bg}(x,0)$ are comparable, 
so we can use either of them in the right hand side of the estimates. Let us  denote by $\{y^i\}$ the normal geodesic coordinates for $\bar{g}$ around $x$.
For $k=1,\dots,4$ (and a different constant $C(k)$), we also got
\begin{equation*} %\label{eq:curv-decay-con}
    \abs{\nabla^{k}_{\bar{g}}R_{\bar{g}}(y)}\leq\frac{C(k)}{d_{\tg}(x,0)^k},\quad \abs{\nabla^{k}_{\bg}R^{\bar{g}}_{ij}(y)}\leq\frac{C(k)}{d_{\tg}(x,0)^k} \quad \text{when $d_{\bar{g}}(x,0)<\frac{\delta}{2}$, $\abs{y}<\frac{d_{\tg}(x,0)}{2}$}.
\end{equation*}
 As a consequence, we have the following expansions for $\abs{y}<d_{\tg}(x,0)/{2}$, (cf. \cite{aubinbook}, \cite{Lee-Parker-87}):
\begin{gather}
\label{eq:sc-curv-exp}
    R_{\bar{g}}(y)= d_{\tg}(x,0)^{-2}O^{''}\big(\abs{y}^2\big), \\
    \label{eq:vol-det-exp}
    \det(\bg):=\abs{\bar{g}}=1+ d_{\tg}(x,0)^{-2}O^{''}\big(\abs{y}^4\big),  \\
    \label{eq:gbar-metr-expansion}
    \bar{g}_{ij}(y)=\delta_{ij}-\frac{1}{3}W_{kijl}(x)y^ky^l+ d_{\tilde{g}}(x,0)^{-1}O^{''}\big(\abs{y}^3\big).
\end{gather}
Here $W$ is the Weyl tensor and $O^{(k)}\big(\abs{y}^l\big)$ denotes an error term $\Xi(y)$ such that, for a suitable constant $C>0$ (\emph{independent} of $x$), it holds $\abs{\Xi(y)}\leq C\abs{y}^l$ and $\abs{\nabla^m\Xi(y)}\leq C \abs{y}^{l-m}$ for $m=1,\dots,k$ and $\abs{y}<d_{\tg}(x,0)/{2}$. In other words, we highlighted the dependence upon $d_{\tg}(x,0)$ of the higher order error terms in the expansions in the fixed ball of center $x$ and radius $d_{\tg}(x,0)/2$. This will be crucial for the next estimates.

We are now able to estimate the conformal Laplacian (in normal coordinates for $\bar{g}$ at $x$) of the singular term $\abs{y}^{-2}$. We recall that $t:=d_{\tg}(x,0)$; by \eqref{eq:par-remainder}, \eqref{eq:sc-curv-exp} and \eqref{eq:vol-det-exp}, we got
\begin{equation}\label{eq:clapest-1}
   \big\vert L_{\bg}(\abs{y}^{-2})\big\vert\leq C t^{-2},
\end{equation}
for a suitable constant $C>0$ (independent of $x$) and for $\abs{y}<t/2$. 

\medskip

We aim next  to obtain an $L^p$ estimate for $L_{\bg}(\abs{y}^{-2})$ in the whole domain $B_\delta(0)$. However, since the metric $\tilde{g}$ is only $C^{1,1}$ at the origin, we need to slightly modify the \acc test function'' employed (as $d_{\bg}(x,\cdot)$ is not smooth outside $x$). Nevertheless, we are able to obtain the following estimate for the Green's function:
\begin{lemma}\label{lem:Green-exp-dir}
    For every $x\in B_{\delta/4}(0)\backslash\{0\}$, let $\bg=\bg_x$ be the metric defined in \eqref{eq:gbar-definition} and denote by $\{y^i\}$ the normal coordinates for $\bg$ centered at $x$ (which are defined on $B_{\delta/2}(0)$). Let $\bar{G}_x$ be the Green's function for $L_{\bg}$ defined as in \eqref{eq:dir-green-lift}. Then the following expansion holds: 
    \begin{equation}\label{eq:Green-lift-exp}
        \bar{G}_x(y)=\frac{1}{\abs{y}^2}+\phi_x(y), \qquad \forall\, \abs{y}<t/2,
    \end{equation}
    where $\phi_x\in C^1(B_{t/2}^{\bg}(x))$ and, for any $\mu>0$, $\norm{\phi_x}_{C^0(B^{\bg}_{t/2}(x))}\leq Ct^{-\mu}$ and $\norm{\nabla_{\bg} \phi_x}_{C^0(B^{\bg}_{t/2}(x))}\leq C t^{-1-\mu}$, for a positive constant $C>0$ which depends on $\mu$ and $\delta$, but not on $x$.
\end{lemma}

\begin{proof}
    We start by recalling that, by \eqref{eq:lift-metr-exp}, the standard Euclidean coordinates $\{z^i\}$ for $B_\delta(0)\subset \R^4$ centered at $0$ are also normal coordinates for $\tilde{g}$, so one has 
    \begin{equation}\label{eq:metr-eucl-exp-lem-1}
        \tilde{g}_{ij}(z)=\delta_{ij}+O^{''}(\abs{z}^2),
    \end{equation}
    for any $z\in B_\delta(0)$ (the distance function $d_{\tg}(\cdot,0)$ is smooth outside $0$ for $\delta>0$ small enough). Moreover, the same expansion also holds for $\bg_x$ for all $ x\in B_{\delta/4}(0)$ with uniformly bounded error terms (i.e. they do not depend on $x$). Let now $\bar{y}^i$ be \emph{normal coordinates for the Euclidean metric} $g_E$ centered at $x$. Clearly $\abs{\bar{y}}$ is smooth outside $x$; in particular, it is smooth at the origin $0\in B_\delta$. Since the $\{\bar{y}^i\}$'s are exactly a translation of the coordinates $\{z^i\}$ by a fixed vector of length $\simeq t$ (w.r.t. $\bg$), we see that the expansion \eqref{eq:metr-eucl-exp-lem-1} also holds for $\bg$ in $\bar{y}$-coordinates as soon as we are at range $\gtrsim t$. Hence
    \begin{equation}\label{eq:metr-eucl-exp-lem-2}
        \bar{g}_{ij}(\bar{y})=\delta_{ij}+O^{''}(\abs{\bar{y}}^2), \qquad \text{for $\abs{\bar{y}}\geq t/8$.}
    \end{equation}
    
    Let $\chi\colon\R^+\to[0,1]$ be a smooth cutoff function such that $\chi(s)=0$ for $s\leq t/8$, $\chi(s)=1$ for $s\geq t/4$ and $\abs{\chi^{(k)}(s)}\leq C t^{-k}$ for $k=1,2$. Define now the following \acc test function'' $\zeta\colon B_\delta(0)\to\R$ as
    \begin{equation*}
        \zeta(p):=d_{\bar{g}}(x,p)^{-2} \big(1-\chi(d_{\bg}(x,p))\big)+d_{g_E}(x,p)^{-2}\chi(d_{\bg}(x,p)).
    \end{equation*}
    By definition, $\zeta$ is $C^2$ outside the point $x$, so that $L_{\bg}(\zeta)$ will be continuous in $B_\delta\backslash\{x\}$.
    We next estimate $L_{\bg}(\zeta)(p)$ in different regions of $B_\delta(0)$.

\

    \textbullet \, Assume $d_{\bg}(x,p)< t/8$: then it follows by \eqref{eq:clapest-1} that
    \begin{align}\label{eq:lp-bound-1}
    \norm{L_{\bg} (\zeta)}_{L^p(B^{\bg}_{t/8}(x))}^p\leq C t^{4-2p}, \qquad \forall p>2.
\end{align}

\textbullet \, Assume $t/8< d_{\bg}(x,p)<t/4$: let us call, with a little abuse of notation, $\abs{y}=d_{\bg}(x,p)$ and $\abs{\bar{y}}=d_{g_E}(x,p)$. Then
\begin{align}
\notag
   \Delta_{\bg}(\zeta)(p)=&\Delta_{\bg}\big((1-\chi({\abs{y}}))\abs{y}^{-2}+\chi({\abs{y}})\abs{\bar{y}}^{-2}\big) \\[0.5ex]
   \notag
   =& - (\Delta_{\bg}\chi(\abs{y}))\abs{y}^{-2} +   (1-\chi({\abs{y}})) \Delta_{\bg}(\abs{y}^{-2})-2 \nabla_{\bg} \chi(\abs{y})\cdot \nabla_{\bg}(\abs{y}^{-2}) \\[0.5ex]
   \notag
   &+(\Delta_{\bg}\chi(\abs{y}))\abs{\bar{y}}^{-2}+\chi(\abs{y}) \Delta_{\bg}(\abs{\bar{y}}^{-2})+2\nabla_{\bg} \chi(\abs{y})\cdot \nabla_{\bg}(\abs{\bar{y}}^{-2}) \\[0.5ex]
   \notag
   =&\Delta_{\bg}\chi(\abs{y})\big[\abs{\bar{y}}^{-2}-\abs{y}^{-2}\big]+(1-\chi({\abs{y}})) \Delta_{\bg}(\abs{y}^{-2})+\chi(\abs{y}) \Delta_{\bg}(\abs{\bar{y}}^{-2}) \\[0.5ex]
   \label{eq:lap-conf-lem}
   &+2 \nabla_{\bg}\chi(\abs{y})\cdot \big[\nabla_{\bg}(\abs{\bar{y}}^{-2})-\nabla_{\bg} (\abs{y}^{-2})\big].
\end{align}
We already know that $\Delta_{\bg}(\abs{y}^{-2})=O(t^{-2})$, and clearly $\Delta_{\bg}\chi(\abs{y})=O(t^{-2})$ by definition of $\chi$. By virtue of \eqref{eq:metr-eucl-exp-lem-2}, in $\bar{y}$-coordinates one has $\sqrt{\abs{\bar{g}}}=1+O^{''}(\abs{\bar{y}}^2)$, so it follows from \eqref{eq:par-remainder} that $\Delta_{\bg}(\abs{\bar{y}}^{-2})=O(\abs{\bar{y}}^{-2})=O(t^{-2})$ in this region. Again by \eqref{eq:metr-eucl-exp-lem-2} (we are at scale $t$), one has $d_{\bg}(x,p)=d_{g_E}(x,p)+O^{''}(d_{g_E}(x,p)^3)$; as a consequence,
\begin{equation*}
    \abs{y}^{-2}=(\abs{\bar{y}}+O(t^3))^{-2}=\abs{\bar{y}}^{-2}+O(1),
\end{equation*}
and
\begin{equation*}
    \nabla_{\bg}(\abs{y}^{-2})=\nabla_{\bg}(\abs{\bar{y}}^{-2}+O^{'}(1))=\nabla_{\bg}(\abs{\bar{y}}^{-2})+O(t^{-1}).
\end{equation*}
Using all these relations inside \eqref{eq:lap-conf-lem}, we find that
\begin{equation*}
    \Delta_{\bg}(\zeta)(p)=O(t^{-2})
\end{equation*}
in the desired region, therefore
\begin{equation}\label{eq:lp-bound-2}
    \norm{L_{\bg} (\zeta)}_{L^p(B^{\bg}_{t/4}(x)\backslash B^{\bg}_{t/8}(x))}^p\leq C t^{4-2p}, \qquad \forall p>2.
\end{equation}

\textbullet \, Finally, again by \eqref{eq:metr-eucl-exp-lem-2}, we see that $L_{\bg}(\zeta)(p)=L_{\bg}(\abs{\bar{y}}^{-2})=O(\abs{\bar{y}}^{-2})$ in the complementary region $\{ p\in B_\delta(0)\mid d_{\bg}(x,p)>t/4\}$, which implies 
\begin{equation}\label{eq:lp-bound-3}
    \norm{L_{\bg} (\zeta)}_{L^p(B_\delta(0)\backslash B^{\bg}_{t/4}(x))}^p\leq C t^{4-2p}, \qquad \forall p>2.
\end{equation}

Putting together \eqref{eq:lp-bound-1}, \eqref{eq:lp-bound-2}, \eqref{eq:lp-bound-3}, we see that 
\begin{equation}\label{eq:lp-bound-fin}
\norm{L_{\bg} (\zeta)}_{L^p(B_\delta(0))}\leq C t^{\frac{4-2p}{p}}, \qquad \forall p>2.
\end{equation}
Let now $\phi_x$ be the solution to 
\begin{equation*}
    \begin{cases*}
        L_{\bg} (\phi_x) =-L_{\bg}(\zeta) & \text{in $B_\delta(0)$,} \\
        \phi_x=-\zeta & \text{in $\p B_\delta(0)$.}
    \end{cases*}
\end{equation*}
By construction, the Green's function for $L_{\bg}$ is given by $\bar{G}_x=\zeta+\phi_x$. 
Notice that $\norm{\zeta}_{W^{2,p}(\partial B_\delta)}\leq C=C(p,\delta)$, $\forall \, p>2$. By virtue of \eqref{eq:lp-bound-fin}, we can now use elliptic estimates (for $p=2+\mu$ and $p=4+\mu$ and $\mu$ small) together with the Sobolev embedding theorem in order to easily recover the expansion \eqref{eq:Green-lift-exp}. In particular, even if $L_{\bg}$ depends upon $x$, its coefficients are, by construction, uniformly bounded in $C^0$ and uniformly elliptic $\forall\, x\in B_{\delta/4}(0)\backslash\{0\}$; as a consequence, we have elliptic estimates with constants depending upon $p$, $\delta$, but not on $x$, see Theorem 9.13 and Lemma 9.17 in \cite{Gilbarg-Trudinger}. This concludes the proof. 
\end{proof}

\begin{remark}\label{rem:green-antip-exp}
    The above proof also shows that, if we still denote by $\{{y}^i\}$  normal coordinates for $\bg$ centered at $x$, then, letting $b>1$, one has
    \begin{equation*}
    \bar{G}_{-x}(y)=\frac{1}{4t^2}+\alpha_x(y), \qquad \forall \, \abs{y}<t^b,
    \end{equation*}
    where $\alpha_x$ is a $C^1$ function satisfying $\abs{\alpha_x(y)}\leq C t^{b-3}$ and $\abs{\nabla_{\bg}\alpha_x(y)}\leq C t^{-3}$, for a suitable $C=C(b)>0$ (independent of $x$) and for all $\abs{y}<t^b$.
\end{remark}

\subsection{Expansion of the Green's function near a conical point}\label{subsec:green-exp-manif}

We now employ the previous results in order to recover an expansion for the Green's function of a suitable conformal metric near a conical point $P\in M$. Let $q\in B^g_{\delta/4}(P)$ and call as before $\sigma_P^{-1}(q)=\{x,-x\}$, $\tilde{g}=\sigma_P^*g$ ($\sigma_P$ is defined in \eqref{eq:prj-def}), $t=d_g(q,P)=d_{\tilde{g}}(x,0)$. As explained in Section \ref{subseq-parametrix}, we can pass from $\tilde{g}$ to an equivariant (cf. Remark \ref{rem:equiv-metric}) conformal metric $\bar{g}=\bar{g}_x$ in $B_\delta(0)$, whose Green's function $\bar{G}_x$ satisfies the expansion \eqref{eq:Green-lift-exp} in normal coordinates for $\bar{g}$ around $x$. Since $\bar{g}=\tilde{g}$ in $B_\delta\backslash B_{\delta/2}$, we can push $\bar{g}$ down on $M$ via $\sigma_P$, defining a new metric $g^q$ which is conformal to $g$. In particular, one has
\begin{equation*}
    g^q=e^{f^q}g, \qquad f^q:=f_x\circ \sigma_P^{-1},
\end{equation*}
(note that $f^q$ is well-defined, being $f_x$  antipodally symmetric).
We can now employ \eqref{eq:green-relations}, Lemma \ref{lem:Green-exp-dir} and Remark \ref{rem:green-antip-exp} to recover the following expansion for the Green's function on $M$ associated to $L_{g^q}$:
\begin{lemma}\label{lem:green-exp-manifold}
    For any $q\in B^g_{\delta/4}(P)$, let $g^q$ be defined as above and consider $G_q$ the Green's function for $L_{g^q}$ with pole at $q$. Then, $\forall \, b>1$, the following expansion holds in normal coordinates $\{z^i\}$ for $g^q$ centered at $q$:
    \begin{equation}\label{eq:green-manif-expansion}
 G_q(z)=\frac{1}{\abs{z}^2}+A_q+\beta_q(z), \qquad \forall \, \abs{z}\leq t^b,
    \end{equation}
    where $A_q=1/(4t^2)+O(t^{b-3})$ and $\beta_q$ is a $C^1$ function satisfying
    \begin{equation*}
        \beta_q(0)=0, \quad\abs{\beta_q(z)}\leq C t^{b-3}, \quad \abs{\nabla_{g^q}\beta_q}\leq C t^{-3}, \qquad \forall \, \abs{z}\leq t^b,
    \end{equation*}
for a suitable $C=C(b)>0$ independent of $q$. In particular, the mass of the asymptotically-flat manifold $(M,G_q^2 \, g)$ diverges inverse-quadratically with respect to $d_g(q,P)$.
\end{lemma}

\begin{proof}
    The proof follows immediatly from \eqref{eq:green-relations}, Lemma \ref{lem:Green-exp-dir} and Remark \ref{rem:green-antip-exp}, noticing that $H_x$ in \eqref{eq:green-relations} is uniformly bounded in $C^1(B^{g}_{\delta/2})$ by a constant which does not depend on $q\in B^g_{\delta/4}(P)$.
\end{proof}

%% file: variational_argument.tex
\section{The variational argument}\label{sec:variational-argument}
This section will provide a detailed explanation of the variational argument at the base of the proof of Theorem \ref{t:main}.
As already mentioned in the introduction, \cite[Theorem 1.1]{freguglia-malchiodi-2024-pre} allows, in our setting, to find a Yamabe metric (minimizer) whenever  \eqref{s:xi-conf} is satisfied for at least one conical point $P$. More generally, we have existence of a Yamabe metric anytime $\Y(M,[g])<\Y_S$ (cf. \cite{ACM}) and, even if $\Y(M,[g])=\Y_S$, the minimum could sometimes be attained, as for the case of $\Sp^4/\Z_2$: a {\em football} with two antipodal conical points. As a consequence of this, from now on we will make 
the following: 

\begin{assumption}
\eqref{s:xi-conf} $\forall \,P$ does not hold  and $\Y(M,[g])=\Y_S=\frac{1}{\sqrt{2}}\Y_4$ is \underline{not attained}.
\end{assumption}

The fact that \eqref{s:xi-conf} does not hold implies that, by Corollary \ref{cor:c1extension}, after a conformal change \underline{$g$ admits a $\Z_2$-lift to a $C^{1,1}$ metric} in a neighborhood 
of $0 \in \R^4$ near each conical point. 

\smallskip 

In this particular situation, we want to show that it is possible to employ a \emph{mountain pass} scheme in order to find a solution for \eqref{eq:Y}. Let $P_1,P_2$ denote any two conical points of $M$, and consider, for a small $\delta>0$, a smooth cutoff function $\chi_\delta(r)$ such that:
\begin{equation}\label{eq:chi-delta}
    \begin{cases*}
        \chi_\delta(r)=1 & \text{for $r<\delta$,} \\
        \chi_\delta(r)=0 & \text{for $r\geq 2\delta$,} \\
        \abs{\nabla^{(k)}_g \chi_\delta(r)}\leq C\delta^{-k} & \text{$\forall r>0$ and $k=1,2.$}
    \end{cases*}
\end{equation}
Letting $s$ be the geodesic distance from $P_i$,  define a test function $\phi_{\epsilon,P_i}$ (highly concentrated at $P_i$) by
\begin{equation*} %\label{eq:mountain-pass-extremal-function}
    \phi_{\epsilon,P_i}(s):=
    \begin{cases*}
        U_\epsilon(s)\chi_\delta(s) & \text{for $s\leq 2\delta$}, \\
        0 & \text{anywhere else}.
    \end{cases*}
\end{equation*}
From this definition and the preceding discussion, it is clear that, for $i=1,2$, the quotient $Q_g(\phi_{\epsilon,P_i})$ approaches $\mathcal{Y}_S$ from above as $\epsilon\to0^+$.

Given a small and fixed $\epsilon\ll 1$, let
\begin{equation*}%\label{eq:mountain-pass-paths}
    \Pi:=\left\{ \gamma\in C([0,1];W^{1,2}(M)) \mid \gamma(0)=\phi_{\epsilon,P_1}, \gamma(1)=\phi_{\epsilon,P_2}\right\}
\end{equation*}
and define the min-max value $c$ as 
\begin{equation*}%\label{eq:min-max-value}
    c:=\inf_{\gamma\in \Pi}\max_{\tau\in [0,1]} Q_g(\gamma(\tau)).
\end{equation*}
Clearly both $\Pi$ and $c$ depend on $\epsilon$.
An adaptation of the well-known \emph{concentration-compactness} principle to our singular context allows us to prove the following (cf. \cite[Lemma 2.3]{Bianchi-96}):

\begin{lemma}\label{lemma:concentration-compactness}
   It holds $c> \max\{Q_g(\phi_{\epsilon,P_1}), Q_g(\phi_{\epsilon,P_2})\} >\mathcal{\Y}_4/{\sqrt{2}}$ provided $\epsilon$ is small enough.
\end{lemma}

\begin{proof} 
    We assume $\eps = \eps_k \searrow 0$:  
first fix  small constants $\hat{\eps}, \delta > 0$, and we will take then $\eps_k \ll 
    \hat{\eps} \ll \delta$. 
Consider now an admissible curve $\gamma_k \in \Pi$. Since the Yamabe energy is invariant by dilation, by properly scaling the functions 
$\phi_{\epsilon,P_1}, \phi_{\epsilon,P_2}$, 
we can assume without loss of generality that $\|\gamma_k(t)\|_{L^4(M)} = 1$ 
for all $t \in [0,1]$. Fixing a small numer $r > 0$, we distinguish  two cases. 

\

\noindent {\bf Case 1.} {\em For all $t \in [0,1]$ there exists $q_t \in M$ such that $\int_{B_{r}(q_t)} |\gamma_k(t)|^{4} d\mu_g \geq 1 - \hat{\eps}$}. 

Notice that $q_t$ is not unique, but all such $q_t$ must be contained 
in a ball of size $4 r$. Since $\gamma_k$ is a continuous curve in $W^{1,2}$, and hence in 
$L^4$, there must be a value $t_k$ of $t$ such that  
$B_{\frac{1}{8} d(P_1, P_2)} (q_{t_k}) $ consists of regular points.

Assuming that the statement is false, 
letting $\eps_k \searrow 0$ we would have that $\gamma_k(t_k)$ is a minimizing sequence 
for $Q_{{g}}$, so $Q_{{g}}(\gamma_k(t_k)) \to \frac{\sqrt{2}}{2} \Y_4$. By Ekeland's variational principle, see e.g. 
\cite{St-book}, there exists a minimizing  Palais-Smale sequence $\tilde{\gamma}_k$ such that $\|\tilde{\gamma}_k - \gamma_k(t_k)
\|_{W^{1,2}(M)} \to 0$, so also $Q_{{g}}(\tilde{\gamma}_k) \to \frac{\sqrt{2}}{2} \Y_4$.
Since we are assuming the infimum of $Q_{{g}}$, equal to 
$\frac{\sqrt{2}}{2} \Y_4$, is not attained, by the result in \cite{St-84}, which can be rather easily adapted to the 
present situation, $\tilde{\gamma}_k$ must develop 
a finite number of bubbles. Since also $q_{t_k}$ is 
not approaching any singular point and since the local 
Yamabe constant of smooth points coincides with $\Y_4$, 
we must have that 
\[
\lim_k Q_{{g}}(\tilde{\gamma}_k) \geq \Y_4, 
\]
against the contradiction assumption.

\

\noindent {\bf Case 2.} {\em There exists $t_k \in [0,1]$ such that for all $q$  $\int_{B_{r}(q)} |\gamma_k(t_k)|^{2^*} d\mu_g < 1 - \hat{\eps}$}. 

We let again $\eps_k \searrow 0$, and still assume by contradiction that the statement is false. Then we can find a minimizing 
Palais-Smale sequence $\tilde{\gamma}_k$ as before 
for $Q_{{g}}$ such that $Q_{{g}}(\tilde{\gamma}_k) \to 
\frac{\sqrt{2}}{2} \Y_4$. 

Since we are in Case 2 (and we are assuming that the Yamabe quotient is not attained), the sequence $\tilde{\gamma}_k$  must 
develop at least two bubbles. More precisely $\tilde{\gamma}_k$ can be 
written as 
\begin{equation}\label{eq:PS}
    \tilde{\gamma}_k = \sum_{i=1}^{j_1} u_{p_{i,k},\eps_{i,k}} + \sum_{l=1}^{j_2} v_{p_{l,k},\eps_{l,k}} + o_k(1),
\end{equation}
where the sequences $(p_{i,k})_k$ are converging 
to some singular points, $(p_{l,k})_k$ to 
some points in $M$ (regular or singular), $\eps_{i,k}, \eps_{l,k} \to 0$ and where 
the $u_{i,k}$'s and the $v_{l,k}$'s have profiles of singular bubbles and regular bubbles respectively. Since for both of these we have that 
\[
\lim_k Q_{{g}}(u_{p_{i,k},\eps_{i,k}}) \geq 
\frac{\sqrt{2}}{2} \Y_4; \qquad 
\lim_k Q_{{g}}(v_{p_{l,k},\eps_{i,k}}) \geq 
\frac{\sqrt{2}}{2} \Y_4. 
\]
by the fact that $j_1 + j_2 \geq 2$ and standard concavity arguments, see for example Chapter III in \cite{St-book}, we have again 
\[
\lim_k Q_{{g}}(\tilde{\gamma}_k) \geq \Y_4, 
\]
which is still a contradiction. This concludes the proof. 
\end{proof}

\

As an immediate consequence of Lemma \ref{lemma:concentration-compactness} and a 
standard deformation argument,  we can find a Palais-Smale sequence $(u_n)_n$ for $Q_g$ at level $c$.  As we will see at the end of this section, proving $c<\Y_4$ would rule out all possible blow-up scenarios for such a Palais-Smale sequence. 
In order to show that $c<\Y_4$, we will exhibit a competitor $\bar{\gamma}\in\Pi$ such that $\max_{\tau\in[0,1]} Q_g(\bar{\gamma}(\tau))<\Y_4$; this is the most delicate and technically challenging part of the proof of Theorem \ref{t:main}. 

\subsection*{Construction of the competitor}

\noindent
The competitor $\bar{\gamma}$ we are going to exhibit will be, roughly speaking, a continuous path of highly concentrated bubbles with centers along a suitable curve $\hat{\gamma}$ supported on $M$ and connecting the two conical points $P_1$ and $P_2$. In order to keep the value of the Yamabe quotient of each bubble below $\Y_4$, we must pay attention to the error terms of its expansion.

There are two different regimes, each requiring different competitors:
\begin{itemize}
    \item when we are close to the conical points $P_1$ and $P_2$, a good competitor is represented by the image of a ``double bubble'', that is, the image on the manifold of the projection (w.r.t. ${\sigma}_{P_i}$ defined as in \eqref{eq:prj-def}) of a sum of two antipodal bubbles which are centered at points $\pm t\nu$ respectively, for some $\nu\in \Sp^3$ and with scaling parameter $\epsilon$;
    \item when we are \acc far'' from the conical points, a good competitor is represented by a bubble suitably glued to a multiple of the Green's function for the conformal Laplacian, as  done by Schoen for the resolution of Yamabe problem in low dimension and on LCF manifolds, cf. \cite{Schoen-84}, \cite{Lee-Parker-87}.
\end{itemize}
However, in order to pass from one type of competitor to the other in a continuous way, we also need to show that there exists an \emph{intermediate regime} in which it is possible to interpolate in between while keeping the Yamabe quotient below $\Y_4$.

As we will show, this is indeed possible by virtue of the fact that, near the conical points, the error terms in the expansion of the Yamabe quotient are of the \emph{same order} (and uniformly negative) for the two different competitors.

\subsection{Competitors near the conical points}
To begin, we define suitable test functions near the conical points.

On a small ball $B_{2\delta}(0)\subseteq\R^4$, we define $\tilde{u}_{\epsilon,t\nu}(y):=\widehat{U}_{\eps,t\nu}(y)\chi_\delta(\abs{y})$, where $\widehat{U}_{\eps,t\nu}$ is the sum of two bubbles defined in \eqref{eq:double-bubble} and $\chi_\delta$ is the cutoff function defined in \eqref{eq:chi-delta}.

\medskip
 
\noindent \textbf{Notation.}
    Since all the computations involving integrals of $\widehat{U}_{\eps,t\nu}$ over balls are invariant with respect to $\nu\in\Sp^3$, for the sake of simplicity, we will often assume $\nu=\mathbf{e}_1$ and write $\tilde{u}_{\epsilon,t}$ in place of $\tilde{u}_{\epsilon,t\nu}$.

\medskip 

Given a conical point $P$ and ${\sigma}_P:B^{\tilde{g}}_{2\delta}(0)\to  B^g_{2\delta}(P)$ the projection defined in \eqref{eq:prj-def} (recall that $\tilde{g}$ is the lifted metric), we define a test function $u_{\epsilon,t}$ on $M$ by 
    \begin{equation}\label{eq:doubl-bub-test-M}
    u_{\epsilon,t}^P(q)=u_{\epsilon,t}(q):=
        \begin{cases*}
        \tilde{u}_{\epsilon,t}({\sigma}_P^{-1}(q)) & \text{if $d_g(P,q)<2\delta,$} \\
        0 & \text{if $d_g(P,q)\geq 2\delta$,}
        \end{cases*}
    \end{equation}
    where it is understood that, by $\tilde{u}_{\epsilon,t}({\sigma}_P^{-1}(q))$, we mean the value of $\tilde{u}_{\epsilon,t}$ in  any  of the two points in the pre-image of $q$. This function is \emph{well-defined} since $\tilde{u}_{\epsilon,t}$ is antipodally symmetric by construction. 

    \begin{remark}
        When $t=0$, one obtains $u_{\epsilon,0}=2\phi_{\epsilon,P}$, therefore $Q_g(u_{\epsilon,0})=Q_g(\phi_{\epsilon,P})$; moreover, the map $t\to Q_g(u_{\epsilon,t})$ is continuous from $[0, \delta/2)$ to $W^{1,2}(M)$ for any fixed $\epsilon\ll1$. 
    \end{remark}

    Regarding the expansion of Yamabe's quotient at $u_{\epsilon,t}$, we have the following result, whose proof will be given in Section \ref{sec:doub-bubble}:
    \begin{proposition}\label{prop:exp.doub-bub}
        Let $u_{\epsilon,t}$ be defined as in \eqref{eq:doubl-bub-test-M}. For any $\alpha\in\big(1/2,1\big)$ and $t=\epsilon^\alpha$, the following expansion holds:
        \begin{equation*}%\label{eq:dbub-y-exp}
            Q_g(u_{\epsilon,t})=6\mathcal{S}_4-A\epsilon^{2(1-\alpha)}+o\big(\epsilon^{2(1-\alpha)}\big), \qquad \text{as $\epsilon\to0$},
        \end{equation*}
        where $A>0$ is a positive constant (explicitly given by \eqref{eq:A-formula}).
    \end{proposition}

    \begin{remark}
        As it will be clear from the proof, the particular choice of $t$ as function of $\epsilon$ is necessary and it is motivated by the fact that the higher order terms in the metric expansion \eqref{eq:lift-metr-exp} generate errors which become too big for larger $t$.
    \end{remark}

\subsection{Competitors far from the conical points}
When we are farther from conical points, the problem behaves similarly to the ``regular'' Yamabe problem, therefore we can follow the same strategy outlined in \cite{Schoen-84}, \cite{Lee-Parker-87}. 

By the arguments in Sections \ref{subseq-parametrix} and \ref{subsec:green-exp-manif}, we know that, for any point $q$ of $M$ at distance $t\leq\delta/4$ from a conical point $P$, there exists a conformal metric $g^q$ (which agrees with $g$ outside a small ball centered at $q$) such that, if $G_q$ denotes the Green's function of the conformal Laplacian $L_{g^q}$ centered at $q$ (and normalized in such a way that $L_{g^q} G_q=4\pi^2 a\delta_q$, see again Section \ref{sec:greenfunct} for the existence of $G_q$), then, in normal coordinates for $g^q$ around $q$ and $\forall \, b>1$, one has the following expansion: 
\begin{align*}
    G_q(z)=\frac{1}{\abs{z}^2}+A_q+\beta_q(z), \qquad \forall \, \abs{z}\leq t^b,
\end{align*}
where $t=d_g(q,P)$ and $\beta_q$ is a suitable $C^1$ error term, see Lemma \ref{lem:green-exp-manifold} for the details. Let us call $\tau=t^\gamma$ for some $\gamma>b$ (to be specified soon), and let $\chi_\tau$ be a smooth cutoff defined by formula \eqref{eq:chi-delta} with $\tau$ in place of $\delta$. We define a test function $w_{q,\epsilon,\tau}$ as follows:
\begin{equation}\label{eq:sch-bub-test-M}
    \begin{cases*}
       w_{q,\epsilon,\tau}(z)= U_\epsilon(z) & \text{for $\abs{z}\leq\tau$}, \\
       w_{q,\epsilon,\tau}(z)= \frac{1}{\nu}\big( G_q(z)-\chi_{\tau}(z)\beta_q(z)\big) & \text{for $\tau<\abs{z}\leq 2\tau$}, \\
       w_{q,\epsilon,\tau}= \frac{1}{\nu}G_q & \text{in $M\backslash B^{g^q}_{2\tau}(q)$},
    \end{cases*}
\end{equation}
where $\nu\in\R$ is chosen in such a way to make the match at $\abs{z}=\tau$ continuous, that is, we ask 
\begin{equation}\label{eq:cont-match}
    \frac{c_4\epsilon^{-1}}{1+\epsilon^{-2}\tau^2}=\frac{1}{\nu}\Big(\frac{1}{\tau^2}+A_q\Big).
\end{equation}

Regarding the expansion of $Q_{g^q}$ at $w_{q,\epsilon,\tau}$  we have the following result:
\begin{proposition}\label{prop:exp-sch-bub-collapse}
    For $t:=d_{g}(P,q)$ and $\tau=\tau(t):=t^{\omega/\alpha}$, let $w_{q,\epsilon}$ be defined as in \eqref{eq:sch-bub-test-M}. Assume that $\omega$ and $\alpha$ satisfy
    \begin{equation}\label{eq:alphbar-assumpt}
        1>\omega>\alpha>1/2 \quad \text{and} \quad 2+2\alpha-4\omega>0.
    \end{equation}
    Then there exists $0<\bar{\epsilon}\ll\bar{\delta}$ such that, if $\delta\leq \bar{\delta}$, $\epsilon\leq\bar{\epsilon}$, one has
    \begin{equation*}
        Q_{g^q}(w_{q,\epsilon})<6\mathcal{S}_4, \qquad  \forall t\in [\epsilon^\alpha,\delta/4].
    \end{equation*}
    Moreover, the following expansion holds when $t=\epsilon^{\alpha}$:
    \begin{equation*}%\label{eq:schbub-y-exp}
        Q_{g^q}(w_{q,\epsilon})=6\mathcal{S}_4-A\epsilon^{2(1-\alpha)}+o\big(\epsilon^{2(1-\alpha)}\big), \qquad \text{as $\epsilon\to0$},
    \end{equation*}
    where $A>0$ is given by \eqref{eq:A-formula}.
\end{proposition}

The proof of Proposition \ref{prop:exp-sch-bub-collapse} will be given in Section \ref{sec:greenbubl-expansion}. As explained there,  assumption \eqref{eq:alphbar-assumpt} is necessary in order to control some error terms.

\medskip

More generally, when $q$ is at distance $\geq \delta/4$ from all conical points of $(M,g)$, we can still define a test function as in \eqref{eq:sch-bub-test-M} with $\tau=(\delta/4)^{\omega/\alpha}$. In this case, the expansion of Yamabe quotient is the same as the one obtained on a smooth manifold in \cite{Schoen-84}, \cite{Lee-Parker-87}, and is given by
\begin{equation}\label{eq:green-test-standard-expansion}
    Q_{g^q}(w_{q,\epsilon})=6\mathcal{S}_4- A_q \epsilon^2+\epsilon^2 o_\delta(1) +o(\epsilon^2),
\end{equation}
where $A_q$ is a positive multiple of the ADM mass of the associated scalar-flat, asymptotically flat (AF) orbifold $(M\backslash \{q\},G_q^2g)$. In this setting, we can use the recent positive mass theorem of Dai-Sun-Wang \cite[Theorem 1.1]{Dai-Sun-Wang-24-preprint} for AF manifolds with conical singularities in order to deduce that $Q_{g^q}(w_{q,\epsilon})<6\mathcal{S}_4$ for any such choice of $q$.

\subsection{Interpolating in the middle}

In Propositions \ref{prop:exp.doub-bub} and \ref{prop:exp-sch-bub-collapse}, we obtained an expansion for the Yamabe quotient of the test functions $u_{\epsilon,t}$ and $w_{q,\epsilon}$ which are \acc centered'' at distance $t=\epsilon^\alpha$ from a conical point 
 $P\in M$. Assuming $q={\sigma}_P(t\nu)$ (so that both test functions are also centered at the same point), we now  want to interpolate between $u_{\epsilon,t}$ and $w_{q,\epsilon}$ while keeping the value of the Yamabe quotient strictly below the critical level $\Y_4=6\mathcal{S}_4$.
 This is ensured by the following result:
\begin{proposition}\label{prop:interp-expansion}
    Let $u_{\epsilon,t}$ and $w_{q,\epsilon}$ be given as in \eqref{eq:doubl-bub-test-M} and \eqref{eq:sch-bub-test-M} respectively, and assume that $q={\sigma}_P(t\nu)$ (${\sigma}_P$ is defined in \eqref{eq:prj-def}), $t=\epsilon^\alpha$ and $\tau=\epsilon^{\omega}$, with $\alpha,\omega$ satisfying \eqref{eq:alphbar-assumpt}. For $\lambda\in [0,1]$, let
    \begin{equation}\label{eq:psi-lambd}
        \psi_\lambda:=\lambda w_{q,\epsilon}+(1-\lambda) e^{-\frac{f^q}{2}}u_{\epsilon,t}, \qquad \varphi_\lambda:=e^{\frac{f^q}{2}}\psi_\lambda=\lambda e^{\frac{f^q}{2}}w_{q,\epsilon}+(1-\lambda)u_{\epsilon,t},
    \end{equation}
    where $g^q:=e^{f^q}g$ is defined in Section \ref{subsec:green-exp-manif}. Then
    \begin{equation}\label{eq:interp-expansion}
        Q_g(\varphi_\lambda)=Q_{g^q}(\psi_\lambda)=6\mathcal{S}_4-A\epsilon^{2(1-\alpha)}+o\big(\epsilon^{2(1-\alpha)}\big), \qquad \text{as $\epsilon\to0$},
    \end{equation}
    where $A>0$ is given by \eqref{eq:A-formula}.
\end{proposition}

\begin{remark}
    By the conformal covariance of Yamabe's quotient, we have $Q_g(h)=Q_{g^q}\big(e^{-\frac{f^q}{2}}h\big)$ and, vice-versa, $Q_{g^q}(\tilde{h})=Q_g\big(e^{\frac{f^q}{2}}\tilde{h}\big)$. This motivates the $e^{\pm\frac{f^q}{2}}$-term in the definition of $\psi_\lambda$ above, as well as the first equality in \eqref{eq:interp-expansion}. 
\end{remark}

\subsection{Construction of the competitor}
We are now ready to explicitly define our competitor $\bar{\gamma}$. To begin, consider two distinct singular points, say $P_1, P_2$, having 
closest distance: without loss of generality we can assume 
that $d_g(P_1, P_2) = 1$. Consider next a geodesic $\hat{\gamma}\colon[0,1]\to M$ joining $P_1$ and $P_2$ and a 
 small parameter $\delta>0$. We have then the following properties, by the triangular inequality: 
\begin{equation*}
    \begin{cases*}
        \hat{\gamma}(0)=P_1, \quad \hat{\gamma}(1)=P_2, \\
        \hat{\gamma}\big( (0,1)\big)\cap S=\emptyset, \\
        B^g_{4\delta}\big(\hat{\gamma}(s)\big)\cap\{P_3,\dots,P_l\}=\emptyset & \text{$\forall s\in[0,1]$}, \\
        \hat{\gamma}(s)={\sigma}_{P_1}(s{\nu}_1) & \text{for $s\in[0,\delta]$}, \\
       \hat{\gamma}(s)={\sigma}_{P_2}((1-s){\nu}_2) & \text{for $s\in[1-\delta,1]$,} \\
       B_{\frac{\delta}{2}}^g\big(\hat{\gamma}(s)\big)\cap \{P_1,P_2\} =\emptyset & \text{for $s\in[\delta,1-\delta]$},
   \end{cases*}
\end{equation*}
where $\nu_1,\nu_2\in\Sp^3$ and ${\sigma}_P$ is defined as in \eqref{eq:prj-def}. The last three conditions imply
\begin{equation*}
   \big\{s\mid d_g\big(\hat{\gamma}(s),\{P_1,P_2\}\big)<b\big\}:=[0,b)\cup(1-b,1], \quad \forall \, b\in(0,\delta/2]. 
\end{equation*}
At this point, let $t=t(\mu):=\epsilon^{\alpha}+(\mu-2)(1-2\epsilon^{\alpha})$, $\tau=\tau(t):=\min\{t^{{\omega/\alpha}}, (1-t)^{\omega/\alpha}, (\delta/2)^{\omega/\alpha}\}$ and define the curve $\bar{\bar{\gamma}}\colon[0,5]\to W^{1,2}(M)$ as follows:
\begin{equation}\label{eq:gammabarbar}
    \bar{\bar{\gamma}}(\mu):=\begin{cases*}
        u^{P_1}_{\epsilon,\mu\epsilon^{\alpha}\nu_1} & \text{for $\mu\in[0,1]$}, \\
        \varphi^{P_1}_{\mu-1,\nu_1} & \text{for $\mu\in[1,2]$}, \\
        e^{\frac{f^{\hat{\gamma}(t)}}{2}}w_{\hat{\gamma}(t),\epsilon,\tau(t)} & \text{for $\mu\in[2,3]$}, \\
        \varphi^{P_2}_{4-\mu,\nu_2} & \text{for $\mu\in[3,4]$}, \\
        u^{P_2}_{\epsilon,(5-\mu)\epsilon^{\alpha}\nu_2} & \text{for $\mu\in[4,5]$}.
    \end{cases*}
\end{equation}
Finally, we define $\bar{\gamma}\colon[0,1]\to W^{1,2}(M)$ as:
\begin{equation*}%\label{eq:bub-path}
    \bar{\gamma}(\mu):=\bar{\bar{\gamma}}(5\mu).
\end{equation*}

\begin{proposition}\label{prop:gammabar-yam}
    $\bar{\gamma}\in \Pi$ and, for $\epsilon$ small enough,  there holds $Q_g(\bar{\gamma}(\mu))<\mathcal{Y}\big(\Sp^4,[g_{\Sp^4}]\big)$, $\forall \mu\in [0,1]$.
\end{proposition}
\begin{proof}
    The continuity of $\bar{\gamma}$ follows directly by construction. In particular, it is clear from \eqref{eq:doubl-bub-test-M}, \eqref{eq:sch-bub-test-M} and \eqref{eq:psi-lambd} that all the test functions $u_{\epsilon,t}$, $\varphi_\lambda$ and $w_{q,\epsilon}$ are continuous in $W^{1,2}(M)$ with respect to the parameters $t,\lambda,q$ (and $\epsilon$); moreover, in \eqref{eq:gammabarbar} we always have continuous transitions between one test function and the other.

    As for the inequality, it is a consequence of Propositions \ref{prop:exp.doub-bub}, \ref{prop:exp-sch-bub-collapse} and \ref{prop:interp-expansion} when the parameter $\mu$ in $\bar{\bar{\gamma}}$ lies in the intervals $[0,2]$ or $[3,5]$, while, for $\mu\in[2,3]$, it is a consequence of Proposition \ref{prop:exp-sch-bub-collapse} and of \eqref{eq:green-test-standard-expansion} coupled with the  positive mass theorem from \cite{Dai-Sun-Wang-24-preprint} for AF conical manifolds.
\end{proof}

\medskip

\begin{proof}[Proof of Theorem \ref{t:main}] 
    By Lemma \ref{lemma:concentration-compactness} and Proposition \ref{prop:gammabar-yam}, the above min-max scheme produces a Palais-Smale sequence     $(\tilde{\gamma}_k)_k$ at level $c \in \left( \frac{\sqrt{2}}{2} \Y_4, \Y_4 \right)$. 
    Since we can always replace $\tilde{\gamma}_k$ by its absolute value without affecting the Yamabe quotient, we can also suppose that $\tilde{\gamma}_k \geq 0$. 
If by contradiction there  is no solution to the Yamabe equation on $(M,g)$, $(\tilde{\gamma}_k)_k$ must develop $j_1$ singular bubbles and 
$j_2$ regular bubbles as in \eqref{eq:PS} (with $j_1 + j_2 \geq 1$), which are positive by the above comment. 

Letting as before $(C(\R\projective^3),g_0):=((0,+\infty)\times \R\projective^3, ds^2+s^2h_0)$ be the metric cone over $\R\projective^3$,
we next notice that every positive finite-energy solution of the Yamabe equation on $(C(\R\projective^3),g_0)$ lifts to a weak and 
hence regular solution to the Yamabe equation in $\R^4$, and must therefore 
be of the type $U$ as in \eqref{eq:std-norm-bub}. As a consequence, the Yamabe quotient of each bubble $u_{p_{i,k},\eps_{i,k}}$ must be equal to 
$\frac{\sqrt{2}}{2} \Y_4$ (and that of each bubble $v_{p_{l,k},\eps_{l,k}}$
equal to $\Y_4$). Hence, it is easy to see that
\begin{equation} \label{eq:j1j2}
    \lim_k Q_g(\tilde{\gamma}_k) = (j_1 + 2 j_2)^{\frac{1}{2}} \frac{\sqrt{2}}{2} \Y_4, 
\end{equation}
which gives a contradiction since $c \in \left( \frac{\sqrt{2}}{2} \Y_4, \Y_4 \right)$. 
\end{proof}

%% file: doublebubble.tex
\section{Expansions of the Yamabe quotient}\label{section:expansions}

This technical section contains the expansions of Yamabe's quotients for all the different test functions $u_{\epsilon,t}$, $w_{q,\epsilon}$ and $\psi_\lambda$ defined in Section \ref{sec:variational-argument}.

\subsection{Proof of Proposition \ref{prop:exp.doub-bub}}
\label{sec:doub-bubble}

In order to prove Proposition \ref{prop:exp.doub-bub}, we need to  expand  the Yamabe quotient on $M$ for the ``double bubble'' $u_{\epsilon,t}$ defined in \eqref{eq:doubl-bub-test-M}.
 For the sake of clarity, we recall that, in what follows, $\delta>0$ is a small fixed number and $t,\epsilon$ are two positive small parameters such that $t=t(\epsilon)\to 0$ as $\epsilon\to0$; as it will be clear later, we will also assume that $\epsilon/t\to0$ as $\epsilon\to0$. We will use $C$ to denote various positive constants which are always greater than $1$, which may change from line to line and which \emph{do not depend} upon $\epsilon,t$ and $\delta$. Moreover, for the sake of simplicity, we will denote $\widehat{U}_{\epsilon,t}(y):=\widehat{U}_{\epsilon,t\mathbf{e}_1}(y)$.

By definition, the Yamabe quotient $Q_g(u_{\epsilon,t})$ satisfies
\begin{align}\label{eq:sob-quot-db-loc}
    Q_g(u_{\epsilon,t})&=\frac{\int_M \big(a\abs{\nabla_{g} u_{\epsilon,t}}_{{g}}^2+R_{g} (u_{\epsilon,t})^2\big)\,d\mu_{g}}{\Big(\int_M \abs{u_{\epsilon,t}}^4\,d\mu_g\Big)^\frac{1}{2}}=\frac{\int_{B_{2\delta}} \big(a\abs{\nabla_{\tilde{g}} \tilde{u}_{\epsilon,t}}_{\tilde{g}}^2+R_{\tilde{g}} (\tilde{u}_{\epsilon,t})^2\big)\,d\mu_{\tilde{g}}}{\sqrt{2}\Big(\int_{B_{2\delta}} \abs{\tilde{u}_{\epsilon,t}}^4\,d\mu_{\tilde{g}}\Big)^{\frac{1}{2}}}, 
\end{align}
that is, $Q_g(u_{\epsilon,t})=\frac{1}{\sqrt{2}}Q_{\tilde{g}}(\tilde{u}_{\epsilon,t})$, where we recall that $\tilde{u}_{\epsilon,t}(y):=\widehat{U}_{\eps,t}(y)\chi_\delta(\abs{y})$.

 \subsubsection*{Denominator}
Recalling \eqref{eq:lift-metr-exp}, one has
\begin{align}
\label{eq:den-exp-gen}
\int_{B_{2\delta}}\abs{\tilde{u}_{\epsilon,t}}^4\,d\mu_{\tilde{g}}=\int_{B_\delta}\abs{\widehat{U}_{\epsilon,t}}^4(1+O(\abs{y}^2))\,dy+\int_{B_{2\delta}\backslash B_\delta}\abs{\widehat{U}_{\epsilon,t}}^4\chi_\delta^4(\abs{y})(1+O(\abs{y}^2))\,dy.
\end{align}
First of all, we focus on the principal term of \eqref{eq:den-exp-gen}, that is,
\begin{equation}\label{eq:denom-double-bubble}
    \int_{B_\delta}\abs{\widehat{U}_{\epsilon,t}}^4\,dy=\int_{B_\delta}\big( \Uet^4+\Uetm^4+4\Uet^3\Uetm+4\Uet\Uetm^3+6\Uet^2\Uetm^2\big)\,dy.
\end{equation}
We have, as $\epsilon\to 0$,
\begin{align}
    \notag
    \int_{B_\delta}\Uetm^4\,dy=\int_{B_\delta}\Uet^4\,dy &=1-c_4^4\int_{\R^4\backslash B_\delta}\frac{\epsilon^{-4}}{(1+\epsilon^{-2}\abs{y-t\mathbf{e}_1}^2)^4}\,dy \\
    \label{eq:buest1}
    &= 1+O(\epsilon^4)\int_{\R^4\backslash B_\delta}\frac{dy}{\abs{y}^8}= 1+O(\epsilon^4).
\end{align}
Consider the last term of \eqref{eq:denom-double-bubble}:
\begin{equation*}
    \int_{B_\delta}\Uet^2\Uetm^2=\int_{B_\delta}\frac{c_4^4\epsilon^{-4}}{\Uetden^2\Uetmden^2}\,dy.
\end{equation*}
We split $B_\delta$ as follows:
\begin{equation}\label{eq:int-split-1}
B_\delta=B_{t/2}(t\mathbf{e}_1)\cup B_{t/2}(-t\mathbf{e}_1)\cup \big( B_{2t}(0)\backslash(B_{t/2}(t\mathbf{e}_1)\cup B_{t/2}(-t\mathbf{e}_1))\big)\cup \big( B_\delta\backslash B_{2t(0)}\big).
\end{equation}
One has
\begin{align*}
\Uetmint\Uet^2\Uetm^2&=\Uetint\Uet^2\Uetm^2\leq \frac{C}{t^4}\Uetint\frac{dy}{\Uetden^2} \\
&\leq \frac{C}{t^4}\Big(\Uetintint\,dy+\Uetintext\frac{\epsilon^4\,dy}{\abs{y-t\mathbf{e}_1}^4}\Big) \\
&\leq \frac{C}{t^4}\Big( \epsilon^4 +\epsilon^4 \log\Big(\frac{t}{\epsilon}\Big)\Big)\leq C\frac{\epsilon^4}{t^4}\log\Big(\frac{t}{\epsilon}\Big),  \quad \text{as $\epsilon\to 0$},
\end{align*}

\begin{align*}
    \Uetmidint\Uet^2\Uetm^2&\leq C\frac{\epsilon^4}{t^8}\int_{B_{2t}(0)}\,dy \leq C\frac{\epsilon^4}{t^4}, \quad \text{as $\epsilon\to 0$},
\end{align*}

\begin{align*}
    \Uetextint\Uet^2\Uetm^2\leq  C\epsilon^4\Uetextint\frac{dy}{\abs{y}^8}\leq C\frac{\epsilon^4}{t^4}, \quad \text{as $\epsilon\to0$},
\end{align*}
so, in the end, one gets
\begin{equation}\label{eq:buest3}
    \int_{B_\delta} \Uet^2\Uetm^2=O\Big(\frac{\epsilon^4}{t^4}\log\frac{\epsilon}{t}\Big)=o\Big(\frac{\epsilon^2}{t^2}\Big), \quad \text{as $\epsilon\to0$}.
\end{equation}
%%%%%%%%%%%%%%%%%%%%%%%%%%%%%%%%%%%%%%%%
Consider now the third and fourth terms of \eqref{eq:denom-double-bubble}. Let $\tau=\tau(\epsilon)>0$ be another parameter such that $0<\epsilon\ll \tau \ll t$. We claim that 
\begin{equation}\label{eq:est-uetuetm-int}
    \int_{B_\delta \backslash B_\tau(t\mathbf{e}_1)}\Uet^3 \Uetm\,dy=o\Big(\frac{\epsilon^2}{t^2}\Big), \quad \text{as $\epsilon\to0$.}
\end{equation}
Indeed, this can be easily seen by splitting the integral on the three regions $B_{2t}(0)\backslash\big( B_{t/2}(-t\mathbf{e}_1)\cup B_\tau(t\mathbf{e}_1)\big)$, $B_{t/2}(-t\mathbf{e}_1)$, $B_\delta \backslash B_{2t}(0)$ and arguing similarly to what did in order to obtain \eqref{eq:buest3}.
Using this formula together with \eqref{eq:Y-eq-Rn} and an integration by parts, we get
\begin{align*}
    \int_{B_\delta} \Uet^3 \Uetm \,dy&=\int_{B_\tau(t\mathbf{e}_1)}\Uet^3\Uetm \,dy+o\Big(\frac{\epsilon^2}{t^2}\Big) \\
    &=-\frac{1}{\mathcal{S}_4}\int_{\p B_\tau(t\mathbf{e}_1)}(\p_\nu \Uet) \Uetm\,d\sigma+\frac{1}{\mathcal{S}_4}\int_{B_\tau(t\mathbf{e}_1)}\nabla \Uet \nabla \Uetm \,dy+o\Big(\frac{\epsilon^2}{t^2}\Big).
\end{align*}
We then observe that, for $y\in B_\tau(x)$, one has $\abs{\Uetm(y)} \lesssim \epsilon t^{-2}$ and $\abs{\nabla \Uetm}\lesssim \epsilon t^{-3}$, hence
\begin{equation}\label{eq:uetuetm-est}
    \Big\vert\int_{B_\tau(t\mathbf{e}_1)}\nabla \Uet\cdot \nabla \Uetm\,dy\Big\vert\leq C \frac{\epsilon}{t^3}\int_{B_\tau(0)}\frac{\epsilon^{-3}\abs{y}}{(1+\epsilon^{-2}\abs{y}^2)^2}\,dy=O\Big(\frac{\epsilon^2 \tau}{t^3}\Big)=o\Big(\frac{\epsilon^2}{t^2}\Big).
\end{equation}
Also, a direct computation shows that
\begin{equation*}
    \int_{\p B_\tau(t\mathbf{e}_1)}(\p_\nu \Uet) \Uetm\,d\sigma=-\pi^2 c_4^2 \frac{\epsilon^2}{t^2}+o\Big(\frac{\epsilon^2}{t^2}\Big),
\end{equation*}
therefore
\begin{align}\label{eq:buest2}
    \int_{B_\delta}\Uet^3 \Uetm \,dy=\frac{\pi^2 c_4^2}{\mathcal{S}_4} \frac{\epsilon^2}{t^2}+o\Big(\frac{\epsilon^2}{t^2}\Big), 
\end{align}
and the same is true by symmetry for the integral of $\Uet \Uetm^3$ over $B_\delta$.
Substituting \eqref{eq:buest1}, \eqref{eq:buest3} and \eqref{eq:buest2} into \eqref{eq:denom-double-bubble}, we finally obtain
\begin{equation}\label{eq:den-main-term}
    \int_{B_\delta}\abs{\widehat{U}_{\epsilon,t}}^4\,dy=2+\frac{4}{3}\frac{A}{\mathcal{S}_4}\frac{\epsilon^2}{t^2}+o\Big(\frac{\epsilon^2}{t^2}\Big), \quad \text{as $\epsilon\to0$},
\end{equation}
where 
\begin{equation}\label{eq:A-formula}
    A:=6\pi^2 c_4^2.
\end{equation}
We now want to show that, under a suitable choice of $t$ as a function of $\epsilon$, the remaining terms of \eqref{eq:den-exp-gen} generate a lower order contribution as $\epsilon\to0$. 
Splitting the integrals similarly to \eqref{eq:int-split-1}, we estimate
\begin{multline}
    \left|\int_{B_\delta}\Uet^4O(\abs{y}^2)\,dy\right|\leq C\int_{B_\delta}\Uet^4\abs{y}^2\,dy\leq C\int_{B_\delta}\frac{\epsilon^{-4}\abs{y}^2}{\Uetden^4}\,dy \\
    \leq C\Big( \int_{B_{t/2}(0)}\frac{\epsilon^{4}\abs{y}^2}{t^8}\,dy+t^2\int_{B_{t/2}(t\mathbf{e}_1)} \Uet^4\,dy 
    +\int_{B_{2t}\backslash( B_{t/2}(0)\cup B_{t/2}(t\mathbf{e}_1))}\frac{\epsilon^{4}}{t^6}\,dy+\int_{B_\delta\backslash B_{2t}}\frac{\epsilon^{4}}{\abs{y}^6}\,dy\Big) \\
    \label{eq:err-d-est1}
    \leq C\Big( \frac{\epsilon^4}{t^2}+t^2+\frac{\epsilon^4}{t^2}+\frac{\epsilon^4}{\delta^2}\Big)= O(t^2)+o\Big(\frac{\epsilon^2}{t^2}\Big), \quad \text{as $\epsilon\to0$}.
\end{multline}
Moreover, splitting again the integrals as in \eqref{eq:int-split-1} and arguing as above, it is also easy to show that
\begin{equation}\label{eq:err-d-est2}
    \int_{B_\delta}\big(\Uet^3 \Uetm+ \Uet^2\Uetm^2+\Uet\Uetm^3\big)\abs{y}^2\,dy=o\Big(\frac{\epsilon^2}{t^2}\Big), \quad \text{as $\epsilon\to0$}.
\end{equation}
We now assume the following:
\begin{equation}\label{eq:alpha-assump}
    t=t(\epsilon):=\epsilon^{\alpha}, \quad \text{for $1/2<\alpha<1$}.
\end{equation}
From \eqref{eq:err-d-est1}, \eqref{eq:err-d-est2} and \eqref{eq:alpha-assump} we get
\begin{equation}\label{eq:den-hot}
     \int_{B_\delta}\abs{\widehat{U}_{\epsilon,t}}^4\abs{y}^2\,dy=O(t^2)+o\Big(\frac{\epsilon^2}{t^2}\Big)= o\Big(\frac{\epsilon^2}{t^2}\Big)= o(\epsilon^{2(1-\alpha)}), \quad \text{as $\epsilon\to0$}.
\end{equation}
Consider now the last term of \eqref{eq:den-exp-gen}; being $\abs{\widehat{U}_{\epsilon,t}(y)}\leq (C \epsilon)/{\abs{y}^2}$ for $\abs{y}\geq\delta$, one gets
\begin{equation}\label{eq:err-d-est3}
    \int_{B_{2\delta}\backslash B_\delta}\abs{\widehat{U}_{\epsilon,t}}^4\chi_\delta^4(\abs{y})\abs{y}^2\,dy\leq C \int_{B_{2\delta}\backslash B_\delta}\frac{\epsilon^4}{\abs{y}^8}\,dy=O(\epsilon^4), \quad \text{as $\epsilon\to0$}.
\end{equation}

By virtue of \eqref{eq:den-main-term}, \eqref{eq:den-hot} and \eqref{eq:err-d-est3}, we can now display the complete expansion of \eqref{eq:den-exp-gen}:
\begin{equation}\label{eq:den-db-expansion}
    \int_{B_{2\delta}}\abs{\tilde{u}_{\epsilon,t}}^4\,d\mu_{\tilde{g}}=2+8\frac{\pi^2 c_4^2}{\mathcal{S}_4}\frac{\epsilon^2}{t^2}+o\Big(\frac{\epsilon^2}{t^2}\Big)=2+\frac{4}{3}\frac{A}{\mathcal{S}_4}\epsilon^{2(1-\alpha)}+o(\epsilon^{2(1-\alpha)}), \quad \text{as $\epsilon\to0$.}
\end{equation}

\subsubsection*{Numerator} 
We now want to expand the numerator of \eqref{eq:sob-quot-db-loc}:
\begin{align}
\notag
    \int_{B_{2\delta}} \big(a&\abs{\nabla_{\tilde{g}} \tilde{u}_{\epsilon,t}}_{\tilde{g}}^2+R_{\tilde{g}} (\tilde{u}_{\epsilon,t})^2\big)\,d\mu_{\tilde{g}}=\int_{B_\delta}\big(a\abs{\nabla_{\tilde{g}} \Uethat}_{\tilde{g}}^2+R_{\tilde{g}} (\Uethat)^2\big)(1+O(\abs{y}^2))\,dy \\
    \label{eq:num-exp-gen}
    &+\int_{B_{2\delta}\backslash B_\delta}\big(a\abs{\chi_\delta \nabla_{\tilde{g}}\Uethat+\Uethat \nabla_{\tilde{g}}\chi_\delta}_{\tilde{g}}^2+R_{\tilde{g}}\chi_\delta^2\Uethat^2\big)(1+O(\abs{y}^2))\,dy.
\end{align}
If $\tilde{g}$ is as in \eqref{eq:lift-metr-exp}, then, writing $\tg^{-1}$ as a Neumann series, we immediately see that 
\begin{equation*}
\tilde{g}^{ij}(y)=\delta^{ij}+O(\abs{y}^2),
\end{equation*}
therefore
\begin{equation}\label{eq:gradg-est}
\left|\abs{\nabla_{\tilde{g}}u}_{\tilde{g}}^2-\abs{\nabla u}^2\right|\leq C\abs{y}^2\abs{\nabla u}^2.
\end{equation}

To begin, consider the main term
\begin{align}\label{eq:num-grad-exp}
\int_{B_\delta}\abs{\nabla\Uethat}^2\,dy=2\int_{B_\delta}\abs{\nabla\Uet}^2\,dy+2\int_{B_\delta}\nabla\Uet\nabla\Uetm\,dy.
\end{align}
Integrating by parts, we get
\begin{align*}
    \int_{B_\delta}\abs{\nabla\Uet}^2\,dy&=\int_{\R^4}\abs{\nabla\Uet}^2\,dy-\int_{\R^4\backslash B_\delta}\abs{\nabla\Uet}^2\,dy \\
    &=\int_{\R^4} (-\Delta\Uet)\Uet\,dy-\int_{\R^4\backslash B_\delta}\frac{4 c_4^2 \epsilon^{-6}\abs{y-t\mathbf{e}_1}^2}{\Uetden^4}\,dy \\
    &=\mathcal{S}_4\int_{\R^4}\Uet^4\,dy+O(\epsilon^2)\int_{\R^4\backslash B_\delta}\frac{dy}{\abs{y}^6}=\mathcal{S}_4+O(\epsilon^2), \quad \text{as $\epsilon\to0$},
\end{align*}
where we used \eqref{eq:Y-eq-Rn} in the last line. For the other term of \eqref{eq:num-grad-exp}, by an integration by parts, \eqref{eq:Y-eq-Rn} and \eqref{eq:buest2} we get
\begin{align}\label{eq:gradprod-integral}
    \int_{B_\delta}\nabla\Uet\nabla\Uetm=\pi^2 c_4^2\frac{\epsilon^2}{t^2}+o\Big(\frac{\epsilon^2}{t^2}\Big), \quad  \text{as $\epsilon\to0$}.
\end{align}
These estimates and \eqref{eq:A-formula} imply
\begin{equation}\label{eq-n-g-e-f}
    \int_{B_\delta}\abs{\nabla\Uethat}^2=2\mathcal{S}_4+\frac{A}{3}\frac{\epsilon^2}{t^2}+o\Big(\frac{\epsilon^2}{t^2}\Big), \quad \text{as $\epsilon\to0$}.
\end{equation}

We now turn to the lower-order terms of \eqref{eq:num-exp-gen}. By virtue of \eqref{eq:gradg-est} and \eqref{eq:lift-metr-exp} we see that
\begin{align}
\notag
    \Big|\int_{B_\delta}\abs{\nabla_{\tilde{g}}\Uethat}^2(1+O(\abs{y}^2))\,dy-&\int_{B_\delta}\abs{\nabla \Uethat}^2\,dy\Big|\leq  C\int_{B_\delta}\abs{\nabla\Uethat}^2\abs{y}^2\,dy \\
    \label{eq:estrem-1}
    &\leq C\Big( \int_{B_\delta}\abs{\nabla\Uet}^2\abs{y}^2\,dy+\int_{B_\delta}\nabla\Uet\nabla\Uetm\abs{y}^2\,dy\Big).
\end{align}
Splitting the integrals as in \eqref{eq:err-d-est1}, after some basic estimates we get
\begin{align*}
    \int_{B_\delta}\abs{\nabla\Uet}^2\abs{y}^2\,dy&=4c_4^2\int_{B_\delta}\frac{\epsilon^{-6}\abs{y-t\mathbf{e}_1}^2\abs{y}^2}{\Uetden^4}\,dy \\ 
    &\leq C\Bigg\{ \frac{\epsilon^2}{t^6}\int_{B_{t/2}}\abs{y}^2\,dy+\frac{t^2}{\epsilon^6}\Big(\int_{B_\epsilon(t\mathbf{e}_1)}\abs{y-t\mathbf{e}_1}^2\,dy+\int_{B_{t/2}\backslash B_\epsilon(t\mathbf{e}_1)}\frac{\epsilon^8}{\abs{y-t\mathbf{e}_1}^6}\,dy\Big) \\
   & +\int_{B_{2t}\backslash( B_{t/2}(0)\cup B_{t/2}(t\mathbf{e}_1))}\frac{\epsilon^2}{t^4}\,dy+\int_{B_\delta\backslash B_{2t}}\frac{\epsilon^{2}} {\abs{y}^4}\,dy\Bigg\} \\[0.8ex]
   & = O(\epsilon^2)+ O(t^2)+O(\epsilon^2\log t) =O(t^2)= O(\epsilon^{2\alpha}), \quad \text{as $\epsilon\to0$},
\end{align*}
where the last equality follows from \eqref{eq:alpha-assump}.
For the other term, arguing similarly we obtain
\begin{align*}
    \left|\int_{B_\delta}\nabla\Uet\nabla\Uetm\abs{y}^2\,dy\right|= O(\epsilon^2) + O(\epsilon^2 \log t), \quad \text{as $\epsilon\to0$.}
\end{align*}
We can now apply those estimates to \eqref{eq:estrem-1} to get 
\begin{align}\label{eq:estrem2}
    \Big|\int_{B_\delta}\abs{\nabla_{\tilde{g}}\Uethat}^2(1+O(\abs{y}^2))\,dy-\int_{B_\delta}\abs{\nabla \Uethat}^2\,dy\Big|=O(t^2)=O(\epsilon^{2\alpha}), \quad \text{as $\epsilon\to0$.}
\end{align}

Consider now the scalar curvature term of \eqref{eq:num-exp-gen}; recalling \eqref{eq:lift-metr-exp} and $R_{\tilde{g}}\in L^\infty(B_{2\delta})$, one has
\begin{align*}
    \Big|\int_{B_\delta} R_{\tilde{g}} (\Uethat)^2(1+O(\abs{y}^2))\,dy\Big|\leq  C\int_{B_\delta}\big(\Uet^2+\Uet\Uetm\big)\,dy.
\end{align*}
But
\begin{align*}
    \int_{B_\delta}\Uet^2 = c_4^2\int_{B_\delta}\frac{\epsilon^{-2}}{\Uetden^2}\,dy&\leq C\int_{B_\epsilon(t\mathbf{e}_1)}\epsilon^{-2}\,dy+C\int_{B_{2\delta}\backslash B_{\epsilon}(0)}\frac{\epsilon^2}{\abs{y}^4}\,dy \\
    &= O(\epsilon^2) +O(\epsilon^2\log\epsilon), \quad \text{as $\epsilon\to0$},
\end{align*}
while 
\begin{align*}
    \int_{B_\delta}\Uet\Uetm&=c_4^2\int_{B_\delta}\frac{\epsilon^{-2}}{\Uetden\Uetmden}\,dy \\
    &\leq C\bigg( \int_{B_{t/2}(t \mathbf{e}_1)}\frac{dy}{t^2\Uetden}+C\epsilon^2+\Uetextint\frac{\epsilon^2}{\abs{y}^2}\,dy\bigg) \\
    &=O\Big(\frac{\epsilon^4}{t^2}\Big)+O(\epsilon^2)+O(\epsilon^2\log t)=O(\epsilon^2\log t), \quad \text{as $\epsilon\to0$.}
\end{align*}
Hence
\begin{align}\label{eq:estrem3}
    \int_{B_\delta} R_{\tilde{g}} (\Uethat)^2(1+O(\abs{y}^2))\,dy= O(\epsilon^2\log\epsilon), \quad \text{as $\epsilon\to0$}.
\end{align}
Finally,for $\abs{y}\geq\delta$ (fixed), an easy estimate on the last term of \eqref{eq:num-exp-gen} gives
\begin{align}\label{eq:estrem4}
   \Big| \int_{B_{2\delta}\backslash B_{\delta}}\big(a\abs{\nabla_{\tilde{g}} \tilde{u}_{\epsilon,t}}_{\tilde{g}}^2+R_{\tilde{g}} (\tilde{u}_{\epsilon,t})^2\big)\,d\mu_{\tilde{g}}\Big|\leq C\epsilon^2. 
\end{align}
Substituting \eqref{eq-n-g-e-f}, \eqref{eq:estrem-1}, \eqref{eq:estrem2}, \eqref{eq:estrem3} and \eqref{eq:estrem4} inside \eqref{eq:num-exp-gen} we get, after recalling \eqref{eq:alpha-assump} and that $a=6$, the following expansion of the numerator:
\begin{align}\label{eq:num-db-exp}
\int_{B_{2\delta}}\big(a\abs{\nabla_{\tilde{g}} \tilde{u}_{\epsilon,t}}_{\tilde{g}}^2+R_{\tilde{g}} (\tilde{u}_{\epsilon,t})^2\big)\,d\mu_{\tilde{g}}=12\mathcal{S}_4+2A\epsilon^{2(1-\alpha)}+o\big(\epsilon^{2(1-\alpha)}\big), \quad \text{as $\epsilon\to0$.}
\end{align}

\subsubsection*{Conclusion}
From \eqref{eq:sob-quot-db-loc}, \eqref{eq:den-db-expansion}, \eqref{eq:num-db-exp} and a Taylor expansion of the denominator, we have
\begin{align*}
Q_g(u_{\epsilon,t})=\frac{6\mathcal{S}_4+A\epsilon^{2(1-\alpha)}+o\big(\epsilon^{2(1-\alpha)}\big)}{\Big(1+\frac{2}{3}A \mathcal{S}_4^{-1}\epsilon^{2(1-\alpha)}+o\big(\epsilon^{2(1-\alpha)}\big)\Big)^{\frac{1}{2}}} =6\mathcal{S}_4-A\epsilon^{2(1-\alpha)}+o\big(\epsilon^{2(1-\alpha)}\big), \quad \text{as $\epsilon\to 0$}.
\end{align*}

This proves Proposition \ref{prop:exp.doub-bub}. In particular,  $Q_g(u_{\epsilon,t})<6\mathcal{S}_4=\mathcal{Y}\big(\Sp^4,[g_{\Sp^4}]\big)$ when $\epsilon$ is small enough.

%% file: bubblegreen.tex
\subsection{Proof of Proposition \ref{prop:exp-sch-bub-collapse}}\label{sec:greenbubl-expansion}

In this subsection we are going to prove Proposition \ref{prop:exp-sch-bub-collapse}.
The computations are similar to those performed in \cite{Lee-Parker-87} (although here we work directly on $(M,g^q)$ instead of working on its conformal blow-up at $q$), but we need to be more careful as the mass of the Green's function $G_q$ and the gluing parameter $\tau$ are not fixed but depend upon the distance from the conical point $P$.

Assume that $t:=d_{g}(q,P)\leq \delta/4$; by virtue of Lemma \ref{lem:green-exp-manifold}, we know that the Green's function for $L_{g^q}$ has the following expansion in normal coordinates for $g^q$ around $q$:
\begin{equation*}
    G_q(z)=\frac{1}{\abs{z}^2}+A_q+ \beta_q(z),
\end{equation*}
where we further recall that
\begin{equation}\label{eq:betaq-est}
    \beta_q(0)=0, \quad \abs{{\beta}_q(z)}\leq C t^{b-3}, \quad \abs{\nabgq \beta_q}\leq Ct^{-3} \quad \text{for any $b>1$, $C=C(b)>0$  and $\forall\, \abs{z}\leq t^b$}.
\end{equation}

\medskip

\noindent
\textbf{Assumption.} In the following, we will take $\epsilon^\alpha\leq t\leq \delta/4$ and $\tau=\tau(t):=t^{\omega/\alpha}$, where, at the moment, we only require $1>\omega>\alpha>\frac{1}{2}$. We also assume that the above constant $b$  satisfies $\omega/\alpha>b>1$.

\medskip

We now compute the various terms in the expansion of $Q_{g^q}(w_{q,\epsilon})$, where $w_{q,\epsilon}$ is defined in \eqref{eq:sch-bub-test-M}.
To begin, recalling that in normal coordinates for $g^q$ around $q$ one has $\dmuq(x)=1+t^{-2}O(\abs{z}^4)$ (cf. \eqref{eq:vol-det-exp}), we compute
\begin{align}
\notag
    \int_{B_\tau^{g^q}(q)}\abs{\nabgq w_{q,\epsilon}}^2\dmuq &=\int_{B_\tau(0)}\abs{\nabla U_\epsilon}^2\big(1+ t^{-2}O(\abs{z}^4)\big)\,dz \\
    \notag
    &=\int_{B_\tau}(-\Delta U_\epsilon) U_\epsilon \,dz +\int_{\partial B_\tau}(\partial_\nu U_\epsilon) U_\epsilon\,d\sigma+t^{-2}\int_{B_{\tau}}\abs{\nabla U_\epsilon}^2 O(\abs{z}^4)\,dz \\
    \notag
    &=\mathcal{S}_4 \int_{B_\tau} U_\epsilon^4 \,dz +\int_{\partial B_\tau}(\partial_\nu U_\epsilon) U_\epsilon\,d\sigma+O(t^{-2}\epsilon^2 \tau^2) \\
    \label{eq:schest-1}
    &=\mathcal{S}_4+\int_{\partial B_\tau}(\partial_\nu U_\epsilon) U_\epsilon\,d\sigma+O\Big(\frac{\epsilon^4}{\tau^4}\Big)+O(t^{-2}\epsilon^2\tau^2).
\end{align}
Regarding the scalar curvature term, being $R_{g^q}=t^{-2}O(\abs{z}^2)$ in our coordinate system (see \eqref{eq:sc-curv-exp})
\begin{align}
\notag
    \int_{B_\tau^{g^q}(q)} \abs{R_{g^q}}w_{q,\epsilon}^2\dmuq&\leq C\int_{B_\tau(0)}\frac{t^{-2}\abs{z}^2\epsilon^{-2}}{\big(1+\epsilon^{-2}\abs{z}^2\big)^2}\,dz \\
    \label{eq:schest-2}
    &\leq C t^{-2}\int_{B_\epsilon}\abs{z}^2\epsilon^{-2}\,dz +t^{-2}\int_{B_\tau\backslash B_\epsilon}\frac{\epsilon^2}{\abs{z}^2}\,dz \leq C t^{-2}\epsilon^2 \tau^2.
\end{align}
Before looking at the integral of the numerator outside $B_\tau^{g^q}(q)$, we  first  expand the denominator:
\begin{align*}
    \int_M (w_{q,\epsilon})^4\dmuq=1+ O\Big(\frac{\epsilon^4}{\tau^4}\Big)+\int_{M\backslash B_\tau^{g^q}(q) }\Big(\frac{1}{\nu}G_q\Big)^4\dmuq. 
\end{align*}
Using the estimate $\abs{G_q(p)}\leq\frac{C}{d_{g^q}(p,q)^2}$ for some constant $C>0$ together with the fact that $\nu= O(\epsilon^{-1})$ (see \eqref{eq:cont-match} or \eqref{eq:estest3} below), we see that
\begin{align*}
    \int_{M\backslash B_\tau^{g^q}(q) }\Big(\frac{1}{\nu}G_q\Big)^4\dmuq=O\Big(\frac{\epsilon^4}{\tau^4}\Big),
\end{align*}
therefore 
\begin{align}\label{eq:schest-3}
    \int_M (w_{q,\epsilon})^4\dmuq=1 +O\Big(\frac{\epsilon^4}{\tau^4}\Big). 
\end{align}

We now focus on the integral of the numerator on the region $M\backslash B_\tau^{g^q}(q)$. By definition of $w_{q,\epsilon}$ (see \eqref{eq:sch-bub-test-M}), one has
\begin{align}
\notag
    &\int_{M\backslash B_\tau^{g^q}(q)}\big(a\abs{\nabgq w_{q,\epsilon}}^2+ R_{g^q}w_{q,\epsilon}^2\big)\dmuq= \frac{1}{\nu^2}\int_{M\backslash B_\tau^{g^q}(q)}\big( a\abs{\nabgq G_q}^2+R_{g^q} G_q^{2}\big) \dmuq \\
    \label{eq:ringr}
    &+\frac{1}{\nu^2}\int_{A_{\tau,2\tau}}\Big[a\abs{\nabgq(\chi_\tau \beta_q)}^2+ R_{g^q}\chi_{\tau}^2\beta_q^2-2 a\nabgq G_q\cdot \nabgq(\chi_\tau\beta_q)-2 R_{g^q}G_q\chi_\tau\beta_q\Big]\dmuq, 
\end{align}
where $A_{\tau,2\tau}:=\{z\in M \mid \tau<d_{g^q}(z,P)<2\tau\}$. Consider the first integral on the RHS of \eqref{eq:ringr}: by definition of $G_q$, we can integrate by parts and obtain
\begin{align}\label{eq:ringbt1}
    \int_{M\backslash B_\tau^{g^q}(q)}\big( a\abs{\nabgq G_q}^2+R_{g^q} G_q^{2}\big) \dmuq= - a\int_{\partial B_\tau^{g^q}(q)}\big(\partial_{\nu}G_q\big) G_q \,d\sigma_{g^q},
\end{align}
where the negative sign is due to the fact that the unit normal $\nu$ is \emph{outward pointing}. Consider now the second integral in \eqref{eq:ringr}; by \eqref{eq:chi-delta} and \eqref{eq:betaq-est}, we easily see that
\begin{align*}
    \abs{\nabgq(\chi_\tau\beta_q)}^2 &=\abs{\nabgq \chi_\tau}^2\beta_q^2+\abs{\nabgq \beta_q}^2\chi_\tau^2+2 \nabgq \chi_\tau \cdot \nabgq \beta_q \\
    &\leq C\Big(\frac{1}{\tau^2 t^{6-2b}}+\frac{1}{t^{6}}+\frac{1}{\tau t^{3}}\Big) \leq \frac{C}{\tau^2 t^{6-2b}+t^6},
\end{align*}
and also
\begin{gather*}
    \abs{\chi_\tau^2\beta_q^2 R_{g^q}}\leq Ct^{2b-6}, \\
    \abs{R_{g^q}G_q\chi_\tau\beta_q}\leq \frac{C}{\tau^2t^{3-b}}, \\
    \abs{\nabgq G_q \cdot \nabgq(\chi_\tau \beta_q)}\leq C\Big( \frac{1}{\tau^4 t^{3-b}}+\frac{1}{\tau^3 t^3}\Big).
\end{gather*}
Hence, recalling the relation between $\tau$ and $t$ and the fact that $\omega/\alpha>b>1$, we see that the main contribution comes from the last term and is of order $O(\tau^{-4}t^{b-3})$. As a consequence,
\begin{align}
\notag
    \frac{1}{\nu^2}  \int_{A_{\tau,2\tau}}\Big|\abs{\nabgq(\chi_\tau \beta_q)}^2+ & R_{g^q}\chi_{\tau}^2\beta_q^2-2 a\nabgq G_q\cdot \nabgq(\chi_\tau\beta_q)-2 R_{g^q}G_q\chi_\tau\beta_q\Big|\dmuq \\
    \label{eq:ringrt2}
    &\leq C\int_{A_{\tau,2\tau}}\frac{\epsilon^2}{\tau^4 t^{3-b}}\dmuq\leq C\frac{\epsilon^2}{t^{3-b}}.
\end{align}
Combining \eqref{eq:ringr}, \eqref{eq:ringbt1} and \eqref{eq:ringrt2}, one gets
\begin{align}\label{eq:schest-4}
    \int_{M\backslash B_\tau^{g^q}(q)}\big(a\abs{\nabgq w_{q,\epsilon}}^2+ R_{g^q}w_{q,\epsilon}^2\big)\dmuq=-\frac{a}{\nu^2}\int_{\partial B_\tau^{g^q}(q)}\big(\partial_{\nu}G_q\big) G_q \,d\sigma_{g^q} +O\Big(\frac{\epsilon^2}{t^{3-b}}\Big).
\end{align}

Finally, putting \eqref{eq:schest-1}, \eqref{eq:schest-2}, \eqref{eq:schest-3} and \eqref{eq:schest-4} together and recalling that $a=6$, one obtains the following expansion:
\begin{equation}\label{eq:schexp-mid}
    Q_{g^q}(w_{q,\epsilon})=6\mathcal{S}_4+6\int_{\partial B_\tau(0)} \big(\partial_\nu U_\epsilon\big) U_\epsilon \,d\sigma-\frac{6}{\nu^2} \int_{\p B_\tau^{g^q}(q)}\big(\partial_\nu G_q\big) G_q \,d\sigma_{g^q}+O\Big(\frac{\epsilon^2}{t^{3-b}}\Big)+O\Big(\frac{\epsilon^4}{\tau^4}\Big).
\end{equation}
We now focus on the two boundary integrals of \eqref{eq:schexp-mid}. At $\abs{z}=\tau$ one has
\begin{equation}\label{eq:estest1}
    \big(\partial_\nu U_\epsilon\big) U_\epsilon=-\frac{2c_4^2\epsilon^{-4}\tau}{\big(1+\epsilon^{-2}\tau^2\big)^3}=-2c_4^2\frac{\epsilon^2}{\tau^5}+6c_4^2\frac{\epsilon^4}{\tau^7}+\frac{\epsilon^2}{\tau^5}o\Big(\frac{\epsilon^2}{\tau^2}\Big) \quad \text{as $\frac{\epsilon}{\tau}\to0$}. 
\end{equation}
On $\partial B_\tau^{g^q}(q)$ instead there holds (recall $\tau=t^{\omega/\alpha}$)
\begin{align}
\notag
    \big(\partial_\nu G_q\big) G_q&=\Big(-\frac{2}{\tau^3}+ O(t^{-3})\Big)\Big(\frac{1}{\tau^2}+\frac{1}{4t^2}+O(t^{b-3})\Big) \\
    \label{eq:estest2}
    &=-\frac{2}{\tau^5}-\frac{1}{2\tau^3t^2}+ O(\tau^{-3}t^{b-3}), \quad \text{as $t\to0$.}
\end{align}
We also notice that, by virtue of \eqref{eq:cont-match}, one further has
\begin{align}
\notag
    \frac{1}{\nu}&=\frac{c_4\epsilon^{-1}}{(1+\epsilon^{-2}\tau^2)\Big(\frac{1}{\tau^2}+A_q\Big)}=\frac{c_4\epsilon}{(1+\epsilon^2 \tau^{-2})\Big(1+\tau^2 A_q\Big)} \\
    \label{eq:nu-exp}
    &=c_4\epsilon\Big(1-\tau^2 A_q-\frac{\epsilon^2}{\tau^2}+\epsilon^2 A_q+o\big(\tau^2 A_q^2\big)+o\Big(\frac{\epsilon^2}{\tau^2}\Big) \Big), \quad \text{as $\epsilon,t\to0$, $t\geq\epsilon^\alpha$}.
\end{align}
If we now recall the definition of $A_q$ (see Lemma \ref{lem:green-exp-manifold}) and assume that $2+2\alpha-4\omega>0$ (in addition to $1>\omega>\alpha>\frac{1}{2}$), then, for any $t\in[\epsilon^{\alpha},\delta/4]$, (and $\delta$ sufficiently small), one obtains 
\begin{align}\label{eq:estest3}
    \frac{1}{\nu}=c_4\epsilon\Big(1-\frac{\tau^2}{4t^2}+o\Big(\frac{\tau^2}{t^2}\Big)\Big), \quad \text{as $\epsilon\to0$.}
\end{align}
We can now substitute \eqref{eq:estest1}, \eqref{eq:estest2} and \eqref{eq:estest3} inside the boundary integrals of \eqref{eq:schexp-mid} (recall also that $d\sigma_{g^q}=\big(1+O(t^{-2}\tau^4)\big)$ because of \eqref{eq:vol-det-exp}) to get
\begin{align}
\notag
    \int_{\partial B_\tau(0)} \big(\partial_\nu U_\epsilon\big)& U_\epsilon \,d\sigma-\frac{1}{\nu^2} \int_{\p B_\tau^{g^q}(q)}\big(\partial_\nu G_q\big) G_q \,d\sigma_{g^q} \\
    \notag
    &=\int_{\partial B_\tau}\Big[6c_4^2\frac{\epsilon^4}{\tau^7}+\frac{1}{\tau^3}o\Big(\frac{\epsilon^4}{\tau^4}\Big)-\frac{c_4^2}{2}\frac{\epsilon^2}{\tau^3 t^2}-\frac{c_4^2}{4}\frac{\epsilon^2}{\tau t^4}+O\Big(\frac{\epsilon^2}{\tau^3t^{3-b}}\Big)+\epsilon^2o(\tau^{-3}t^{-2})\Big]\,d\sigma \\
    \label{eq:diffest-uet-1}
    &=-\pi^2c_4^2\frac{\epsilon^2}{t^2}\Big(1+\frac{\tau^2}{2t^2}+O(t^{b-1})+o_t(1)\Big) +O\Big(\frac{\epsilon^4}{\tau^4}\Big),
 \end{align}
where $o_t(1)\to0$ as $t\to0$. In particular, up to taking $\delta>0$ small enough, we may assume that, for $t<\delta$, one has
\begin{align*}
    1+\frac{\tau^2}{2t^2}+O(t^{b-1})+o_t(1)>\frac{3}{4}.
\end{align*}
Hence, for any $t\in[\epsilon^\alpha,\delta/4]$, by \eqref{eq:schexp-mid} and the above estimates we finally obtain
\begin{align*}
    Q_{g^q}(w_{q,\epsilon})<6\mathcal{S}_4-\frac{9}{2}\pi^2c_4^2\frac{\epsilon^2}{t^2}+O\Big(\frac{\epsilon^4}{\tau^4}\Big),
\end{align*}
which, by virtue of \eqref{eq:alphbar-assumpt}, implies that there exists $\bar{\epsilon}>0$ small enough so that
\begin{align*}
     Q_{g^q}(w_{q,\epsilon})<6\mathcal{S}_4 \quad \text{$\forall \epsilon<\bar{\epsilon}$, $\forall t\in[\epsilon^{\alpha},\delta/4]$. }
\end{align*}
In particular, when $t=\epsilon^\alpha$, one has the following expansion as $\epsilon\to0$:
\begin{align*}
     Q_{g^q}(w_{q,\epsilon})=6\mathcal{S}_4 -A \epsilon^{2(1-\alpha)}+o\big(\epsilon^{2(1-\alpha)}\big),
\end{align*}
where $A$ is given by \eqref{eq:A-formula}.
This completes the proof of Proposition \ref{prop:exp-sch-bub-collapse}.

%% file: Interpolation.tex
\subsection{Proof of Proposition \ref{prop:interp-expansion}}\label{sec:interpolation}

This subsection is devoted to estimating the Yamabe quotient $Q_{g^q}$ of the convex combination $\psi_\lambda$ defined in \eqref{eq:psi-lambd} of the ``double bubble'' $u_{\epsilon,t}$ (see \eqref{eq:doubl-bub-test-M}) and the test function $w_{q,\epsilon}$ defined in \eqref{eq:sch-bub-test-M}.
\begin{remark}\label{rem:interp-exp-assumptions}
    In the following, we assume to be in the setting of Proposition \ref{prop:interp-expansion}; in particular, we will always suppose that $t=\epsilon^\alpha$ and $\tau=\epsilon^\omega$, where $\alpha,\omega$ satisfy \eqref{eq:alphbar-assumpt}.
\end{remark}

\noindent
\textbf{Notation.} Since here $t$ depends on $\epsilon$ and $q$ is fixed, we will write for simplicity $u_\epsilon$ and $w_\epsilon$ in place of $u_{\epsilon,t}$ and $w_{q,\epsilon}$ respectively. 

\medskip

To begin, we immediately notice that, on $M\backslash B_\delta^{g^q}(P)$, one has
\begin{align*}
    u_{\epsilon}\leq C \frac{\epsilon}{\delta^2}, \qquad \abs{\nabgq u_{\epsilon}}\leq C \frac{\epsilon}{\delta^3}.
\end{align*}
Similarly, using the estimate $\abs{G_q(x)}\leq\frac{C}{d_{g^q}(q,x)^2}$ we also get 
\begin{align*}
    \abs{w_{\epsilon}}\leq  \frac{C}{\delta^2}, \qquad \abs{\nabgq w_{\epsilon}}\leq  \frac{C}{\delta^3} \quad \text{on $M\backslash B_\delta^{g^q}(P)$}.
\end{align*}
From these estimates and the definition of $\psi_\lambda$, it easily follows that
\begin{align*}
    \int_{M\backslash B_\delta^{g^q}(P)}\big|a\abs{\nabgq\psi_\lambda}^2+R_{g^q}\psi_\lambda^2\big|\dmuq\leq C\Big(\frac{\epsilon^2}{\delta^{6}}+\frac{\epsilon^2}{\delta^4}\Big)\textit{Vol}_{g^q}\big(M\backslash B_\delta^{g^q}(P)\big)= O\Big(\frac{\epsilon^2}{\delta^6}\Big),
\end{align*}
\begin{align*}
    \int_{M\backslash B_\delta^{g^q}(P)}\abs{\psi_\lambda}^4\dmuq\leq C\frac{\epsilon^4}{\delta^{8}}\textit{Vol}_{g^q}\big(M\backslash B_\delta^{g^q}(P)\big) =O\Big(\frac{\epsilon^4}{\delta^8}\Big).
\end{align*}
As a consequence, if (as in Section \ref{sec:greenfunct}) we denote by $\bg=\sigma_P^* g^q$ the equivariant lift of $g^q$ (which extends $C^{1,1}$ at the origin by Corollary \ref{cor:metric-smooth-extension}) to $B_{2\delta}(0)\subset \R^4$ defined as in \eqref{eq:prj-def}, then we see that
\begin{align}\label{eq:int-yam-quotient}
 Q_{g^q}(\psi_\lambda)=\frac{\int_{B_\delta^{\bg}(0)}\big(a\abs{\nabgb \phi_\lambda}_{\bg}^2+R_{\bg}\phi_\lambda^2\big)\dmub+O\Big(\frac{\epsilon^2}{\delta^6}\Big)}{\sqrt{2}\Big(\int_{B_\delta^{\bg}(0)}\abs{\phi_\lambda}^4\dmub+O\Big(\frac
 {\epsilon^4}{\delta^8}\Big)\Big)^{\frac{1}{2}}},
\end{align}
where now $\phi_\lambda:=\psi_\lambda\circ\sigma_P$ denotes the equivariant lift of $\psi_\lambda$ to $B_{\delta}(0)\subset \R^4$. Let us also define
\begin{equation}\label{eq:bubw-def}
\bu:= u_\epsilon\circ \sigma_P, \qquad \qquad\bw:=w_\epsilon\circ\sigma_P;
\end{equation}
it is then clear that
\begin{equation*}
    \phi_\lambda=\lambda\bw +(1-\lambda) e^{-f/2}\bu,
\end{equation*}
where $f$ is the function $f_x$ defined in \eqref{eq:f_x-definition}.

Let $\sigma_P^{-1}(q)=\pm x=\pm t\mathbf{e}_1\in B_\delta$ (where $t=d_{g}(q,P)$). In order to compute \eqref{eq:int-yam-quotient}, we will split the integrals in the regions $B^{\bg}_{\tau}(\pm x)$ and $B_\delta(0)\backslash \big(B^{\bg}_{\tau}(x)\cup B_\tau^{\bg}(-x)\big)$.

\begin{remark}
    (a) Notice that the \acc bubble profiles'' of $\bu$ and $\bw$ are \emph{not equal} in the regions $B^{\bg}_{\tau}(\pm x)$. Indeed, the profile of $\bw$ is defined in normal coordinates for $\bar{g}$ centered at $\pm x$, while the one of $\bu$ is defined with respect to normal coordinates for $\tilde{g}$ centered at the origin.
    Nevertheless, this difference only generates higher order error terms, see Lemmas \ref{lem:error-estimates-interpolation} and \ref{lem:error-estimates-interpolation-boundary} below.

    \medskip

    \noindent
    (b) By definition of the conformal factor $f^q$ in \eqref{eq:conformal-factor-f^q}, we see that, in normal coordinates for $\bg$ centered at $\pm x$, one has the following expansion:
    \begin{equation*}%\label{eq:conf-fact-err}
        e^{-(f^q(y))/{2}}=1+O^{''}(\abs{y}^2), \qquad \forall \abs{y}\leq t/2,
    \end{equation*}
    where the error term does not depend on $t$. In particular, the conformal factor will generate an error of order $\tau^2$ at distance $\tau$ from $\pm x$.
\end{remark}

\subsubsection*{Denominator}
By an explicit calculation, we find 
\begin{align*}
    \int_{B_{\delta}^{\bg}}\abs{\phi_\lambda}^4\dmub&=\int_{B_{\delta}^{\bg}} \big[\lambda^4\abs{\bw}^4+4\lambda^3(1-\lambda)e^{-f/2}\bw^3\bu+6\lambda^2(1-\lambda)^2e^{-f}\bw^2 \bu^2 \\
    & \quad \qquad \qquad  +4\lambda(1-\lambda)^3 e^{-(3f)/2}\bw  \bu^3+(1-\lambda)^4e^{-2f}\abs{\bu}^4\big]\dmub.
\end{align*}
We can split this integral in the subregions $B_\tau^{\bg}(\pm x)$ and the complementary region. Then, applying Lemma \ref{lem:error-estimates-interpolation} (see also Remark \ref{rem:errorlemma-interp}) and exploiting the antipodal symmetry, we got
\begin{align*}
    \int_{B_{\tau}^{\bg}(x)\cup B_\tau^{\bg}(-x)}\abs{\phi_\lambda}^4 \dmub &=2\int_{B_\tau^{\bg}(x)}\abs{\phi_\lambda}^4 \dmub \\
    &=2\int_{B_\tau(x)}\Big[\Uet^4+4(1-\lambda)  \Uet^3\Uetm +6(1-\lambda)^2 \Uet^2 \Uetm^2 \\
    &\qquad +4(1-\lambda)^3 \Uet \Uetm^3 +(1-\lambda)^4 \Uetm^4\Big]\,dy +O(t^2).
\end{align*}
  By e.g. \eqref{eq:buest3}, the contribution coming from the integral of $\Uet^2 \Uetm^2$ is of order $o(\epsilon^2/t^2)$, and we can easily see that the same holds for the integrals over $B_\tau(x)$ of $\Uet \Uetm^3$ and $\Uetm^4$. As a consequence, we can rewrite the previous expression as
\begin{align}\label{eq:int-den-est-1}
     \int_{B_{\tau}^{\bg}(x)\cup B_\tau^{\bg}(-x)}\abs{\phi_\lambda}^4 \dmub =2\int_{B_\tau(x)} \Uet^4\,dy +8(1-\lambda)\int_{B_\tau(x)}\Uet^3 \Uetm\,dy+o\Big(\frac{\epsilon^2}{t^2}\Big).
\end{align}
Recalling that $\norm{\Uet}_{L^4(\R^4)}=1$ and arguing as in \eqref{eq:buest1}, we easily see that the first integral on the right-hand side is equal to $1+O(\epsilon^4/\tau^4)=1+o(\epsilon^2/t^2)$ (recall \eqref{eq:alphbar-assumpt}). As for the other one, we can use \eqref{eq:est-uetuetm-int} and \eqref{eq:buest2} with $s=\tau$. Substituting these expansions in \eqref{eq:int-den-est-1} and using \eqref{eq:A-formula}, we obtain
\begin{equation}\label{eq:int-den-exp-1}
     \int_{B_{\tau}^{\bg}(x)\cup B_\tau^{\bg}(-x)}\abs{\phi_\lambda}^4 \dmub=2+\frac{4}{3}(1-\lambda)\frac{A}{\mathcal{S}_4}\Big(\frac{\epsilon^2}{t^2}\Big)+o\Big(\frac{\epsilon^2}{t^2}\Big).
\end{equation}

At last, we consider the integral of $\abs{\phi_\lambda}^4$ inside the remaining region $B^{\bg}_\delta(0)\backslash\big(B^{\bg}_\tau(x)\cup B^{\bg}_\tau(x)\big)$. Using the estimates $\nu^{-1}=O(\epsilon)$ (cf. \eqref{eq:nu-exp}) and $\abs{{G}_q(\cdot)}\leq C d_{g^q}(q,\cdot)^{-2}$ for some $C>0$, one gets
\begin{equation}\label{eq:int-den-exp-2}
    \int_{B^{\bg}_\delta(0)\backslash\big(B^{\bg}_\tau(x)\cup B^{\bg}_\tau(-x)\big)}\abs{\phi_\lambda}^4\dmub=O\Big(\frac{\epsilon^4}{\tau^4}\Big)=o\Big(\frac{\epsilon^2}{t^2}\Big),
\end{equation}
where the last equality follows from \eqref{eq:alphbar-assumpt}.

\subsubsection*{Numerator}

Consider now the numerator of \eqref{eq:int-yam-quotient}, starting with the scalar curvature term. Recalling as above that the error terms of conformal factor and metric expansion generate higher order contributions, we have (see Lemma \ref{lem:error-estimates-interpolation} and Remark \ref{rem:errorlemma-interp})
\begin{align*}
    \int_{B_\delta^{\bg}} R_{\bg}\phi_\lambda^2\dmub=2\int_{B_\tau(x)}R_{\bg}\big( \Uet + (1-\lambda)\Uetm\big)^2\,dy+\int_{D} R_{\bg} \big(\lambda \bw +(1-\lambda)\bu\big)^2\,dy +O(t^2),
\end{align*}
where $D:=B_\delta(0)\backslash\big(B_\tau(x)\cup B_\tau(-x)\big)$. Recalling \eqref{eq:sc-curv-exp}, we easily see that 
\begin{align*}
    \int_{B_\tau(x)}\big\vert R_{\bg}\big( \Uet + (1-\lambda)\Uetm\big)^2\big\vert \,dy\leq C \int_{B_\tau(0)} t^{-2}\abs{y}^{2}\frac{\epsilon^{-2}}{(1+\epsilon^{-2}\abs{y}^2)^2}\,dy=O\Big(\frac{\epsilon^2 \tau^2}{t^2}\Big).
\end{align*}
On the other hand, using the estimates $\nu^{-1}=O(\epsilon)$ (cf. \eqref{eq:nu-exp}) and $\abs{{G}_q(\cdot)}\leq C d_{g^q}(q,\cdot)^{-2}$, one obtains
\begin{align*}
    \int_D\big\vert R_{\bg} \big(\lambda \bw +(1-\lambda)\bu\big)^2\big\vert\,dy\leq C \int_{B_\delta(0)\backslash B_\tau(0)}\frac{\epsilon^2}{\abs{y}^4}\,dy=O \big(\epsilon^2 \log(\tau)\big).
\end{align*}
The previous estimates together with \eqref{eq:alphbar-assumpt} now imply
\begin{align}\label{eq:scal-curv-interp-est}
    \int_{B_\delta^{\bg}} R_{\bg}\phi_\lambda^2\dmub=o\Big(\frac{\epsilon^2}{t^2}\Big).
\end{align}

\medskip

Next, we focus on the gradient term
\begin{align}
\notag
    \int_{B_\delta^{\bg}}\abs{\nabla_{\bg} \phi_\lambda}^2_{\bg}\dmub=&\int_{B_\delta^{\bg}}\big[\lambda^2\abs{\nabla_{\bg}\bw}^2+2\lambda(1-\lambda)e^{-f/2}\nabla_{\bg}(\bw)\cdot\nabla_{\bg}(\bu)+(1-\lambda)^2e^{-f}\abs{\nabla_{\bg}\bu}^2 \\
    \notag
    &\quad +2\lambda(1-\lambda)\nabla_{\bg}(\bw)\cdot\nabla_{\bg}(e^{-f/2})\bu+(1-\lambda)^2\bu^2\abs{\nabla_{\bg}(e^{-f/2})}^2 \\
    \label{eq:int-num-first-formula}
    &\quad \quad \quad+2(1-\lambda)^2\nabla_{\bg}(e^{-f/2})\cdot\nabla_{\bg}(\bu) \bu e^{-f/2}\big]\dmub.
\end{align}
As for the denominator, we first consider the integral in the subregions $B_\tau^{\bg}(\pm x)$. By looking at the definition of $f$ in Section \ref{sec:greenfunct}, we easily see that, in normal coordinates $\{y^i\}$ centered at $x$ (or $-x$), one has $\abs{\nabla_{\bg}(e^{-f(y)/2})}_{\bg}\leq C\abs{y}$, $\forall \abs{y}\leq \tau$, for some $C>0$ independent of $x$. As a consequence of this and using the definitions of $\bu,\bw$, we deduce that
\begin{gather}
\label{eq:int-errest-1}
    \int_{B_{\tau}^{\bg}(\pm x)}\abs{\nabla_{\bg}(\bw)\cdot\nabla_{\bg}(e^{-f/2})\bu}\dmub\leq C\int_{B_{\tau}(0)}\frac{\abs{y}^2\epsilon^{-4}}{(1+\epsilon^{-2}\abs{y}^2)^3}\,dy=O(\epsilon^2\log\epsilon), \\
    \label{eq:int-errest-2}
    \int_{B_{\tau}^{\bg}(\pm x)}\abs{\bu^2\abs{\nabla_{\bg}(e^{-f/2})}^2}\dmub\leq C\int_{B_{\tau}(0)}\frac{\abs{y}^2 \epsilon^{-2}}{(1+\epsilon^{-2}\abs{y}^2)^2}\,dy =O(\epsilon^2\tau^2), \\
    \label{eq:int-errest-3}
    \int_{B_{\tau}^{\bg}(\pm x)}\abs{\nabla_{\bg}(e^{-f/2})\cdot\nabla_{\bg}(\bu) \bu e^{-f/2}}\dmub\leq C\int_{B_{\tau}(0)}\frac{\abs{y}^2\epsilon^{-4}}{(1+\epsilon^{-2}\abs{y}^2)^3}\,dy =O(\epsilon^2\log\epsilon).
\end{gather}
From \eqref{eq:int-num-first-formula}, Lemma \ref{lem:error-estimates-interpolation} and these formulae, we obtain 
\begin{align}
\notag
    \int_{B_\tau^{\bg}(x)\cup B_\tau^{\bg}(-x)}\abs{\nabla_{\bg}& \phi_\lambda}^2_{\bg}\dmub=\int_{B_\tau^{\bg}(x)\cup B_\tau^{\bg}(-x)}\big[\lambda^2\abs{\nabla_{\bg}\bw}^2+2\lambda(1-\lambda)e^{-f/2}\nabla_{\bg}(\bw)\cdot\nabla_{\bg}(\bu) \\
    \notag
    & \qquad \qquad +(1-\lambda)^2e^{-f}\abs{\nabla_{\bg}\bu}^2\big]\dmub +O( \epsilon^2\log\epsilon) \\
    \notag
    & \qquad\quad\,\,\,\,   =2\int_{B_\tau(x)}\big[\lambda^2\abs{\nabla \Uet}^2+2\lambda(1-\lambda)\nabla(\Uet)\cdot \nabla (\Uet+\Uetm) \\
    \label{eq:int-num-1}
    &\qquad \qquad +(1-\lambda)^2\abs{\nabla(\Uet+\Uetm)}^2\big]\,dy+ O(t^2).
\end{align}
At this point, we observe that, for $y\in B_\tau(x)$, one has $\abs{\nabla \Uetm}\lesssim \epsilon t^{-3}$, therefore
\begin{align}\label{eq:uetm-int-err}
    \int_{B_\tau(x)}\abs{\nabla\Uetm}^2 \,dy=O\Big(\frac{\epsilon^2\tau^4}{t^6}\Big).
\end{align}
Substituting this formula and \eqref{eq:uetuetm-est} (for $\tau=s$) inside  \eqref{eq:int-num-1}, recalling Remark \ref{rem:interp-exp-assumptions} and the fact that $\Uet$ is radially symmetric, we find
\begin{align*}
    \int_{B_\tau^{\bg}(x)\cup B_\tau^{\bg}(-x)}\abs{\nabla_{\bg} \phi_\lambda}^2_{\bg}\dmub=2\int_{B_\tau(x)}\abs{\nabla \Uet}^2\,dy+o\Big(\frac{\epsilon^2}{t^2}\Big).
\end{align*}
We can now integrate by parts and use \eqref{eq:Y-eq-Rn} to obtain
\begin{align}\label{eq:int-num-int-int}
    \int_{B_\tau^{\bg}(x)\cup B_\tau^{\bg}(-x)}\abs{\nabla_{\bg} \phi_\lambda}^2_{\bg}\dmub=2\mathcal{S}_4+2\int_{\partial B_\tau(x)}(\partial_\nu \Uet) \Uet \,d\sigma +o\Big(\frac{\epsilon^2}{t^2}\Big).
\end{align}
Notice that the integral on the right-hand side of \eqref{eq:int-num-int-int} is of order $\epsilon^2/\tau^2$.
\medskip

We next consider the integral in the remaining region $D=B^{\bg}_\delta(0)\backslash\big(B_\tau^{\bg}(x)\cup B^{\bg}_\tau(-x)\big)$. To begin, by looking at the definition of $f$ in Section \ref{sec:greenfunct}, we see that $f\equiv 0$ outside $B_{t/2}(\pm x)$ and that $\abs{\nabla_{\bg}(e^{-f(y)/2})}_{\bg}\leq C\abs{y}$, $\forall \abs{y}\leq t/2$. As a consequence,  arguing as for \eqref{eq:int-errest-1}, \eqref{eq:int-errest-2}, \eqref{eq:int-errest-3} we infer 
\begin{align}\label{eq:int-form-interm}
    \int_{D}\abs{\nabla_{\bg} \phi_\lambda}^2_{\bg}\dmub=\int_D\big[\lambda^2\abs{\nabla_{\bg}\bw}^2+2\lambda(1-\lambda)\nabla_{\bg}(\bw)\cdot\nabla_{\bg}(\bu) 
      +(1-\lambda)^2\abs{\nabla_{\bg}\bu}^2\big] \dmub +o\Big(\frac{\epsilon^2}{t^2}\Big).
\end{align} 
We now proceed to expand \eqref{eq:int-form-interm}. By arguing similarly to what done for \eqref{eq:schest-4}, we easily see that
\begin{align}\label{eq:int-num-t1}
    \int_D\abs{\nabla_{\bg}\bw}^2\dmub= -\frac{2 }{\nu^2}\int_{\partial B^{\bg}_\tau(x)}\big(\partial_{\nu}G_x\big) G_x \,d\sigma_{\bg} +o\Big(\frac{\epsilon^2}{t^2}\Big),
\end{align}
where $\nu$ denotes the \emph{outward pointing} unit normal  and $G_x:=G_q\circ \sigma_P$ is the equivariant lift of $G_q$ to $B_\delta$. Notice that here we must also add the scalar curvature term and take into consideration the boundary integral on $\partial B_\delta$; however, both terms turn out to be of higher order (the boundary integral on $\partial B_\delta$ is manifestly $o(\epsilon^2/\delta^2)$, while, for the scalar curvature term, see \eqref{eq:scal-curv-interp-est}).

Consider now the second term on the RHS of \eqref{eq:int-form-interm}. Integrating by parts, exploiting the antipodal symmetry and Lemma  \ref{lem:error-estimates-interpolation-boundary}, we get
\begin{align}
\notag
    \int_D \nabla_{\bg}(\bw)&\cdot\nabla_{\bg}(\bu)\dmub=\frac{1}{\nu}\int_D(-\Delta_{\bg} G_x) (\bu)\dmub -\frac{2}{\nu}\int_{\partial B^{\bg}_\tau(x)}(\p_\nu G_x)(\bu)\,d\sigma+o\Big(\frac{\epsilon^2}{t^2}\Big) \\
    \label{eq:int-num-t2}
    &=-\frac{2c_4\epsilon}{\tau^2 \nu}\int_{\p B_\tau^{\bg}(x)}(\partial_\nu G_x)  \,d\sigma    -\frac{c_4\epsilon}{2t^2\nu}\int_{\p B_\tau^{\bg}(x)}(\p_\nu G_x) \,d\sigma +o\Big(\frac{\epsilon^2}{t^2}\Big),
\end{align}
where in the last equality we also used the definition of $G_x$ and the estimate \eqref{eq:scal-curv-interp-est} on the term involving the scalar curvature. Notice that the cutoff term in $\bw$ (see \eqref{eq:sch-bub-test-M}) generates an higher order contribution which can be estimated exactly as in \eqref{eq:ringrt2}.

Consider now the last term on the RHS of \eqref{eq:int-form-interm}: by formula \ref{eq:errlemma-est5} in Lemma \ref{lem:error-estimates-interpolation},
\begin{align*}
    \int_D \abs{\nabla_{\bg}\bu}^2\dmub=2\int_{D'} \abs{\nabla \Uet}^2\,dy +2\int_{D'} \nabla \Uet \cdot \nabla \Uetm\,dy+O\Big(\frac{\epsilon^2 t^2}{\tau^2}\Big),
\end{align*}
where $D':=B_\delta(0)\backslash\big(B_\tau(x)\cup B_\tau(-x)\big)$.
Using e.g. the equations between \eqref{eq:uetm-int-err} and \eqref{eq:int-num-int-int}, we deduce
\begin{align*}
    \int_{D'} \abs{\nabla \Uet}^2\,dy=\int_{\R^4}\abs{\nabla \Uet}^2\,dy-\int_{B_\tau(x)}\abs{\nabla \Uet}^2\,dy+o\Big(\frac{\epsilon^2}{t^2}\Big)=-\int_{\partial B_\tau(x)}(\partial_\nu \Uet)\Uet\,d\sigma+o\Big(\frac{\epsilon^2}{t^2}\Big).
\end{align*}
Similarly, by applying \eqref{eq:gradprod-integral} and \eqref{eq:uetuetm-est}, we got
\begin{align*}%\label{eq:prod-form-interp}
    \int_{D'}\nabla  \Uet \cdot \nabla \Uetm=\int_{B_\delta}\nabla \Uet \cdot \nabla\Uetm\,dy-2\int_{B_\tau(x)}\nabla \Uet \cdot \nabla\Uetm\,dy =\pi^2 c_4^2\frac{\epsilon^2}{t^2}+o\Big(\frac{\epsilon^2}{t^2}\Big),
\end{align*}
which, together with the previous equation, implies
\begin{align}\label{eq:int-num-t3}
    \int_D \abs{\nabla_{\bg}\bu}^2\dmub=2\pi^2 c_4^2\frac{\epsilon^2}{t^2}-2\int_{\partial B_\tau(x)}(\partial_\nu \Uet)\Uet\,d\sigma+o\Big(\frac{\epsilon^2}{t^2}\Big).
\end{align}
Substituting \eqref{eq:int-num-t1}, \eqref{eq:int-num-t2} and \eqref{eq:int-num-t3} inside \eqref{eq:int-form-interm} and recalling \eqref{eq:A-formula}, we have
\begin{align}
\notag
     \int_{D}\abs{\nabla_{\bg} \phi_\lambda}^2_{\bg}\dmub=&\frac{-2\lambda^2}{\nu^2}\int_{\p B_\tau^{\bg}(x)}(\p_\nu G_x) G_x\,d\sigma_{\bg}-2(1-\lambda)^2\int_{\p B_\tau(x)}(\p_\nu \Uet) \Uet\,d\sigma +(1-\lambda)^2\frac{A}{3}\frac{\epsilon^2}{t^2} \\
     \label{eq:int-num-int-ext}
     &-4\lambda(1-\lambda)\frac{c_4 \epsilon}{\tau^2\nu}\int_{\p B_\tau^{\bg} (x)}(\p_\nu G_x) \,d\sigma-\lambda(1-\lambda)\frac{c_4\epsilon}{t^2\nu}\int_{\p B_\tau^{\bg} (x)}(\p_\nu G_x) \,d\sigma+o\Big(\frac{\epsilon^2}{t^2}\Big).
\end{align}
Combining \eqref{eq:int-num-int-int} and \eqref{eq:int-num-int-ext}, we  obtain the following expression for the numerator of \eqref{eq:int-yam-quotient}:
\begin{align}
\notag
     \int_{B_\delta^{\bg}}\abs{\nabla_{\bg} \phi_\lambda}^2_{\bg}&\dmub=2\mathcal{S}_4+2\lambda^2\Big(\int_{\p B_\tau(x)}(\p_\nu \Uet)\Uet\,d\sigma-\frac{1}{\nu^2}\int_{\p B_\tau^{\bg}(x)}(\p_\nu G_x) G_x\,d\sigma_{\bg}\Big) \\
     \notag
     &+4\lambda(1-\lambda)\Big(\int_{\p B_\tau(x)} ( \p_{\nu} \Uet) \Uet\,d\sigma - \frac{c_4\epsilon}{\tau^2 \nu}\int_{\p B_{\tau}^{\bg}(x)} (\p_\nu G_x) \,d\sigma_{\bg}\Big)\\
     \label{eq:num-interm-form}
     &-\lambda(1-\lambda)\Big(\frac{c_4\epsilon}{t^2 \nu}\int_{\p B_\tau^{\bg}(x)}(\p_\nu G_x)\,d\sigma_{\bg}\Big) +(1-\lambda)^2\frac{A}{3}\frac{\epsilon^2}{t^2}+o\Big(\frac{\epsilon^2}{t^2}\Big).
\end{align}
By looking at \eqref{eq:green-manif-expansion}, \eqref{eq:estest1} and \eqref{eq:estest3}, one easily sees that
\begin{align}
\notag
    \int_{\p B_\tau(x)} &( \p_{\nu} \Uet) \Uet\,d\sigma - \frac{c_4\epsilon}{\tau^2 \nu}\int_{\p B_{\tau}^{\bg}(x)} (\p_\nu G_x)\,d\sigma_{\bg}  \\
    \label{eq:diffest-uet-2}
    &=2\pi^2 \tau^3\Big[-2c_4^2\frac{\epsilon^2}{\tau^5}-\Big(c_4\frac{\epsilon}{\tau^2}\Big)c_4\epsilon\Big(1-\frac{\tau^2}{4t^2}\Big)\Big(-\frac{2}{\tau^3}\Big)\Big]+o\Big(\frac{\epsilon^2}{t^2}\Big)=-\pi^2 c_4^2\frac{\epsilon^2}{t^2}+o\Big(\frac{\epsilon^2}{t^2}\Big).
\end{align}
Similarly, we also have
\begin{align}\label{eq:diffest-uetm}
    \frac{c_4\epsilon}{t^2 \nu}\int_{\p B_\tau^{\bg}(x)}(\p_\nu G_x)\,d\sigma_{\bg}=\frac{2\pi^2\tau^3c_4\epsilon}{t^2}\Big[c_4\epsilon\Big(1-\frac{\tau^2}{4t^2}\Big)\Big(-\frac{2}{\tau^3}\Big)\Big]+o\Big(\frac{\epsilon^2}{t^2}\Big)=-4\pi^2 c_4^2\frac{\epsilon^2}{t^2}+o\Big(\frac{\epsilon^2}{t^2}\Big).
\end{align}
Finally, substituting \eqref{eq:diffest-uet-1}, \eqref{eq:diffest-uet-2} and \eqref{eq:diffest-uetm} inside \eqref{eq:num-interm-form} and recalling \eqref{eq:A-formula}, we obtain
\begin{align}\label{eq:int-num-finalexp}
    \int_{B_\delta^{\bg}}\abs{\nabla_{\bg} \phi_\lambda}^2_{\bg}\dmub=2\mathcal{S}_4+(1-2\lambda)\frac{A}{3} \frac{\epsilon^2}{t^2}+o\Big(\frac{\epsilon^2}{t^2}\Big).
\end{align}

\subsubsection*{Conclusion}

Substituting \eqref{eq:int-den-exp-1}, \eqref{eq:int-den-exp-2} and \eqref{eq:int-num-finalexp} inside \eqref{eq:int-yam-quotient} and recalling that $a=6$ and $t=\epsilon^\alpha$, we find
\begin{align*}
     Q_{g^q}(\psi_\lambda)&=\frac{1}{\sqrt{2}}\frac{12\mathcal{S}_4+2(1-2\lambda)A \epsilon^{2(1-\alpha)}+o\big(\epsilon^{2(1-\alpha)}\big)}{\Big(2+\frac{4}{3}(1-\lambda)A\mathcal{S}_4^{-1}\epsilon^{2(1-\alpha)}+o\big(\epsilon^{2(1-\alpha)}\big)\Big)^{1/2}} \\
     &=\Big(6\mathcal{S}_4+(1-2\lambda)A \epsilon^{2(1-\alpha)}+o\big(\epsilon^{2(1-\alpha)}\big)\Big)\Big(1-\frac{1}{3}(1-\lambda)A\mathcal{S}_4^{-1}\epsilon^{2(1-\alpha)}+o\big(\epsilon^{2(1-\alpha)}\big)\Big) \\
     &=6\mathcal{S}_4-A\epsilon^{2(1-\alpha)}+o\big(\epsilon^{2(1-\alpha)}\big), \qquad \text{as $\epsilon\to0$}.
\end{align*}
This proves Proposition \ref{prop:interp-expansion}.

\subsubsection*{Error estimates on bubble profiles}

The next two lemmas collect some estimates on the behavior of integrals involving powers of $\bu, \bw$, or their gradients, which were employed in the computations above. The takeaway is that, in normal coordinates for $\bg$ centered at $\pm x$, we are always able to substitute $\bu$ with an \acc exact'' double bubble $\Uethat$ when computing the integrals.

\begin{lemma}\label{lem:error-estimates-interpolation}
    Let $\bu, \bw$ be defined as in \eqref{eq:bubw-def}, let $x=t\mathbf{e}_1$ and let $t,\tau$ be as in Proposition \ref{prop:interp-expansion}.
    Then the following estimates hold:
    \begin{gather}
    \label{eq:errlemma-est1}
        \int_{B_\tau^{\bg}(x)}\bu^4\dmub=\int_{B_\tau(x)} \Uethat^4(y)\,dy+O(t^2), \\
        \label{eq:errlemma-est2}
        \int_{B_\tau^{\bg}(x)}\bu^3 \bw\dmub=\int_{B_\tau(x)} \Uethat^3(y) \Uet(y)\,dy+O(t^2), \\
        \label{eq:errlemma-est3}
        \int_{B_\tau^{\bg}(x)} \abs{\nabla_{\bg}\bu}^2\dmub=\int_{B_\tau(x)}\abs{\nabla \Uethat(y)}^2\,dy+O(t^2), \\
        \label{eq:errlemma-est4}
        \int_{B_\tau^{\bg}(x)} \nabla_{\bg} \bu\cdot  \nabla_{\bg} \bw \dmub=\int_{B_\tau(x)} \nabla \Uethat \cdot \nabla \Uet \,dy+O(t^2), \\
        \label{eq:errlemma-est5}
        \int_{D}\abs{\nabla_{\bg}\bu}^2\dmub =\int_{D'}\abs{\nabla \Uethat}^2\,dy +O\Big(\frac{\epsilon^2 t^2}{\tau^2}\Big),
    \end{gather}
    where $\Uethat=\widehat{U}_{\epsilon,t\mathbf{e}_1}$ is the double-bubble defined in \eqref{eq:double-bubble}, $D=B_\delta(0)\backslash \big(B_\tau^{\bg}(x)\cup B_\tau^{\bg}(-x)\big)$ and $D'=B_\delta(0)\backslash \big(B_\tau(x)\cup B_\tau(-x)\big)$.
\end{lemma}

\begin{proof}
    We start by observing that the expansions \eqref{eq:metr-eucl-exp-lem-1} for $\tilde{g}$ at $0$ (which holds in $B_\delta$) and \eqref{eq:gbar-metr-expansion} for $\bg$ at $x$ (which holds in $B_{t/2}$) imply 
    \begin{gather*}
        d_{\tg}(x,p)=\abs{x-p}(1+O^{''}(t^2)) \qquad \forall p\in B_{t/2}(x), \\
        d_{\bg}(x,p)=\abs{x-p}(1+O^{''}(\tau^2)) \qquad \forall p\in B_{2\tau}(x), 
    \end{gather*}
    where $\abs{\, \cdot \,}$ denotes the Euclidean distance and $O^{''}(t^2)$ denotes an error term, say $\Psi$, such that $\abs{\Psi}\leq Ct^2$ and $\abs{\nabla^k \Psi}\leq Ct^{2-k}$ for $k=1,2$, where $C$ does not depend upon $t$ or $\tau$. Moreover, by looking at the definition of $\bg$, \eqref{eq:gbar-definition}, we also see that
    \begin{align*}
        d_{\tg}(x,p)=d_{\bg}(x,p)+O^{''}\big(d_{\bg}(x,p)^3\big), \qquad d_{\bg}(x,p)=d_{\tg}(x,p)+O^{''}\big(d_{\tg}(x,p)^3\big), \qquad \forall p \in B_{2\tau}(x).
    \end{align*}

To begin, we notice that \eqref{eq:errlemma-est1} immediately follows once we put ourselves in Euclidean coordinates at $x$ and recall the expansion \eqref{eq:vol-det-exp} for the volume element (the error term is even smaller than the one in \eqref{eq:errlemma-est1}).
Let us now focus on \eqref{eq:errlemma-est2}; the above estimates imply that, in $\bg$-normal coordinates $\{z^i\}$ centered at $x$, one has
\begin{align}\label{eq:perturbed-bubble}
    \bu(z)=\frac{c_4 \epsilon^{-1}}{1+\epsilon^{-2}\abs{z}^2(1+O^{''}(t^2))}+\frac{c_4\epsilon^{-1}}{1+\epsilon^{-2}\abs{z+2x}^2(1+O^{''}(t^2))}.
\end{align}
Notice that, in general, if $z$ corresponds to a point $p\in B_\tau(x)$, then $\abs{z+2x}\not=d_{\bg}(p,-x)$ as the distance function is not smooth at the origin; nevertheless, we still have the above expression for $\bu$ in $\bg$-normal coordinates at $x$. Using \eqref{eq:perturbed-bubble}, a change of variables and the definition of $\bw$, we have
\begin{align*}
    \int_{B_\tau^{\bg}(x)}&\bu^3 \bw\dmub \\ 
     =&\int_{B_\tau(x)}\bigg[\frac{c_4 \epsilon^{-1}}{1+\epsilon^{-2}\abs{y-x}^2(1+O(t^2))}+\frac{c_4\epsilon^{-1}}{1+\epsilon^{-2}\abs{y+x}^2(1+O(t^2))}\bigg]^3\frac{c_4\epsilon^{-1}\,dy}{1+\epsilon^{-2}\abs{y-x}^2}+O(t^2)
\end{align*}
(the slight change of domain is irrelevant for the next estimates). It is  sufficient to estimate the difference between each term in the above integral and their \acc exact'' counterpart. For instance, 
\begin{align*}
    \int_{B_\tau(x)}\bigg[&\frac{ \epsilon^{-4}}{\big(1+\epsilon^{-2}\abs{y-x}^2(1+O(t^2))\big)^3\big(1+\epsilon^{-2}\abs{y-x}^2\big)}-\frac{\epsilon^{-4}}{\big(1+\epsilon^{-2}\abs{y-x}^2\big)^4}\bigg]\,dy  \\
    &=\int_{B_\tau(0)}\frac{\epsilon^{-4}\big(1+\epsilon^{-2}\abs{y}^2\big)^3-\epsilon^{-4}\big(1+\epsilon^{-2}\abs{y}^2(1+O(t^2))\big)^3}{\big(1+\epsilon^{-2}\abs{y}^2(1+O(t^2))\big)^3\big(1+\epsilon^{-2}\abs{y}^2\big)^4}\,dy \\
    &=O(t^2\epsilon^{-4})\int_{B_\tau(0)}\frac{\big(\epsilon^{-6}\abs{y}^6+3\epsilon^{-2}\abs{y}^2+3\epsilon^{-4}\abs{y}^4\big)}{\big(1+\epsilon^{-2}\abs{y}^2\big)^7}\,dy,
\end{align*}
    and  we can now split the integral in the regions $B_\epsilon(x)$ and $B_\tau(x)\backslash B_\epsilon(x)$ and use basic estimates to show that the quantity is an $O(t^2)$. Similarly, we got
    \begin{align*}
        \int_{B_\tau(0)}\bigg[&\frac{\epsilon^{-4}}{\big(1+\epsilon^{-2}\abs{y}^2(1+O(t^2))\big)^2\big(1+\epsilon^{-2}\abs{y+2x}^2(1+O(t^2))\big)(1+\epsilon^{-2}\abs{y}^2)} \\
        & \qquad \quad\qquad\qquad\qquad\qquad\qquad\qquad-\frac{\epsilon^{-4}}{(1+\epsilon^{-2}\abs{y}^2)^3(1+\epsilon^{-2}\abs{y+2x}^2)}\bigg]\,dy \\[0.5ex]
=&O(t^2\epsilon^{-4})\int_{B_\tau(0)}\frac{\epsilon^{-2}\abs{y+2x}^2+2\epsilon^{-2}\abs{y}^2+2\epsilon^{-4}\abs{y}^2\abs{y+2x}^2+\epsilon^{-4}\abs{y}^4+\epsilon^{-6}\abs{y+2x}^2\abs{y}^4}{(1+\epsilon^{-2}\abs{y}^2)^5(1+\epsilon^{-2}\abs{y+2x}^2)^2}\,dy \\
=&O(t^{-2})\int_{B_\tau(0)}\frac{\epsilon^{-2}t^2+\epsilon^{-2}\abs{y}^2+\epsilon^{-4}\abs{y}^2 t^2+\epsilon^{-4}\abs{y}^4+\epsilon^{-6}\abs{y}^4t^2}{(1+\epsilon^{-2}\abs{y}^2)^5}\,dy.
    \end{align*}
    As above, we can split the integral in the regions $B_\epsilon(x)$ and $B_\tau\backslash B_\epsilon(x)$ and use basic estimates to prove that this quantity is an $O(\epsilon^2)$ (i.e. $O(\epsilon^2/t^2) O(t^2)$, cf. \eqref{eq:F3-estimate-2}).
    Arguing in the exact same way, we can show that all the remaining differences are of even higher order, in particular thay are $O(t^2)$. This shows \eqref{eq:errlemma-est2}.

    The estimates \eqref{eq:errlemma-est3} and \eqref{eq:errlemma-est4} can be obtained in the same way by starting from \eqref{eq:perturbed-bubble} and estimating the differences between \acc perturbed'' functions and regular one.

    Finally, to prove estimate \eqref{eq:errlemma-est5}, we can put ourselves in geodesic coordinates for $\bar{g}$ centered at $0$ and observe that, by \eqref{eq:estrem-1}, the difference between the integrals in \eqref{eq:errlemma-est5} can be controlled by 
    \begin{equation*}
        \int_{D'}\abs{\nabla \Uethat}^2\abs{y}^2\,dy.
    \end{equation*}
    By \eqref{eq:estrem2}, this quantity is at least of order  $O(t^2)$, which is already enough for our purposes. However, we can be more precise and split $D'$ into the regions $B_{t/2}(\pm x)\backslash B_\tau(\pm x)$, $B_{2t}(0)\backslash B_{t/2}(\pm x)$ and $B_\delta \backslash B_{2t}(0)$ and estimate the integrand in each component in order to readily deduce \eqref{eq:errlemma-est5}. 
\end{proof}

\begin{remark}\label{rem:errorlemma-interp}
    With the same argument as above, we can also  deduce similar estimates for other combinations of $\bu$ and $\bw$ in $B^{\bg}_\tau(x)$ or $D$, like $\bu^2 \bw^2$, $\bw^3 \bu$ or $\bu \bw$, which are also employed in the estimates above. In all cases, the order of the error term is equal to $t^2$ times the order of the integral. 
\end{remark}

The estimate in the next Lemma is less refined, but still sufficient for our purposes:

\begin{lemma}\label{lem:error-estimates-interpolation-boundary}
    Let $\bu$ be defined as in \eqref{eq:bubw-def} and let $x=t\mathbf{e}_1$, $G_x:=G_q\circ \sigma_P$ and $\tau,t$ as in Proposition \ref{prop:interp-expansion}. Then one has
    \begin{align*}%\label{eq:boundary-error-estimates-interp}
        \frac{1}{\nu}\int_{\p B_\tau^{\bg}(x)}(\p_\nu G_x)\Big(\bu -\frac{c_4\epsilon}{\tau^2}-\frac{c_4\epsilon}{4t^2}\Big)\,d\sigma_{\bg}=o\Big(\frac{\epsilon^2}{t^2}\Big).
    \end{align*}
\end{lemma}
\begin{proof}
    Recalling \eqref{eq:perturbed-bubble}, we compute
    \begin{equation*}
        \frac{c_4\epsilon^{-1}}{1+\epsilon^{-2}\tau^2(1+O(t^2))}-\frac{c_4\epsilon}{\tau^2}=-\frac{c_4\epsilon+c_4\epsilon^{-1}\tau^2O(t^2)}{\tau^2\big(1+\epsilon^{-2}\tau^2(1+O(t^2))\big)}\simeq \frac{\epsilon^3}{\tau^4}+\frac{\epsilon t^2}{\tau^2}=O\Big(\frac{\epsilon t^2}{\tau^2}\Big),
    \end{equation*}
    and, for $\abs{y}=\tau$, 
    \begin{equation*}
        \frac{c_4\epsilon^{-1}}{1+\epsilon^{-2}\abs{y+2x}^2(1+O(t^2))}-\frac{c_4\epsilon}{4t^2}\simeq \frac{\epsilon^3}{t^4}+\frac{\epsilon \tau}{t^3}=O\Big(\frac{\epsilon\tau}{t^3}\Big).
    \end{equation*}
    Therefore, recalling \eqref{eq:estest3} and that $\abs{\p_\nu G_x}\leq C/\tau^3$ in $\p B_\tau^{\bg}(x)$, we get
    \begin{align*}
         \frac{1}{\nu}\int_{\p B_\tau^{\bg}(x)}(\p_\nu G_x)\Big(\bu -\frac{c_4\epsilon}{\tau^2}-\frac{c_4\epsilon}{4t^2}\Big)\,d\sigma_{\bg}\simeq \epsilon\Big[ O\Big(\frac{\epsilon t^2}{\tau^2}\Big)+O\Big(\frac{\epsilon\tau}{t^3}\Big)\Big]=o\Big(\frac{\epsilon^2}{t^2}\Big),
    \end{align*}
    which completes our proof.
\end{proof}